%% file: main.tex
\documentclass[11pt]{article}
\usepackage[a4paper, top=1in, left=1in, right=1in, bottom=1in]{geometry}
\usepackage[round]{natbib}
\usepackage{todonotes}
\usepackage{setspace}
\usepackage[utf8]{inputenc}
\usepackage{amsthm, amsmath, amsfonts, amssymb}
\usepackage{mathtools}
\usepackage{bm}
\usepackage{xcolor}
\usepackage{tikz}
\usetikzlibrary{decorations.pathreplacing}
\usetikzlibrary{shapes.misc}
\usetikzlibrary{calc}
\usepackage{hf-tikz}
\usepackage{caption}
\usepackage{subcaption}
\usepackage{enumitem}
\usepackage{url}

\newif\ifanonymous
\anonymousfalse

\usepackage[ruled, linesnumbered]{algorithm2e}

\theoremstyle{definition}

\newtheorem{assumption}{Assumption}
\newtheorem{example}{Example}
\newtheorem*{remark}{Remark}

\theoremstyle{plain}
\newtheorem{theorem}{Theorem}
\newtheorem{lemma}{Lemma}
\newtheorem{proposition}{Proposition}
\newtheorem{corollary}{Corollary}
\newtheorem{observation}{Observation}

\title{A Note on Piecewise Affine Decision Rules \\ for Robust, Stochastic, and Data-Driven Optimization}
\ifanonymous
\author{}
\else
\author{%
Simon Thomä%
\footnote{Chair of Operations Management, RWTH Aachen, 52072 Aachen, Germany}, 
Maximilian Schiffer%
\footnote{School of Management \& Munich Data Science Institute, Technical University of Munich, 80333 Munich, Germany}, 
and Wolfram Wiesemann%
\footnote{Imperial College Business School, Imperial College London, London SW7 2AZ, UK}%
\\
simon.thomae@om.rwth-aachen.de
}
\fi

\date{\today}

\input{makros}

\begin{document}

\maketitle

\begin{abstract}
    Multi-stage decision-making under uncertainty, where decisions are taken under sequentially revealing uncertain problem parameters, is often essential to faithfully model managerial problems. Given the significant computational challenges involved, these problems are typically solved approximately. This short note introduces an algorithmic framework that revisits a popular approximation scheme for multi-stage stochastic programs by \cite{Georghiou2015} and improves upon it to deliver superior policies in the stochastic setting, as well as extend its applicability to robust optimization and a contemporary Wasserstein-based data-driven setting. We demonstrate how the policies of our framework can be computed efficiently, and we present numerical experiments that highlight the benefits of our method.
\end{abstract}

\section{Introduction}

Over the past few decades, three paradigms have gained significant prominence in decision-making under uncertainty: \emph{(i)} \emph{robust optimization}, which assumes that the uncertain problem parameters lie in an uncertainty set, and the decision maker optimizes for the worst-case parameter realizations within this set; \emph{(ii)} \emph{stochastic programming}, which assumes that the uncertain parameters follow a known probability distribution, and the decision maker optimizes a risk measure under this distribution; and \emph{(iii)} \emph{data-driven optimization}, which assumes that historical samples of the uncertain parameters are available, and the decision maker optimizes a risk measure under the worst distribution that is deemed plausible in view of the observed samples. 
Of particular notice are Wasserstein-based data-driven optimization problems, where the set of plausible distributions is defined in terms of their Wasserstein distances to an empirical distribution.

In each of these paradigms, significant progress has been reported for single-stage problems, where all decisions are taken \emph{here-and-now} before any of the uncertain parameters are observed. However, two-stage and multi-stage problems, in which some of the decisions can be taken \emph{wait-and-see} after some or all of the parameter values have been revealed, continue to pose formidable theoretical and computational challenges. Indeed, even linear two-stage problems where the uncertain parameters only occur in the constraints' right-hand sides are known to be strongly NP-hard \citep{BenTal2004}. This negative result has directed much existing and ongoing research towards developing tractable approximations. Constructing \emph{decision rules}, which restrict the dependence of the wait-and-see decisions to low-dimensional function classes of the observed uncertainties, is particularly prominent in this context. Popular classes of decision rules include affine \citep{BenTal2004}, segregate affine \citep{Chen2008, Chen2009, GS10:dro_tractable_approximations}, piecewise constant \citep{Bertsimas2011}, piecewise affine \citep{Vayanos2011, Georghiou2015, BenTal2020, Thomae2024} and polynomial \citep{Bampou2011, Bertsimas2011b} rules, as well as their combinations \citep{Rahal2021}.

This short note builds upon the previous works of \cite{Georghiou2015} and \cite{Bertsimas2023} to develop and analyze a refined class of piecewise affine decision rules for robust, stochastic and data-driven optimization. Our contributions can be summarized as follows.
\begin{enumerate}
    \item We prove that the piecewise affine policies of \cite{Georghiou2015} cannot improve upon simple affine policies in robust problems, and that any improvements in stochastic problems are due to refinements in the objective function, rather than increased flexibility in the wait-and-see decisions. This provides a theoretical explanation of an empirical observation frequently observed in the literature \citep[see, e.g.,][]{Bertsimas2015b, Han2025}. 
    \item We develop a refined framework of piecewise affine policies that addresses this shortcoming and that is applicable to robust, stochastic, and contemporary Wasserstein-based data-driven problems. We show how these policies can be computed efficiently.
    \item We prove strong theoretical approximation bounds under standard assumptions, matching or improving the best-known results in the literature.
    \item We demonstrate through numerical experiments that our framework can yield significant improvements over the piecewise affine policies of \cite{Georghiou2015} as well as the data-driven affine policies of \cite{Bertsimas2023}.
\end{enumerate}

While this short note focuses on refining piecewise affine decision rules within robust, stochastic and data-driven optimization, we acknowledge several alternative paradigms that are widely studied in the broader literature on decision-making under uncertainty. 
Classical approaches rooted in Markov decision processes (MDPs) provide foundational frameworks for sequential decision-making that are extensively documented in the OR/MS literature \citep{Puterman1994}. 
Recent methodological advances include approximate dynamic programming \citep{Powell2011} and reinforcement learning techniques, particularly deep reinforcement learning \citep{FrancoisLavet2018}, which leverage neural network approximations within the MDP framework.
Additionally, alternative data-driven paradigms such as ``predict-then-optimize'' \citep{Bertsimas2020}, imitation learning \citep{Abbeel2004}, and integrated prediction-optimization pipelines in combinatorial optimization \citep{Baty2024} offer promising avenues for linking predictive and prescriptive models. 
However, these methods typically prioritize empirical performance and flexibility at the expense of interpretability, analytical tractability, and rigorous robustness or stochastic guarantees. 
To maintain clarity and brevity, we therefore do not further explore these alternative approaches here.

The remainder of this short note is structured as follows. We introduce our problem setting in Section~\ref{sec:lift}. 
Section~\ref{sec:tightening} develops our framework of decision rules, and Section~\ref{sec:solving_subproblem} shows how the associated policies can be computed efficiently. 
Section~\ref{sec:bounds} establishes theoretical approximation bounds under standard assumptions. 
We report on numerical experiments in Section~\ref{sec:numerical}. All proofs are relegated to the e-companion. Additionally, we provide the sourcecodes and datasets to reproduce the numerical results in the GitHub repository accompanying this paper.%
\ifanonymous
\footnote{\url{https://anonymous.4open.science/r/lifting-based-piecewise-affine-policies/}}
\else
\footnote{\url{https://github.com/tumBAIS/lifting-based-piecewise-affine-policies}}
\fi

\noindent \textbf{Notation.} $\,$
For $n\in\mathbb{N}$, we define $[n] = \{1,\dots,n\}$ and $[n]_0=[n]\cup\{0\}$. Boldface lower-case and upper-case characters $\xvec\in\mathbb{R}^n$ and $\amat\in\mathbb{R}^{m\times n}$ denote vectors and matrices, respectively, and $x_i$ refers to the $i$\textsuperscript{th} element of $\xvec$. 
We denote random variables by a tilde (e.g., $\tilde{x}$) and deterministic quantities, including realizations of random variables, without a tilde. 
We use $(x)_+$ to abbreviate $\max\{x, 0\}$. The vectors $\unitvec_i$, $\nullvec$, and $\unitvec$ represent the unit standard basis vector, the vector of all zeros and the vector of all ones, respectively. Depending on the context, $\unitytary$ acts as the identity matrix or the identity function. For $a, b \in \mathbb{R}$, $a \leq b$, $[a,b]$ is the closed interval from $a$ to $b$, and for $\avec, \bvec \in \mathbb{R}^n$, $\avec\leq \bvec$, $[\avec,\bvec]$ is the hyperrectangle $\{\xvec\in\mathbb{R}^n\,\colon\, \avec\leq\xvec\leq\bvec\}$. We use $\amat^+=(\amat^\intercal \amat)^{-1} \amat^\intercal$ and $F^+$ as the left Moore-Penrose inverse of a matrix $\amat$ and the left-inverse of an injective function $F$, respectively. For a family of sets $\U = \{\U_g\}_{g\in[\nums]}$, we use $\U$ to denote both the family and the union $\U=\bigcup_{g\in[\nums]}\U_g$ of its sets; which interpretation is used will be clear from the context.

\section{Piecewise Affine Policies via Liftings}
\label{sec:lift}

We study multi-stage stochastic programs of the form
\begin{equation}
\tag{$\baseproblem (\PU, \U)$}
\label{eq:P}
\begin{aligned}
    \minimize_{\xvec\in\X} \quad& \mathbb{E}_\PU \left[ \cvec(\suvec)^\intercal\xvec(\suvec)\right]\\
    \st \quad& \amat_g \xvec(\uvec) \leq \bvec_g(\uvec) && \forall g\in[\nums], \uvec\in \U_g,
\end{aligned}  
\end{equation}
where $\suvec$ is a random vector of problem parameters that are supported on the family of compact sets $\U = \{ \U_g \}_{g \in [\nums]}$ and are governed by the probability distribution $\mathbb{P}$. 
The objective coefficients $\cvec (\cdot)$ and constraint right-hand sides $\bvec_g (\cdot)$ are affine in $\uvec$, and separate constraint sets need to be satisfied for each subset $\U_g$ of the support.
The problem parameters exhibit the temporal structure $\uvec = (\uvec_1, \ldots, \uvec_T)$, where $\uvec_t$ is observed at the beginning of time stage $t \in [T]$ and where $\uvec^t = (\uvec_1, \ldots, \uvec_t)$ comprises the parameters observed up until (and including) stage $t$.  
Likewise, the decisions satisfy $\xvec (\cdot) = (\xvec_0, \xvec_1 (\cdot), \ldots, \xvec_T (\cdot))$, where $\xvec_0$ comprises the here-and-now decisions and where the recourse decisions $\xvec_t$ must only depend on the parameter values $\uvec^t$ observed previously. 
This non-anticipativity requirement is captured by the constraint $\xvec\in\X$. 
To avoid confusion, we emphasize that $\uvec_t$ and $\xvec_t$ are subvectors of $\uvec$ and $\xvec$, whereas the scalars $\uscal_i$ and $x_i$ refer to the $i$\textsuperscript{th} elements of $\uvec$ and $\xvec$, respectively.

Problem~\ref{eq:P} unifies several commonly studied problem formulations from the literature.

\begin{example}[Stochastic Programming]
\label{ex:SP}
    Multi-stage stochastic programs of the form
    \begin{equation*}
    \begin{aligned}
        \minimize_{\xvec\in\X} \quad& \mathbb{E}_\mathbb{P} \left[ \cvec(\suvec)^\intercal\xvec(\suvec)\right]\\
        \st \quad& \amat \xvec(\suvec) \leq \bvec(\suvec) && \PU \pas
    \end{aligned}  
    \end{equation*}
    constitute a special case of problem~\ref{eq:P} where $G = 1$, $\U_1$ is the support of $\PU$, that is, the smallest Borel-measurable closed set satisfying $\PU(\U_1) = 1$, and $(\amat_1, \bvec_1) = (\amat, \bvec)$. \hfill $\Diamond$
\end{example}

\begin{example}[Robust Optimization]
\label{ex:RP}
    Multi-stage robust optimization problems of the form
    \begin{equation*}
    \begin{aligned}
        \minimize_{\xvec\in\X} \quad& \max_{\uvec\in \Psi} \; \cvec^\intercal\xvec(\uvec) \\
        \st \quad& \amat \xvec(\uvec) \leq \bvec(\uvec) && \forall \uvec \in\Psi
    \end{aligned}  
    \end{equation*}
    emerge as an instance of problem~\ref{eq:P} where $G = 1$, $\U_1 = \Psi$, and $\mathbb{P}$ is any distribution supported on $\Psi$, if we (i) augment the here-and-now decisions $\xvec_0$ by an epigraphical variable $\tau \in \mathbb{R}$ that is minimized in the objective, and (ii) identify $(\amat_1, \bvec_1)$ with $(\amat, \bvec)$ augmented by the epigraphical constraint $\cvec^\intercal\xvec(\uvec) - \tau \leq 0$. \hfill $\Diamond$
\end{example}

\begin{example}[Data-Driven Optimization]
\label{ex:DP}
    The multi-stage data-driven optimization problem
    \begin{equation*}
    \begin{aligned}
        \minimize_{\xvec\in\X} \quad& \sup_{\PU \in \mathcal{B}_\epsilon(\hat{\PU})} \;  \mathbb{E}_{\PU}\left[\cvec^\intercal\xvec(\suvec)\right] \\
        \st \quad& \amat \xvec(\suvec) \leq \bvec(\suvec) && \PU \pas \; \forall \PU \in \mathcal{B}_\epsilon(\hat{\PU})
    \end{aligned}  
    \end{equation*}
    considers the type-$\infty$ Wasserstein ball $\mathcal{B}_\epsilon(\hat{\PU}) = \{ \mathbb{P} \in \mathfrak{P}(\mathbb{R}^\numu) \, \colon \, \mathrm{d}^\mathrm{W}_\infty (\mathbb{P}, \hat{\PU}) \leq \epsilon \}$ with
    Borel probability measures $\mathfrak{P}(\mathbb{R}^\numu)$ on $\mathbb{R}^\numu$, Wasserstein distance
    \begin{equation*}
        \mathrm{d}^\mathrm{W}_\infty (\mathbb{P}, \mathbb{Q}) = \left\{ \Pi\text{-ess sup} \norm{\suvec - \tilde{\bm{\psi}}} \, \colon \, \left[ \begin{array}{c} \text{$\Pi$ is a joint distribution over $\suvec$ and $\tilde{\bm{\psi}}$} \\ \text{with marginals $\mathbb{P}$ and $\mathbb{Q}$} \end{array} \right] \right\},
    \end{equation*}
    and ground metric $\norm{\cdot}$, where the empirical distribution $\hat{\PU} = \frac{1}{\nums} \sum_{g \in [\nums]} \hat{\uvec}_g$ is supported on the historical samples $\hat{\uvec}_1, \dots, \hat{\uvec}_\nums$ \citep[see, e.g.,][]{Esfahani2018, Blanchet2019, Gao2023}. 
    \citet{Bertsimas2023} show that this problem admits the equivalent reformulation
    \begin{equation*}
    \begin{aligned}
        \minimize_{\xvec\in\X} \quad& \frac{1}{\nums}\sum_{g\in[\nums]} \max_{\uvec\in\U_g} \; \cvec^\intercal \xvec(\uvec) \\
        \st \quad& \amat \xvec(\uvec) \leq \bvec(\uvec) && \forall g\in[\nums], \uvec\in\U_g,
    \end{aligned}
    \end{equation*}
    where $\U_g = \{\uvec \, \colon \, \lVert \uvec - \hat{\uvec}_g \rVert \leq \epsilon\}$ are perturbation sets around the historical samples. The problem therefore emerges as a special case of~\ref{eq:P}, where $\mathbb{P}$ is any distribution supported on $\U$, if we (i) augment the here-and-now decisions $\xvec_0$ by the epigraphical variables $\tau_g \in \mathbb{R}$ whose average is minimized in the objective, and (ii) identify each $(\amat_g, \bvec_g)$, $g \in [G]$, with $(\amat, \bvec)$ augmented by the epigraphical constraint $\cvec^\intercal\xvec(\uvec) - \tau_g \leq 0$. \hfill $\Diamond$
\end{example}

Problem~\ref{eq:P} is NP-hard already when its objective function is deterministic, $T=2$, $\nums = 1$ and $\U_1$ is polyhedral \citep{G02:immunized_solutions}. 
A widely applied conservative approximation replaces the wait-and-see decisions of $\xvec (\uvec)$ with affine decision rules, that is, it imposes an affine dependence of $\xvec_t (\uvec^t)$ on $\uvec^t$ for all $t \in [T]$ \citep{BenTal2004, Kuhn2011}. 
Over the last two decades, this approximation has been refined to different classes of segregated affine \citep{Chen2008, Chen2009, GS10:dro_tractable_approximations} and piecewise affine \citep{Vayanos2011, Georghiou2015} decision rules. 
Here, the key idea is to lift the uncertain parameters $\uvec\in\mathbb{R}^\numu$ to a higher dimensional parameter vector $\upvec\in\mathbb{R}^\numuf_+$, and subsequently solve an affine decision rule approximation in the lifted space, which corresponds to a nonlinear decision rule approximation in the original space. 
To do so, we first select embedding matrices $\embed_t$ that map the original uncertainties $\uvec^t$ up to stage $t$ into stage-$t$ uncertainties $\uevec_t = \embed_t \uvec^t$, $t \in [T]$. 
We require each embedding $\embed^t=(\embed_1 \cdots \embed_t)$ up to stage $t$ to be information-preserving, that is, to be an injective map of $\uvec^t$ to $\uevec^t = (\uevec_1, \ldots, \uevec_t)$. 
Typical embeddings include rotations, scalings, averages and differences. 
The uncertain parameters $\uevec = \uevec^T$ then reside in the embedded support $\UE = \embed \U = \{ \embed \uvec \, \colon \, \uvec \in \U \}\subseteq\mathbb{R}^\numue$, where $\embed = \embed^T$. 
Next, we select for each scalar component $\uescal_i$ of $\uevec$ the $\numuef_i - 1$ breakpoints
$$
\min_{\uevec \in \UE} \; \uescal_i = \;
\underline{\uescal}_i = z_{i0} < z_{i1} < z_{i2} < \dots < z_{i,\numuef_i-1} <  z_{i\numuef_i} = \overline{\uescal}_i
\; = \max_{\uevec \in \UE} \; \uescal_i.
$$
Here, $z_{i0}$ and $z_{i\numuef_i}$ are the lower and upper bounds of the embedded support $\UE$ and not breakpoints. We then define the nonlinear folding operator $\fold$ component-wise via
\begin{align*}
&\fold_{ij}(\uescal_i) = \begin{cases}
    0 & \case \uescal_i < z_{i,j-1}, \\
    \uescal_i - z_{i,j-1} & \case \uescal_i \in [z_{i,j-1}, z_{ij}], \\
    z_{ij} - z_{i,j-1} & \case \uescal_i > z_{ij}
\end{cases}
& \forall i\in [\numue], j\in[\numuef_i].
\end{align*}

$\fold$ piecewise linearly folds each component $\uescal_i$ to $\numuef_i$ lifted components in a staggered manner. 
We denote by $\fold_i(\uescal_i)=(\fold_{i1}(\uescal_i),\dots,\fold_{i\numuef_i}(\uescal_i))\in\mathbb{R}^{\numuef_i}$ the subvector of $\fold(\uevec)$ corresponding to the folding of component $\uescal_i$. 
Geometrically, $\fold_i$ maps $\uescal_i$ to an edge of the $\numuef_i$-dimensional hyperrectangle spanned by the folded breakpoints. 
Figure~\ref{fig:visualization:fold} visualizes both of these interpretations.
By construction, we have $\underline{\uescal}_i + \sum_{j'\in[j]} \fold_{ij'}(\uescal_i) = \min\{\uescal_i, z_{ij}\}$ for all $i\in[\numue]$ and $j\in[\numuef_i]$. 
The lifted uncertainty realizations $\upvec = \fold (\uevec)$ reside in the lifted support $\UF = \fold (\UE) = \{ \fold (\uevec) \, \colon \, \uevec \in \UE \}$. 
We define the partial foldings $\upvec_t = \fold_t (\uevec_t)$ and $(\upvec)^t = \fold^t (\uevec^t)$ in the canonical way. 
Note that the embedding and folding operators have the left-inverses $\embed^+= (\embed^\intercal \embed)^{-1} \embed^\intercal$ and $\fold^+$ defined via $\fold^+_i(\ufvec) = \underline{\uescal}_i + \sum_{j \in [\numuef_i]} \ufscal_{ij}$, respectively. 
We combine the operators and their left-inverses to the lifting operator $\lift = \fold \circ \embed$ and the retraction operator $\retr = \embed^+ \circ \fold^+$, respectively. 
Finally, we define the lifted probability distribution $\PUF$ via $\PUF(S')=\PU(\{\uvec\in\U \,\colon\, \lift(\uvec)\in S'\})$ for all Borel-measurable sets $S'\in \mathfrak{B}(\mathbb{R}^\numuf)$, where $\mathfrak{B}(\mathbb{R}^\numuf)$ is the Borel $\sigma$-algebra on $\mathbb{R}^\numuf$.

\begin{figure}[htb]
    \centering
    \begin{subfigure}[t]{0.48\textwidth}
        \input{figures/folding_visualization_staggered}
        \subcaption{Interpreting $\fold_i$ as staggered piecewise affine functions.}
        
    \end{subfigure}\hspace{.02\textwidth}
    \begin{subfigure}[t]{0.48\textwidth}
        \input{figures/folding_visualization_box}
        \subcaption{Interpreting $\fold_i$ as mapping to edges of the hyperrectangle spanned by the folded breakpoints.}
    \end{subfigure}
    \caption{Visualization of the piecewise affine folding function $\fold_i$.} 
    \label{fig:visualization:fold}
\end{figure}

Using the lifted probability distribution $\PUF$, the lifted support subsets $\UF_g = \lift(\U_g)$ and the retraction operator $\retr$, we define the lifted problem
\begin{equation}
\tag{$\lbaseproblem (\PUF, \UF)$}
\label{eq:LP}
\begin{aligned}
    \minimize_{\xfvec\in\XF} \quad& \mathbb{E}_\PUF \left[ \cvec(\retr(\sufvec))^\intercal\xfvec(\sufvec)\right]\\
    \st \quad& \amat_g \xfvec(\ufvec) \leq \bvec_g(\retr(\ufvec)) && \forall g\in[\nums], \ufvec\in \UF_g,
\end{aligned}  
\end{equation}
where now $\UF = \{ \UF_g \}_{g \in [G]}$, and $\XF$ encodes the non-anticipativity requirement in the lifted space. 
Since the embedding and folding operators preserve information, \ref{eq:LP} is equivalent to \ref{eq:P} if both problems are solved in general decision rules, and \ref{eq:LP} is an at least as tight conservative approximation as \ref{eq:P} if both problems are solved in affine decision rules. 
For the remainder of the paper, we refer to the affine decision rule approximations of \ref{eq:P} and \ref{eq:LP} as $\aff\baseproblem (\PU, \U)$ and $\aff\lbaseproblem (\PUF, \UF)$, respectively.

Unfortunately, the representation of the lifted support subsets $\UF_g$ typically scales exponentially in the representation of the original sets $\U_g$, which renders $\aff\baseproblem (\PUF, \UF)$ intractable in general. To this end, \citet{Georghiou2015} study the outer approximations
\begin{align}
\label{eq:outer_general}
\UFOG_g \;=\;  
\left\{ 
\ufvec \, \colon \,
\fold^+ (\ufvec) \in\UE_g \right\} \;\cap\; 
\FBOXg,
\qquad g\in[\nums],
\end{align}
of the lifted support subsets $\UF_g$ for the special case when $G=1$. Here, $[\underline{\uevec}_g, \overline{\uevec}_g]\subseteq [\underline{\uevec}, \overline{\uevec}]$ is the smallest hyperrectangle containing $\UE_g = \embed \U_g$. Intuitively, the first set in~\eqref{eq:outer_general} relaxes the sufficient condition $\ufvec\in\UF_g=\fold(\UE_g)$ to the necessary condition $\fold^+(\ufvec)\in\fold^+\circ\fold(\UE_g) = \UE_g$. 
Since $\fold^+$ is affine, this set is efficiently representable whenever $\UE_g$ is. The second set in~\eqref{eq:outer_general} is the convex hull of the folded bounding box of the embedded support, which has a polyhedral representation as per the following generalization of \citet[Lemma~4.3]{Georghiou2015}.

\begin{lemma}
The convex hull of the folded bounding box of the embedded support satisfies
\label{lem:rectangular_hull}
\begin{equation*}
\FBOXg
=
\left\{
\begin{aligned}
\ufvec \,\colon\,&
\fold_{ij}(\underline{\uescal}_{gi}) \leq \ufscal_{ij} \leq \fold_{ij}(\overline{\uescal}_{gi}) \;&&\forall i\in[\numue],j\in[\numuef_i]
\\
&\frac{\ufscal_{ij} - \fold_{ij}(\underline{\uescal}_{gi})}
{\fold_{ij}(\overline{\uescal}_{gi}) - \fold_{ij}(\underline{\uescal}_{gi})}
\geq
\frac{\ufscal_{i,j+1} - \fold_{i,j+1}(\underline{\uescal}_{gi})}
{\fold_{i,j+1}(\overline{\uescal}_{gi}) - \fold_{i,j+1}(\underline{\uescal}_{gi})}
\; &&\begin{aligned}
    &\forall i\in[\numue],j\in[\numuef_i] \\[-1ex]
    &\mathrm{ with }\, \underline{\uescal}_{gi} < z_{ij} < \overline{\uescal}_{gi}
\end{aligned}
\end{aligned}
\right\}.
\end{equation*}
\end{lemma}

Lemma~4.3 of \cite{Georghiou2015} relies on the assumption that the bounding box $[\underline{\uevec}_1, \overline{\uevec}_1]=\EBOX$ aligns with the breakpoints, which implies that $\FBOX$ coincides with the convex closure of the vertices of the hyperrectangles visualized in Figure~\ref{fig:visualization:fold}~(b). 
This does not hold for $G>1$, which is covered by Lemma~\ref{lem:rectangular_hull} above, however, since in that case the bounding boxes $[\underline{\uevec}_g, \overline{\uevec}_g]$ may not be aligned with the breakpoints. 
As a consequence, the bounds $\underline{\uevec}_g, \overline{\uevec}_g$ may no longer be lifted to vertices of the hyperrectangles, but instead onto their edges. 
We elaborate on this in Appendix~\ref{appendix:proof:lem:rectangular_hull}.

\begin{figure}[htb]
    \centering
    \begin{tikzpicture}

    \coordinate (e-origin) at (-.5, -1.3);
    \coordinate (z0) at (0, -1.3);
    \coordinate (zl) at (.7, -1.3);
    \coordinate (z1) at (2.5, -1.3);
    \coordinate (zu) at (4.4, -1.3);
    \coordinate (z2) at (5, -1.3);
    \coordinate (e-end) at (5.7, -1.3);

    \draw[->] (e-origin) -- (e-end) node[right] {$\uescal_{i}$};
    \draw (z0) ++(0,.2) -- ++(0,-.4) node[below] {$z_{i0}$};
    \draw (zl) ++(0,.2) -- ++(0,-.4) node[below] {$\underline{\uescal}_{gi}$};
    \draw (z1) ++(0,.2) -- ++(0,-.4) node[below] {$z_{i1}$};
    \draw (zu) ++(0,.2) -- ++(0,-.4) node[below] {$\overline{\uescal}_{gi}$};
    \draw (z2) ++(0,.2) -- ++(0,-.4) node[below] {$z_{i2}$};
    \draw[line width=4, green, draw opacity=.4] (zl) -- (zu);
    
    \coordinate (origin) at (0,0);
    \coordinate (u1-end) at (3.2,0);
    \coordinate (u2-end) at (0,3.2);
    \coordinate (z1-begin) at (2.5,0);
    \coordinate (z2-begin) at (0,2.5);
    \coordinate (z12-end) at (2.5,2.5);
    \coordinate (ll) at (.7, 0);
    \coordinate (lu) at (2.5, 1.9);

    \draw (origin) node[left] {$0$};
    \draw[->] (origin) -- (u1-end) node[right] {$\ufscal_{i1}$};
    \draw[->] (origin) -- (u2-end) node[above] {$\ufscal_{i2}$};
    \draw[dashed] (z12-end)++(0,.7) node[above] {$z_{i1}-z_{i0}$} -- (z1-begin) ;
    \draw[dashed] (z12-end)++(.7,0) node[right] {$z_{i2}-z_{i1}$} -- (z2-begin);
    \fill[fill=black, fill opacity=.2] (ll) -- (z1-begin) -- (lu) -- cycle;
    \draw[line width=4, green, draw opacity=.4] (ll) -- (z1-begin) -- (lu);
    \draw[dotted, very thick] (.43,-.285) -- (2.86,2.28);

    \draw[dotted, ->, opacity=.5, thick, shorten >=3pt] (zl) to node[midway, left, opacity=1] {$\fold_i(\underline{\uescal}_{gi})$} (ll);
    \draw[dotted, ->, opacity=.5, thick, shorten >=3pt] (z1) -- (z1-begin) node[midway, right, opacity=1] {$\fold_i(z_{i1})$};
    \draw[dotted, ->, opacity=.5, thick, shorten >=3pt] (zu) to[bend right] node[midway, right, opacity=1] {$\fold_i(\overline{\uescal}_{gi})$} (lu);
\end{tikzpicture}
    \caption{
    Visualization of $\FBOXg$. The embedded support subset $[\underline{\uescal}_{gi},\overline{\uescal}_{gi}]$ (green line segment on the abscissa) is folded into a two-dimensional non-convex set (two perpendicular green line segments). The convex hull of $\fold ([\underline{\uescal}_{gi},\overline{\uescal}_{gi}])$ is shown as a grey shaded set that is bounded by the inequality $\frac{\ufscal_{i1} - \fold_{i1}(\underline{\uescal}_{gi})}
{\fold_{i1}(\overline{\uescal}_{gi}) - \fold_{i1}(\underline{\uescal}_{gi})}
\geq
\frac{\ufscal_{i2} - \fold_{i2}(\underline{\uescal}_{gi})}
{\fold_{i2}(\overline{\uescal}_{gi}) - \fold_{i2}(\underline{\uescal}_{gi})}$ (dotted line).
    }
    \label{fig:oute\numuef_intuition}
\end{figure}

Since both sets on the right-hand side of~\eqref{eq:outer_general} constitute outer approximations of $\UF_g$, we conclude that $\UFOG_g$ indeed contains $\UF_g$. Replacing the lifted support subsets $\UF_g$ with the outer approximations $\UFOG_g$ thus results in a conservative approximation of problem~\ref{eq:LP}. Note that by linearity of $\fold^+$ and $\embed$, each outer approximation $\UFOG_g$ is convex whenever $\U$ is. 
Figure~\ref{fig:oute\numuef_intuition} visualizes the convex hull of the lifted bounding box $\FBOXg$ for one dimension $i\in[\numue]$.
 
Since $\retr$ is affine, every feasible solution $\xvec$ to $\aff \baseproblem (\PU, \U)$ induces a feasible solution $\xfvec = \xvec \circ \retr$ to $\aff\lbaseproblem (\PUF, \UFOG)$, where $\UFOG = \{ \UFOG_g \}_{g \in [G]}$, with the same objective value. 
Thus, $\aff\lbaseproblem (\PUF, \UFOG)$ is a weakly tighter approximation of \ref{eq:P} than $\aff \baseproblem (\PU, \U)$. 
\citet{Georghiou2015} observed significant improvements of $\aff \lbaseproblem(\PUF, \UFOG)$ over $\aff \baseproblem(\PU, \U)$ in stochastic programs (\emph{cf.}~Example~\ref{ex:SP}). 
However, previous works did not observe any improvements of $\aff \lbaseproblem (\PUF, \UFOG)$ over $\aff \baseproblem(\PU, \U)$, when applying the outer approximation~\eqref{eq:outer_general} to the robust setting (\emph{cf.}~Example~\ref{ex:RP}), see, e.g., \cite{Bertsimas2015b, Han2025}. 
We next provide the theoretical justification for these observations.
\begin{theorem}
\label{thm:negative_SP}

For $\nums = 1$, there exists a subset $\overline{\Psi}{}'\subseteq\UFOG$ with $\retr(\overline{\Psi}{}') = \U$ such that every feasible solution $\xfvec$ in $\aff\lbaseproblem(\PUF, \UFOG)$ has a corresponding feasible solution $\xvec$ in $\aff\baseproblem(\PU, \U)$ with $\xfvec(\ufvec) = \xvec(\retr(\ufvec))$ for all $\ufvec\in\overline{\Psi}{}'$.
\end{theorem}

Theorem~\ref{thm:negative_SP} shows that solutions $\xfvec$ to $\aff \lproblem(\PUF, \UFOG)$ can be interpreted as ``extensions'' of solutions $\xvec$ to $\aff \baseproblem(\PU, \U)$. Thus, we do not gain flexibility in terms of feasibility when solving $\aff \lproblem(\PUF, \UFOG)$ instead of $\aff \baseproblem(\PU, \U)$. However, two solutions $\xfvec, \yvec'$ that extend the same solution $\xvec$, that is, $\xfvec(\ufvec) = \xvec(\retr(\ufvec)) = \yvec'(\ufvec)$ for all $\ufvec\in\overline{\Psi}{}'$, may differ outside $\overline{\Psi}{}'$ and thus yield different objective values under the lifted probability distribution $\PUF$.
In other words, any improvements observed in stochastic programming problems (\emph{cf.}~Example~\ref{ex:SP}) stem from the lifted probability distribution in the objective function.
In robust optimization problems (\emph{cf.}~Example~\ref{ex:RP}), on the other hand, the lifted probability distribution does not affect the objective function, and thus Theorem~\ref{thm:negative_SP} implies that $\aff\lbaseproblem(\PUFO, \UFOG)$ cannot improve over $\aff\baseproblem(\PU, \U)$ in the robust case.

\begin{corollary}
\label{thm:negative_RP}
$\aff \lproblem (\PUF, \UFOG)$ and $\aff\baseproblem (\PU, \U)$ attain the same optimal value for robust optimization problems (\textit{cf.}~Example~\ref{ex:RP}).
\end{corollary}

A special case of Corollary~\ref{thm:negative_RP} that applies to a restricted class of problems with non-negative right-hand-side uncertainty and the absence of the embedding operator $\embed$ has been proved by \citet{Thomae2024}. 
This restricted class of problems excludes prominent applications such as the inventory management problem considered in our numerical experiments, as well as the multi-item newsvendor problem in which the absence of improvements over affine policies has most recently been observed numerically \citep{Han2025}. 
Corollary~\ref{thm:negative_RP} generalizes this result and implies that the piecewise affine decision rules of \citet{Georghiou2015} do not improve over affine policies in the broader class of robust optimization problems covered by Example~\ref{ex:RP}. 

Theorem~\ref{thm:negative_SP} and Corollary~\ref{thm:negative_RP} extend to the setting where the constraint matrices $\amat_g$ are uncertain. 
As $\aff\baseproblem(\PU, \U)$ does, in general, not admit tractable reformulations when $\amat_g$ is uncertain, we defer the discussion of this more general setting to Appendix~\ref{appendix:proof:thm:negative_SP}.

In the following, we develop tighter outer approximations $\UFO$ of the lifted support $\UF$ that result in superior approximations $\aff \lproblem(\PUF, \UFO)$ of $\baseproblem (\PU, \U)$ than $\aff \lproblem (\PUF, \UFOG)$. In particular, our approximations ensure that $\aff \lproblem (\PUFO, \UFO)$ no longer coincides with $\aff \baseproblem (\PU, \U)$ even in the robust case.

\section{Valid Cuts for the Lifted Support $\UF$}
\label{sec:tightening}

We tighten the outer approximations $\UFOG$ of the lifted support $\UF$ through cuts that are satisfied by all lifted uncertainty realizations $\ufvec\in\UF$ but that are violated by at least some of the additional realizations $\ufvec\in\UFOG \setminus \UF$ introduced by the approximation~\eqref{eq:outer_general}. 
To this end, we study $l_1$-norm distances between embedded uncertainty realizations $\uevec\in\UE$ and hyperrectangles $[\zvec^-, \zvec^+] \subseteq \UE$,
$$
\dist(\uevec, [\zvec^-, \zvec^+]) = \min_{\uevec'\in[\zvec^-, \zvec^+]} \norm{\uevec - \uevec'}_1,
$$
where the hyperrectangles are selected from the discrete set $\mathcal{Z} = \{[\zvec^-, \zvec^+] \,\colon\, \zvec^-, \zvec^+\in Z,\; \zvec^-\leq\zvec^+\}$ whose vertices align to our breakpoint grid $Z = \bigtimes_{i\in[\numue]} \{ z_{ij} \,\colon\, j\in[\numuef_i]_0\}$. Note that by construction, all lifted uncertainty realizations $\ufvec\in\UF_g$ satisfy
\begin{equation*}
\dist(\fold^+(\ufvec), [\zvec^-, \zvec^+]) \leq \maxdist_g([\zvec^-, \zvec^+])
\quad \text{with} \quad
\maxdist_g([\zvec^-, \zvec^+]) = \max_{\uevec\in\UE_g}\;\dist(\uevec, [\zvec^-, \zvec^+]).
\end{equation*}
Our cuts for the lifted support rely on an extension of the $l_1$-norm grid distance $d$ to embedded uncertainty realizations $\fold^+(\ufvec)$, $\ufvec\in \UFOG \setminus \UF$ that we characterize next.

\begin{observation}
\label{obs:distance_formalization}
For all $\ufvec\in\UFOG$, we have
\begin{equation}\label{eq:the-mother-of-all-distances}
\dist(\fold^+(\ufvec), [\zvec^-, \zvec^+]) \leq
\sum_{i\in[\numue]} \left(
    \sum\limits_{j \in \numuef^+_i} \ufscal_{ij} 
    + z^-_i - \underline{\uescal}_i
    -\sum\limits_{j \in \numuef^-_i} \ufscal_{ij} \right)
    ,
\end{equation}
where $\numuef^+_i = \{j\in[\numuef_i] \colon z^+_i < z_{ij}\}$ and $\numuef^-_i = \{j\in[\numuef_i] \colon z^-_i \geq z_{ij}\}$ are the index sets corresponding to the folded uncertainty components outside the box $[\zvec^-, \zvec^+]$. 
Inequality~\eqref{eq:the-mother-of-all-distances} is tight for $\ufvec\in\UF$.
\end{observation}

In the following, we use the notational shortcut $\fdist (\ufvec, [\zvec^-, \zvec^+]) = \sum_{i\in[\numue]} (\sum_{j \in \numuef^+_i} \ufscal_{ij} + z^-_i - \underline{\uescal}_i -\sum_{j \in \numuef^-_i} \ufscal_{ij} )$ to refer to the right-hand side of~\eqref{eq:the-mother-of-all-distances}.
Observation~\ref{obs:distance_formalization} motivates us to study the refined outer approximation $\UFO = \{ \UFO_g \}_{g \in [G]}$, where
\begin{equation}\label{eq:new-and-shiny-outer-approximation}
\UFO_g \;=\;  
\left\{ 
\ufvec \in \UFOG_g \, \colon \,
\fdist(\ufvec, [\zvec^-, \zvec^+]) 
\leq 
\maxdist_g ([\zvec^-, \zvec^+]) 
\;\; \forall [\zvec^-, \zvec^+]\in\mathcal{Z} 
\right\},
\qquad
g\in[\nums],
\end{equation}
of the lifted support subsets $\UF_g$. Figure~\ref{fig:grid_construction} illustrates our grid distance $d'$.

\begin{figure}[t]
    \centering
    \begin{tikzpicture}
\foreach \x in {0, ..., 5}{
    \coordinate (z1\x) at (\x-.5,-.7);
    \draw[thin, dashed, opacity=.5] (z1\x) -- ++(0,4.5);
    \draw (z1\x) node[below]{$z_{1\x}$};
    }

\foreach \y in {0, ..., 4}{
    \coordinate (z2\y) at (-.7,\y-.5);
    \draw[thin, dashed, opacity=.5] (z2\y) -- ++(5.5,0);
    \draw (z2\y) node[left]{$z_{2\y}$};
    }

\foreach \x in {0, ..., 5}{
\foreach \y in {0, ..., 4}{
    \fill [black, opacity=.5] (\x-.5,\y-.5) circle (.05);
    }}

\draw[thick, ->] (-0.7,0) -- (4.9,0) node[right] {$\uescal_1$};
\draw[thick, ->] (0,-.7) -- (0,3.9) node[above] {$\uescal_2$};

\begin{scope}[shift={(2,1.5)}]
  \filldraw[thick,green,rotate=-57,fill opacity=.15] (0,0) ellipse [x radius=1.5, y radius=2.8];
\end{scope}

\coordinate (z-) at (2.5,0.5);
\coordinate (z+) at (3.5,2.5);
\filldraw[blue, thick, fill opacity=.15] (z-) rectangle (z+);
\draw (z-) node[above right] {$\zvec^-$};
\draw (z+) node[below left] {$\zvec^+$};

\coordinate (theta) at (1,2);
\draw[thick, dotted] (theta) node[below] {$\uevec$} -- node[above, midway]{$\fdist$} ++(1.5,0);
\fill (theta) circle(.05);

\draw[thick, dotted] (z-) -- ++(0,-.75) node[right]{$\maxdist_g$} -- ++(-2.65,0);
\fill (z-)++(-2.65,-.75) circle(.05);

\draw [decorate,decoration={brace}]
(.45,3.7) -- (2.55,3.7) node [black,midway, above] 
{$\numuef^-_1$};
\draw [decorate,decoration={brace}]
(4.35,3.7) -- (4.65,3.7) node [black,midway, above] 
{$\numuef^+_1$};
\draw [decorate,decoration={brace}]
(4.7,.65) -- (4.7,.35) node [black,midway, right] 
{$\numuef^-_2$};
\draw [decorate,decoration={brace}]
(4.7,3.65) -- (4.7,3.35) node [black,midway, right] 
{$\numuef^+_2$};
\end{tikzpicture}
    \caption{
    Visualization of the breakpoint grid $Z$ (grey), 
    a gridpoint-induced rectangle $[\zvec^-, \zvec^+]$ (blue, solid),
    the index sets $\numuef^-_1 = \{1,2,3\}$, $\numuef^+_1 = \{5\}$, $\numuef^-_2 = \{1\}$ and $\numuef^+_2 = \{4\}$ corresponding to the embedded uncertainty components outside the rectangle, 
    the distance $\fdist(\fold(\uevec), [\zvec^-, \zvec^+])$ between an embedded uncertainty realization $\uevec\in\UE$ and the rectangle, 
    and the maximum distance $\maxdist_g([\zvec^-, \zvec^+])$ between any point in the embedded support subset $\UE_g$ (green, ellipsis) and the rectangle.
    To ease the exposition, we omit the arguments of $\fdist$ and $\maxdist_g$.
    }
    \label{fig:grid_construction}
\end{figure}

Recall that a set $\mathcal{X} \subseteq \mathbb{R}^{\numu}$ is conic representable if $\mathcal{X} = \{ \xvec \in\mathbb{R}^{\numu} \,\colon\, \vmat \xvec \succeq_{\mathcal{K}} \dvec\}$ for some $\vmat$, $\dvec$ and a proper cone $\mathcal{K}$. 
We next present convex reformulations of the problem $\aff\lproblem(\PUF, \UFOG)$ along the lines of \citet{Georghiou2015} and our tightened approximation $\aff \lproblem(\PUF, \UFO)$.

\begin{observation}\label{obs:conic-reformulation}
Assume that the support subsets $\U_g$ are conic representable. Then:
\begin{enumerate}[label=(\roman*)]
    \item\label{obs:conic-reformulation-UFOG} If $\aff\lproblem(\PUF, \UFOG)$ is feasible with a Slater point, it has the same optimal value as the convex conic program
        \begin{equation}
        \label{eq:counterpart_dual}
        \begin{aligned}
            \maximize_{\ufvec, \svec} \quad 
            & - \sum_{g\in[\nums],k\in[\numc_g]} b_{gk} (\retr(\ufvec_{gk})) + (s_{gk}-1)b_{gk}(\retr(\nullvec)) \\
            \st \quad & \sum_{g\in[\nums],k\in[\numc_g]} \amat_{gk}^\intercal \ufvec^\intercal_{gk} = - \mathbb{E}_{\PUF}[\cvec(\retr(\sufvec))\sufvec^\intercal]\\
            &\sum_{g\in[\nums],k\in[\numc_g]} s_{gk} \amat_{gk}^\intercal = -\mathbb{E}_{\PUF}[\cvec(\retr(\sufvec))] && \\
            &\ufvec_{gk} \in s_{gk} \UFOG_g, \;\; s_{gk} \geq 0 && \forall g\in[\nums], k\in[\numc_g],
        \end{aligned}
        \end{equation}
        where $\amat_{gk}$ is the $k$\textsuperscript{th} of the $\numc_g$ rows of $\amat_g$ and $b_{gk}(\uvec)$ is the $k$\textsuperscript{th} component of $\bvec_g(\uvec)$. 
        In particular, the constraint $\ufvec_{gk} \in s_{gk} \UFOG_g$ is conic representable for any  $g\in[\nums]$, $k\in[\numc_g]$, and $s_{gk} \geq 0$. 
    \item\label{obs:conic-reformulation-UFO} If $\aff \lproblem(\PUF, \UFO)$ is feasible with a Slater point, it has the same optimal value as~\eqref{eq:counterpart_dual} tightened by the affine cuts
    \begin{equation}
        \label{eq:counterpart_dual_cuts}
        s_{gk}\fdist\left(\frac{1}{s_{gk}}\ufvec_{gk}, [\zvec^-, \zvec^+]\right) \leq s_{gk}\maxdist_g ([\zvec^-, \zvec^+]) \qquad \forall g\in[\nums], k\in[\numc_g], [\zvec^-, \zvec^+]\in\mathcal{Z}.
    \end{equation}
\end{enumerate}
\end{observation}

We can readily extract optimal policies $\xfvec(\cdot)$ to $\aff\lproblem(\PUF, \UFOG)$ and $\aff \lproblem(\PUF, \UFO)$ from the shadow prices of~\eqref{eq:counterpart_dual} as well as its tightened counterpart presented in Observation~\ref{obs:conic-reformulation}~\ref{obs:conic-reformulation-UFO}. We prefer to work with the dual formulations of both problems, as opposed to the primal formulations that are predominant in the literature, since the duals facilitate the delayed introduction of the cuts from~\eqref{eq:new-and-shiny-outer-approximation} as discussed in the next section.

\section{Efficient Cut Separation}
\label{sec:solving_subproblem}

Our refined outer approximation $\UFO$ in~\eqref{eq:new-and-shiny-outer-approximation} comprises exponentially many constraints indexed by $[\zvec^-, \zvec^+] \in \mathcal{Z}$. Rather than including all of these constraints \emph{a priori}, we use a constraint generation approach to iteratively add constraints that are violated by incumbent solutions $(\ufvec, \bm{s})$ to problem $\aff \lproblem(\PUF, \UFO)$. To this end, we need to solve the separation problem
\begin{equation}
\label{eq:subproblem}
\max_{[\zvec^-, \zvec^+]\in\mathcal{Z}} \fdist(\ufvec_{gk} / s_{gk}, \, [\zvec^-, \zvec^+]) \; - \; \maxdist_g([\zvec^-, \zvec^+])
\end{equation}
for each $g\in[\nums]$ and $k\in[\numc_g]$ with $s_{gk}>0$ to identify maximally violated constraints (\emph{cf.}~Observation~\ref{obs:conic-reformulation}). Unfortunately, determining maximally violated constraints is hard in general.

\begin{proposition}
\label{pro:np_hard}
The separation problem~\eqref{eq:subproblem} is strongly NP-hard.
\end{proposition}

Note, however, that $\UFO$ remains a valid outer approximation of $\UF$ for any subset of the constraints in $\mathcal{Z}$, and thus we can solve the separation problem~\eqref{eq:subproblem} heuristically and/or terminate our constraint generation approach before all violated constraints have been included. These approaches are well-established in the cutting plane literature, and we do not elaborate upon them further in the following. Instead, Section~\ref{sec:symmetric} shows how the separation problem~\eqref{eq:subproblem} can be solved exactly in polynomial time for permutation invariant uncertainty embeddings, and Section~\ref{sec:general} uses these insights \mbox{to approximately solve the separation problem for general embeddings.}

\subsection{Permutation Invariant Uncertainty Embeddings}
\label{sec:symmetric}

In this subsection, we assume that $G = 1$, that is, we focus on stochastic and robust problems (\emph{cf.}~Examples~\ref{ex:SP} and~\ref{ex:RP}). Consequently, we drop the subscript $g$ to minimize clutter. Denote by $\mathcal{S}$ the set of all signed permutation matrices in $\mathbb{R}^{\numue}$, that is, all matrices $\bm{S} \in \mathbb{R}^{\numue \times \numue}$ with elements $-1$, $0$ and $1$ that have exactly one non-zero element in each row and each column. Throughout this subsection, we make the following two assumptions about $Z$ and $\UE$.

\begin{assumption}[Permutation Invariance of Breakpoints]
\label{ass:symmetric_Z} The breakpoints $Z$ are permutation invariant, that is, $\smat \zvec \in Z$ whenever $\zvec \in Z$ and $\smat \in \mathcal{S}$. Additionally, $\nullvec\in Z$.
\end{assumption}

\begin{assumption}[Permutation Invariance of Embedded Support]
\label{ass:symmetric_concave}
\,
\begin{enumerate}[label=(\roman*)]
\item \label{ass:symmetric_U} The embedded support $\UE$ is permutation invariant, that is, $\smat \uevec \in \UE$ if $\uevec \in \UE$ and $\smat \in \mathcal{S}$.
\item \label{ass:concave} The maximal $l_1$-norm of the first $i$ components of the embedded support $\UE$,
$$
\eta(i) = \max_{\uevec \in \UE} \sum_{i'\in[i]} \lvert \uescal_{i'} \rvert,
$$
has non-increasing differences, that is, $\eta(i+1) - \eta(i) \leq \eta(i) - \eta(i-1)$ for all $i \in[\numue-1]$.
\end{enumerate}
\end{assumption}

Assumption~\ref{ass:symmetric_Z} implies that the breakpoints are symmetric around the origin as well as uniform across all axes. The assumption is relatively benign as the breakpoints are a design choice of the decision maker rather than a characteristic of the problem setting. 
Due to Assumption~\ref{ass:symmetric_Z}, the numbers of breakpoints $\numuef_i$ as well as the breakpoint locations $z_{ij}$ no longer depend on the index $i$, which allows us to drop the index in both quantities to further reduce clutter. Assumption~\ref{ass:symmetric_concave}~\ref{ass:symmetric_U} applies a similar condition as Assumption~\ref{ass:symmetric_Z} to the geometry of the embedded support~$\UE$. 
Additionally, Assumption~\ref{ass:symmetric_concave}~\ref{ass:concave} implies that as we consider additional components of $\uevec$, the incremental growth of its maximum $l_1$-norm over $\UE$ is monotonically non-increasing. 
Intuitively, this can be interpreted as a ``diminishing returns'' behavior, where each new component of $\uevec$ contributes less to increasing the maximum norm over $\UE$. 
Note that the condition on $\eta$ in Assumption~\ref{ass:symmetric_concave}~\ref{ass:concave} resembles the notion of submodularity. However, submodularity is typically leveraged when minimizing a function, whereas our setting involves maximization over $\eta$ (\emph{cf.}~Lemma~\ref{lem:symmetric_dist} in Appendix~\ref{appendix:proof:thm:symmetric}), and maximizing a submodular function is known to be NP-hard in general. Instead, one may intuitively regard the non-increasing differences condition as a discrete version of concavity, which enables efficient maximization over $\eta$. 
In contrast to the permutation invariance property, the non-increasing differences of $\eta$ do not appear to admit a straightforward geometric interpretation. 
However, by Lemma~\ref{lem:rotational_solution} in Appendix~\ref{appendix:proof:thm:symmetric}, Assumption~\ref{ass:symmetric_concave}~\ref{ass:concave} can be verified numerically by solving $\numue$ convex problems of the form $\max_{\uevec\in\UE}\sum_{i'\in[i]}\uescal_{i'}$ for $i\in [\numue]$, which is computationally tractable under the assumed convexity of $\U$. 
Additionally, we will see that Assumption~\ref{ass:symmetric_concave} is satisfied by various prominent uncertainty sets, for example by ellipsoids, budgets, norm balls, and their intersections. Together, Assumptions~\ref{ass:symmetric_Z} and~\ref{ass:symmetric_concave} allow us to dramatically reduce the search space when we determine the maximally violated constraint in~\eqref{eq:subproblem}.

\begin{theorem}[Polynomial-Time Separability]
\label{thm:symmetric}
    Assume that Assumptions~\ref{ass:symmetric_Z} and~\ref{ass:symmetric_concave} both hold.
    \begin{enumerate}[label=(\roman*)]
    \item \label{thm:symmetric_algo} The separation problem~\eqref{eq:subproblem} can be solved in time $\mathcal{O}(\numue \numuef \log \numue)$.
    \item \label{thm:symmetric_full} If for every $j \in [\numuef]$ there is at most one $i \in [\numue]$ such that $\eta(i) - \eta(i-1) \in [z_{j-1}, z_j)$, then the separation problem~\eqref{eq:subproblem} has a square solution $[z_j \bm{e}, -z_j \bm{e}]$ for some $j \in [\numuef / 2]$.
    \end{enumerate}
\end{theorem}

While Theorem~\ref{thm:symmetric}~\ref{thm:symmetric_algo} guarantees that the separation problem~\eqref{eq:subproblem} can be solved in polynomial time, our constraint generation approach may still require exponentially many iterations. In contrast, the condition of Theorem~\ref{thm:symmetric}~\ref{thm:symmetric_full} ensures that $J / 2$ cuts are sufficient to solve the separation problem~\eqref{eq:subproblem} for all possible incumbent solutions. In this case, we can directly add those $J / 2$ cuts to each lifted support subset $\UFO_g$, instead of executing a constraint generation approach. As $Z$ is a design choice, we can select it such that condition~\ref{thm:symmetric_full} of Theorem~\ref{thm:symmetric} holds whenever $\eta$ has strictly decreasing differences, that is, whenever $\eta(i+1) - \eta(i) < \eta(i) - \eta(i-1)$ for all $i$. In this context, a straightforward choice of permutation invariant breakpoints that satisfies the property of strictly decreasing differences is
\begin{equation}
\label{eq:full_break_points}
    Z = \Big( \left\{ \eta(i) - \eta(i-1) \, \colon \, i \in [\numue] \right\} \, \cup\, \{0\} \cup \, \left\{ -(\eta(i) - \eta(i-1)) \, \colon \, i \in [\numue] \right\} \Big)^\numue,
\end{equation}
or any superset thereof.

\begin{proposition}[Efficient Separation for Norm-Balls]
\label{pro:pballs}
Any $l_p$-norm ball embedded support $\{\uevec \,\colon\, \norm{\uevec}_{p} \leq \delta\}$ with $1\leq p \leq \infty$ and $\delta>0$ satisfies Assumption~\ref{ass:symmetric_concave}. 
\end{proposition}

\begin{proposition}[Efficient Separation for Intersections]
\label{pro:intersections}
Let $\UE_1, \dots, \UE_n$ with $\eta_1,\dots,\eta_n $ satisfy Assumption~\ref{ass:symmetric_concave}. 
Then $\UE =\bigcap_{h\in [n]} \UE_h$ also satisfies Assumption~\ref{ass:symmetric_concave} for $\eta(i) = \min \{ \eta_h(i) \, \colon \, h\in[n] \}$.
\end{proposition}

Combining Propositions~\ref{pro:pballs} and~\ref{pro:intersections}, we cover several commonly used supports, including ellipsoidal, budgeted, and intersected norm-ball supports.

\subsection{General Uncertainty Embeddings}
\label{sec:general}

We now revisit the general case where $G \geq 1$, which includes the stochastic, robust and data-driven settings (\emph{cf}.~Examples~\ref{ex:SP}--\ref{ex:DP}). Whenever $G > 1$, Assumption~\ref{ass:symmetric_concave} is most likely violated, and thus the polynomial-time separability of Theorem~\ref{thm:symmetric} no longer applies. We can, however, continue to apply the polynomial-time separation algorithms of Theorem~\ref{thm:symmetric} as if Assumption~\ref{ass:symmetric_concave} was satisfied. We next show that this heuristic has attractive approximation guarantees.

\begin{proposition}
    \label{pro:bounding_by_symmetric}
    Let $Z$ satisfy Assumption~\ref{ass:symmetric_Z}, 
    let $\UE = \{\UE_g\}_{g\in[G]}$ be the embedded support, and let
    $\UE^\mathrm{P} = \{ \UE^\mathrm{P}_g \}_{g \in [G]}$ be such that $\UE^\mathrm{P}_g \supseteq \UE_g$ and $\UE^\mathrm{P}_g$ satisfies Assumption~\ref{ass:symmetric_concave} for all $g \in [G]$. 
    Let $\UFO_g$ and $\UFOV{}_g^\mathrm{P}$ be the outer approximations of $\UF_g$ and $\U^\mathrm{P}_g{}'$ generated by the exponential-sized description~\eqref{eq:new-and-shiny-outer-approximation}.
    \begin{enumerate}[label=(\roman*)]
    \item \label{pro:bounding_by_symmetric:partial} The outer approximation $\UFOV{}^\mathrm{A}$ of $\UF$ generated by the separation algorithm of Section~\ref{sec:symmetric} satisfies 
        $$
        \aff\lproblem(\PUF, \UFO) \leq \aff\lproblem(\PUF, \UFOV{}^\mathrm{A}) \leq \aff\lproblem(\PUF, \UFOV{}^\mathrm{P}).
        $$
    \item \label{pro:bounding_by_symmetric:full} If $\UE^\mathrm{P}$ satisfies condition~\ref{thm:symmetric_full} of Theorem~\ref{thm:symmetric}, then statement (i) of this proposition continues to hold if we restrict each $\UFOV{}^\mathrm{A}_g$ to square cuts $[z_j \bm{e}, -z_j \bm{e}]$ for all $j \in [\numuef / 2]$.
    \end{enumerate}
\end{proposition}

Proposition~\ref{pro:bounding_by_symmetric}~\ref{pro:bounding_by_symmetric:partial} shows that applying our polynomial time separation algorithm of Section~\ref{sec:symmetric} directly to a general family of support sets $\U$ provides a conservative approximation of problem~\ref{eq:P} that is at least as tight as first circumscribing $\U$ by any outer approximation that satisfies Assumptions~\ref{ass:symmetric_Z} and~\ref{ass:symmetric_concave} and then applying our separation algorithm to that outer circumscription. Proposition~\ref{pro:bounding_by_symmetric}~\ref{pro:bounding_by_symmetric:full} strengthens this result to the subclass of circumscriptions that satisfy the strictly decreasing differences condition of Theorem~\ref{thm:symmetric}~\ref{thm:symmetric_full}.

\section{Approximation Guarantees}
\label{sec:bounds}

As an immediate implication of \citet[Theorem~6]{Feigea2007}, no finite approximation bounds exist for the general setting of $\baseproblem(\PU, \U)$ presented in Section~\ref{sec:lift} unless $P=NP$. 
However, our policies admit strong approximation bounds under assumptions commonly required in the literature on theoretical bounds for adjustable decision making \citep[see, e.g.,][]{Bertsimas2012, BenTal2020, Han2025}.

\begin{assumption}
\label{ass:approximation}
\,
\begin{enumerate}[label=(\roman*)]
    \item \label{ass:approximation:obj} The objective function $\cvec(\uvec)^\intercal \xvec(\uvec)$ is deterministic: $\cvec(\uvec)^\intercal \xvec(\uvec) = \cvec^\intercal\xvec_0$. 
    \item \label{ass:approximation:rhs} The affine function $\bvec_g(\uvec)$ in the constraint right-hand side has non-positive coefficients for each $g\in [G]$: $\bvec_g(\uvec) = \bmat_g \uvec + \hat{\bvec}_g$ with $\bmat_g, \hat{\bvec}_g \leq \nullvec$. 
    \item \label{ass:approximation:U} The support $\U_g$ is non-negative, full-dimensional, and down-monotone for each $g\in [G]$: $\forall \uvec\in\U_g, \nullvec\leq\hat{\uvec}\leq\uvec \Rightarrow \hat{\uvec}\in\U_g$. 
\end{enumerate}
\end{assumption}

While significantly restricting the general setting of $\baseproblem(\PU, \U)$, Assumption~\ref{ass:approximation} still covers a wide range of relevant problem classes, including robust capacity expansion \citep{Laguna1998}, inventory control \citep{Bertsimas2015}, and network design \citep{Wang2020}. 
Given part~\ref{ass:approximation:U} of Assumption~\ref{ass:approximation}, we can always choose $\embed$, such that the following also holds. 
\begin{assumption}
\label{ass:approximation:UE} The embedded support $\UE_g$ is convex, non-negative, full-dimensional with $e_i\in\UE_g$ for all $i\in[\numue]$, and down-monotone for each $g\in[G]$. 
\end{assumption}

Clearly, our piecewise affine policies inherit all approximation guarantees of affine policies. 
We further tighten these guarantees using \textit{scaling factors} $\beta(\uevec) = \min\{\beta'\in\mathbb{R}_+ \colon \exists \uevec' \in \bigcap_{g\in [G]}\UE_g, \beta'\uevec' \geq \uevec\}$  and the concept of \textit{dominating sets}, that is, sets $\hat{\UE}$ such that $\max_{\uevec\in \hat{\UE}} \beta(\uevec)$ is finite \citep{BenTal2020}.

\begin{proposition}
\label{pro:bounds}
Let $\baseproblem(\PU, \U)$ be a problem instance that satisfies Assumption~\ref{ass:approximation}, and let $\embed$ be such that Assumption~\ref{ass:approximation:UE} holds. Then,
$$
\aff\lproblem(\PUF, \UFO) \leq \min_{\zvec\in Z} \{\beta(\zvec) + \max_{g\in[G]}\maxdist_g([\nullvec, \zvec])\} \cdot \baseproblem(\PU, \U)
$$
and 
$$
\aff\lproblem(\PUF, \UFOV{}^\mathrm{A}) \leq \min_{\zvec\in Z^\mathrm{A}} \{\beta(\zvec) + \max_{g\in[G]}\maxdist_g([\nullvec, \zvec])\} \cdot \baseproblem(\PU, \U),
$$
where $Z^\mathrm{A}\subseteq Z$ is the subset of all gridpoints $\zvec$ such that $[\nullvec, \zvec]\in\mathcal{Z}$ satisfies all cuts of type \eqref{eq:counterpart_dual_cuts}.
\end{proposition}
Proposition~\ref{pro:bounds} directly connects the approximation bounds to gridpoints $\zvec$ such that $[\nullvec, \zvec]$ satisfies our new cuts~\eqref{eq:counterpart_dual_cuts}. 
Enforcing cuts for critical gridpoints secures strong guarantees, even if not all cuts were added. 
Moreover, choosing the right breakpoints has a significant impact on the bounds.
In particular, including a single breakpoint per dimension motivated by the policies introduced by \citet{BenTal2020} replicates all approximation bounds of their policies. 
These bounds also extend to the multi-stage setting and are, to the best of our knowledge, the tightest a priori bounds known so far. 

\section{Numerical Experiments}
\label{sec:numerical}
We test our piecewise affine decision rules on a variant of the single-product multi-period inventory replenishment problem studied by \citet{See2010} and \citet{Bertsimas2023}:

\begin{equation*}
\begin{aligned}
\minimize_{\xvec, y, \bm{I}} \quad& \mathbb{E}_\PU \left[\sum_{t\in[T]} c_0 y + c_t \lvert x_t(\suvec^{t-1})\rvert + h_t (I_t(\suvec^t))_+ + b_t (-I_t(\suvec^t))_+\right]\\
\st \quad& I_{t}(\suvec^{t}) = I_{t-1}(\suvec^{t-1}) + y + x_t(\suvec^{t-1}) - \suscal_t && \PU \pas,\forall t\in[T] \\
& \sum_{t\in[T]} (-I_t(\suvec^t))_+ \leq s && \PU \pas\\
&y\geq0,\; \underline{x}_t \leq x_t(\suvec^{t-1}) \leq \overline{x}_t && \PU \pas, \forall t\in[T]
\end{aligned}
\end{equation*}

In this problem, the decision maker pre-commits to a single replenishment quantity $y$ that applies to each time period $t \in [T]$ here-and-now, and reactively adjusts this quantity by $x_t(\suvec^{t-1})$ wait-and-see upon observing the uncertain customer demands $\suscal_{t-1}$. The objective function minimizes the aggregate sum of preorder costs $c_0 y$, adjustment costs $c_t \lvert x_t(\suvec^{t-1})\rvert$, inventory holding costs $h_t (I_t(\suvec^t))_+$ and backlogging costs $b_t (-I_t(\suvec^t))_+$ across all time periods. The constraints record the inventory levels $I_t(\suvec^{t})$, restrict the backlogged demands to at most $s$ units, and enforce non-negativity as well as the adjustment limits $[\underline{x}_t, \overline{x}_t]$, respectively. 

In the variant studied by \citet{See2010} and \citet{Bertsimas2023}, affine decision rules yield close to optimal policies. 
To make the problem more challenging, we added a pre-committed replenishment decision $y$ that is dynamically adjustable by $x_t$, and the service level constraint involving $s$. 
Other than that, we use parameters similar to \citet{See2010} and \citet{Bertsimas2023}. 
We set the replenishment costs to $c_0=0.01$, the adjustment costs to $c_t = 0.1$, the inventory holding costs to $h_t=0.04$, the backlogging costs to $b_t = 0.2$ for $t \in [T-1]$ and $b_T=2$, and the adjustment limits to $[\underline{x}_t, \overline{x}_t]=[-200, 200]$ for all $t\in[T]$. 
Further, we fix the initial inventory to $I_0(\cdot)=0$. 
We assume that the demands $\suvec$ follow the autoregressive process $\suscal_t = \mu + \nu\big(\tilde{\uscalbase}_t + \alpha\sum_{t'\in[t-1]} \tilde{\uscalbase}_{t'}\big)$, $t\in[T]$, where $\mu = 200$, $\nu = \mu / \sqrt{T}$, $\alpha\in\{0, 0.25, 0.5\}$ and the constituent uncertainties $\tilde{\uvecbase}$ are uniformly distributed over the unit 2-norm ball $\{ \uvecbase \, \colon \, \norm{\uvecbase}_2 \leq 1 \}$, restricted to satisfy the limits $\suvec_t\in[0, 2\mu]$. 
This demand structure reflects independent demand shocks $\tilde{\uscalbase}_t$ that have a permanent impact on future demands, and it is commonly used in the literature to capture long-term impacts of disruptions \citep{Muth1960, Graves1999, See2010, Bertsimas2023}. 
We assume a service level of $s = 0.2 \nu T$ and use the identity embedding $\embed=\unitytary$. 
Thus, $\UE = \U$ only satisfies Assumption~\ref{ass:symmetric_concave} when $\alpha = 0$, allowing us to investigate the performance of $\UFOV{}^\mathrm{A}$ when Assumption~\ref{ass:symmetric_concave} is violated. 
Figure~\ref{fig:experiment:uncertainty} visualizes the retraction of our new cuts and the autoregressive support $\U$ for different values of $\alpha$. 
Note that we do not display the retraction of the outer approximations $\UFO$, as they always retract to $\U$ by construction (i.e., $\retr(\UFO) = \U$) and would thus make the visualization uninformative. As $\alpha$ increases, the figure shows that our cuts become progressively weaker. 

\begin{figure}
    \centering
\input{figures/test_support}
    \caption{Visualization of the autoregressive support $\U$ (green, ellipsis) and the retracted cuts $\retr(\{ 
\ufvec \in \FBOX \, \colon \,
\fdist(\ufvec, [\mu\unitvec, \mu\unitvec]) 
\leq 
\maxdist ([\mu\unitvec, \mu\unitvec]) 
\})$ (blue) for different autoregressive factors $\alpha$ with a single breakpoint $\mu$ in two dimensions.}
    \label{fig:experiment:uncertainty}
\end{figure}
    
In our experiments we benchmark our piecewise affine policies with one breakpoint at $\mu$ in each dimension (LIFT1), our piecewise affine policies with three equidistant breakpoints $\{\mu/2,\mu,3\mu/2\}$ (LIFT3), our piecewise affine policies with three breakpoints $\{\mu+\eta(\lfloor\numue/2\rfloor)-\eta(\lfloor\numue/2\rfloor+1), \mu, \mu+\eta(\lfloor\numue/2\rfloor+1)-\eta(\lfloor\numue/2\rfloor)\}$ (LIFT3') as well as our piecewise affine policies with all $2T-1$ breakpoints as constructed in~\eqref{eq:full_break_points} (LIFTF) against their counterparts GLIFT1, GLIFT3, GLIFT3', and GLIFTF proposed by \cite{Georghiou2015} as well as classical affine policies (AFF). 
For the cut-based policies LIFT1, LIFT3, LIFT3', and LIFTF, we only use square cuts $[z_j\unitvec, -z_j\unitvec]$ for all $j\in[\numuef/2]$. 
These cuts are sufficient to achieve significant improvements. 
Appendix~\ref{appendix:exp:dyn-cut} demonstrates that dynamically adding more cuts only yields marginal benefits at the cost of significant increases in solution times. 

We compare all policies to dynamic programming benchmarks obtained by backward induction over a discretized grid-based state space combined with a grid search over the pre-committed replenishment quantity $y$. 
The discretized state comprises the 
inventory level $I_t(\uvec^t)$, 
the accumulated backlogged demand $\sum_{t'\in[t]} (-I_{t'}(\uvec^{t'}))_+$, 
the summed constituent uncertainty $\sum_{t'\in[t]} \uscalbase_{t'}$ 
and its squared 2-norm $\lvert\lvert\uscalbase^{t}\rvert\rvert^2_2$ for each $t\in[T]$. 
Together, the summed constituent uncertainty and its squared 2-norm uniquely determine the conditional distribution of the remaining uncertainties. 
We discretize the inventory level and backlogged demand by $\delta = 2.5$, the summed constituent uncertainty by $\delta/\nu$ and its squared 2-norm by $(\delta/\nu)^2$. 
This resolution corresponds to an effective demand discretization of approximately $1\%$ of the expected value $\mathbb{E}_{\PU}(\suscal_t) = 200$, and numerical tests, included in the accompanying repository, indicate that further reducing $\delta$ has a negligible impact on the results.
All experiments were run on an AMD EPYC 9334 2.70 GHz processor using Gurobi 11 with a single core per instance.

\begin{figure}[tb]
    \centering
    \includegraphics[width=.8\textwidth]{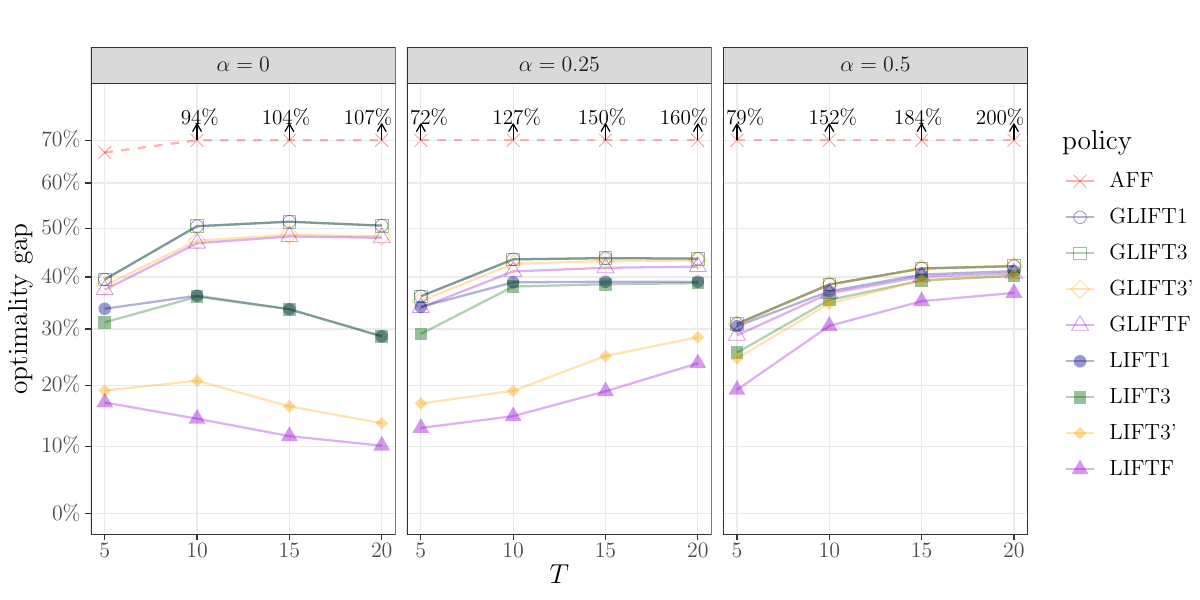}
    \caption{\textbf{Stochastic Programming.} Relative objective improvements of different piecewise affine policies over affine policies for different serial demand correlations $\alpha$ and time horizons $T$. 
    Values of AFF exceeding a $70\%$ optimality gap are not plotted at scale to ensure the visibility of relevant effects from piecewise affine policies.}
    \label{fig:inventory_service_so_performance}
\end{figure}

\begin{table}[tb]
\centering
\begin{tabular}{rrrrrrrrrr}
  \hline
T & AFF & GLIFT1 & GLIFT3 & GLIFT3' & GLIFTF & LIFT1 & LIFT3 & LIFT3' & LIFTF \\ 
  \hline
  5 & 0.00 & 0.01 & 0.01 & 0.02 & 0.04 & 0.01 & 0.01 & 0.02 & 0.08 \\ 
   10 & 0.02 & 0.04 & 0.09 & 0.10 & 3.34 & 0.04 & 0.09 & 0.08 & 6.30 \\ 
   15 & 0.08 & 0.18 & 0.54 & 0.69 & 34.17 & 0.12 & 0.52 & 0.61 & 81.12 \\ 
   20 & 0.24 & 0.47 & 1.20 & 1.44 & 192.49 & 0.31 & 1.18 & 1.59 & 384.49 \\ 
   \hline
\end{tabular}
\caption{\textbf{Stochastic Programming.} Average solution times (in seconds) for the different policies from Figure~\ref{fig:inventory_service_so_performance}.}
    \label{tab:inventory_service_so_runtime}
\end{table}

\paragraph{Stochastic Programming}

Figure~\ref{fig:inventory_service_so_performance} compares the performance of the different policies against the dynamic programming solutions. 
Piecewise affine policies reduce the optimality gap by approximately 50\% relative to their affine counterparts on symmetric instances with $\alpha = 0$, with the cuts proposed in this paper halving the remaining gap once more.
For instances violating Assumption~\ref{ass:symmetric_concave}, the improvements of the new cuts are smaller but remain significant, still almost halving the gap for $\alpha=0.25$ and reducing the gap by around $10\%$ for $\alpha=0.5$.
Table~\ref{tab:inventory_service_so_runtime} shows that these improvements need to be balanced against increases in computation times. 
Interestingly, LIFT3' consistently outperforms LIFT3, without significant increases in computational time, indicating the importance of the breakpoint selection. 
In particular, breakpoints should separate the support into subdomains with significant probability mass where the optimal policy exhibits significant non-linearities. 
In LIFT3' and LIFTF, we separate the support into subdomains based on geometry. 
Identifying subdomains where optimal affine policies change requires domain knowledge or extensive numerical testing. 
In our numerical experiments, for example, we found that placing a breakpoint at the mean demand $\mu$ was beneficial for all tested piecewise affine policies. 
Moreover, LIFT3' achieves similar results as LIFTF while incurring a fraction of the solution times. 
We thus conclude that the additional solution time from constructing breakpoints according to \eqref{eq:full_break_points} may not be worthwhile in practice.

\paragraph{Robust Optimization}

\begin{figure}[tb]
    \centering
    \includegraphics[width=.8\textwidth]{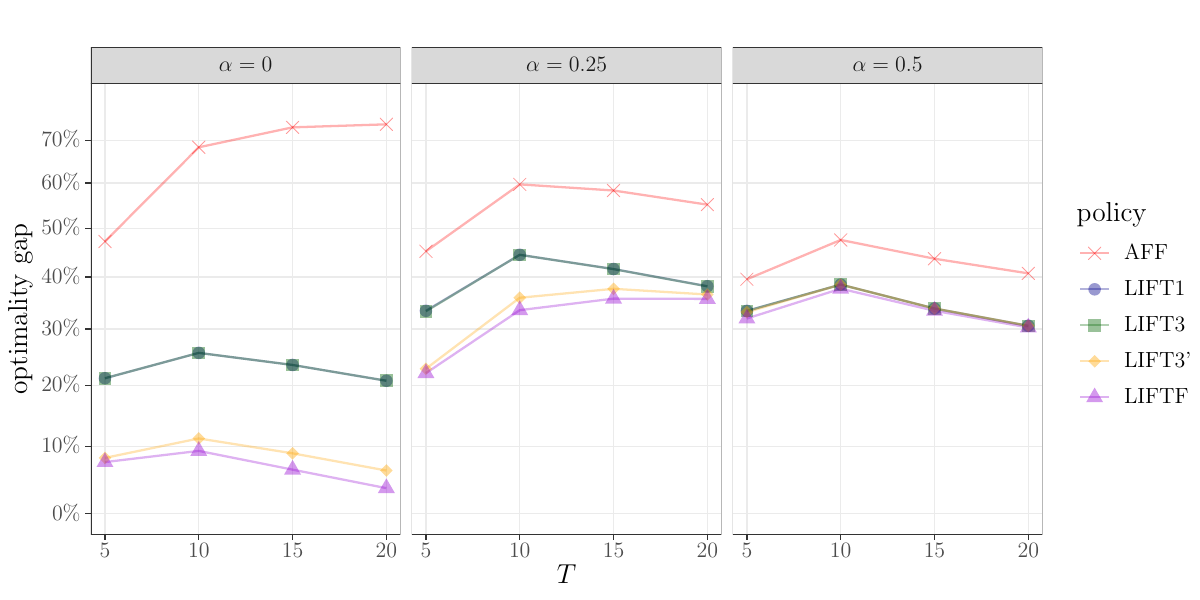}
    \caption{\textbf{Robust Optimization.} Relative objective improvements of different piecewise affine policies over affine policies.
    }
    \label{fig:inventory_service_ro_performance}
\end{figure}

We continue to use the same autoregressive demand process, but we now consider a worst-case objective and require all constraints to be satisfied for every possible demand realization. Figure~\ref{fig:inventory_service_ro_performance} reports the optimality gap improvements of our piecewise affine policies; we omit solution times, which follow a similar pattern as those for the stochastic case. 
We again observe significant improvements of our piecewise affine policies over the affine policies, which amount to an almost ten-fold reduction of the optimality gap for larger instances with a low serial demand correlation. 
Again, these improvements significantly decrease for instances violating Assumption~\ref{ass:symmetric_concave}. 
Note that the piecewise affine policies of \cite{Georghiou2015} coincide with the affine policies in this experiment and are thus omitted.

\paragraph{Data-Driven Optimization} 

\begin{figure}[tb]
    \centering
    \includegraphics[width=.8\textwidth]{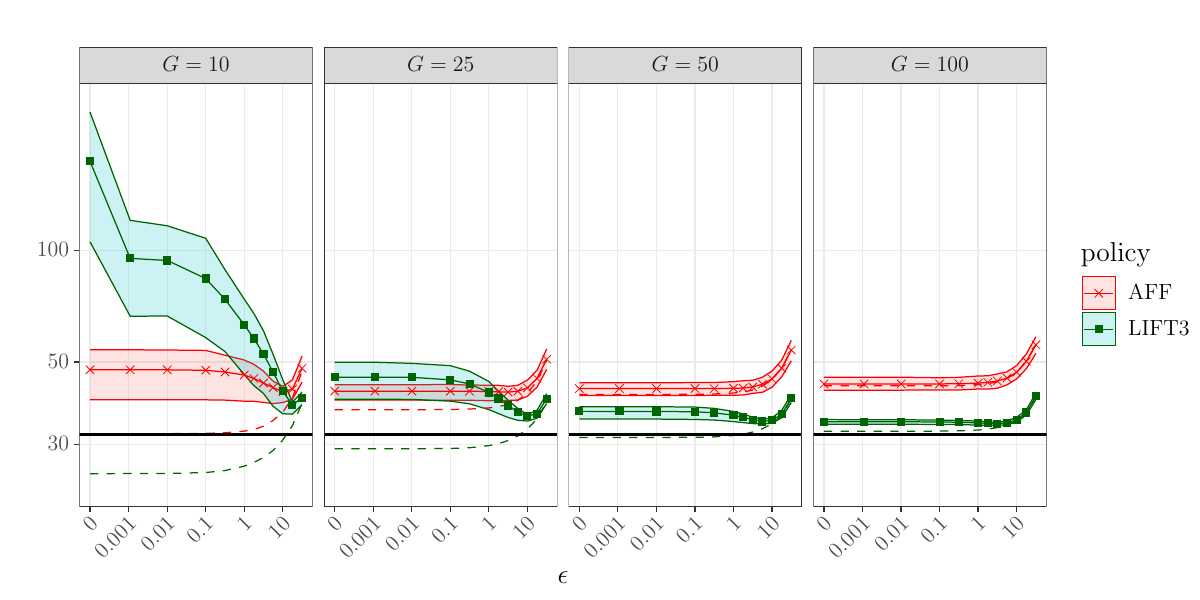}
    \caption{\textbf{Data-Driven Optimization.} Out-of-sample performance of affine and our piecewise affine policies for different radii $\epsilon$. The shaded regions represent the $(20\%, 80\%)$-confidence regions, the dashed lines report average performances on the training data, and the solid black line reports the dynamic programming solution.}
    \label{fig:inventory_dd_performance}
\end{figure}

\begin{figure}[tb]
    \centering
    \includegraphics[width=.8\textwidth]{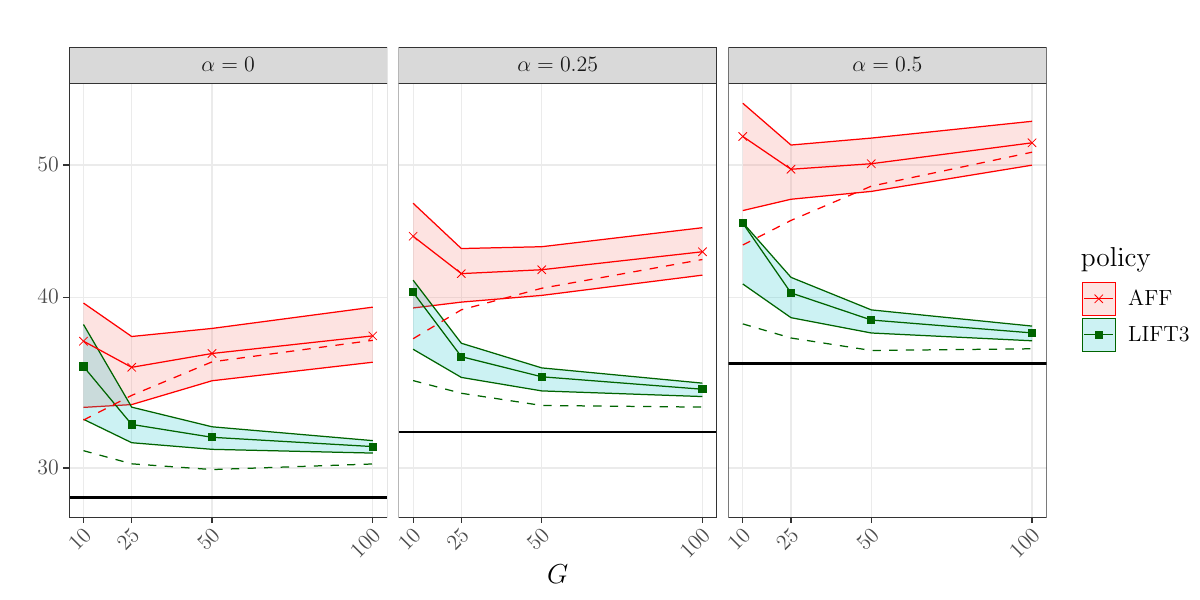}
    \caption{\textbf{Data-Driven Optimization.} Out-of-sample performance of affine and our piecewise affine policies where the radii $\epsilon$ are selected via 5-fold cross-validation. The shaded regions, dashed lines, and solid black line have the same interpretation as in Figure~\ref{fig:inventory_dd_performance}.}
    \label{fig:inventory_cv_performance}
\end{figure}

\begin{table}[tb]
\centering
\begin{tabular}{rrlrl}
  \hline
G & \multicolumn{2}{c}{AFF} & \multicolumn{2}{c}{LIFT3} \\ 
  \hline
 10 & 3.84 & 2.8\% & 38.18 & 3.2\% \\ 
   25 & 20.82 & 0.7\% & 217.31 & 1.1\% \\ 
   50 & 59.81 & 0.3\% & 693.90 & 0.8\% \\ 
  100 & 221.57 & 0.0\% & 2,684.33 & 0.4\% \\ 
   \hline
\end{tabular}
\caption{\textbf{Data-Driven Optimization.} Average solution times (in seconds) and percentage of service level violations for the different policies from Figure~\ref{fig:inventory_cv_performance}.}
\label{tab:inventory_dd_runtime}
\end{table}

For the data-driven framework, we generate training sample paths $\hat{\uvec}_1,\dots,\hat{\uvec}_G$ from the demand distribution and use the rectangular support sets $\U_g = \{\uvec \, \colon \, \lVert \uvec - \hat{\uvec}_g \rVert_\infty \leq \epsilon \}$ around each sample path $\hat{\uvec}_g$ with radii $\epsilon \in \{ 0 \} \cup \{ 10^x \, \colon \, x = -3, -2,\allowbreak -1,\allowbreak -1/2,\allowbreak 0, 1/4, 1/2, 3/4, 1, 5/4, 3/2 \}$. 
Our policy construction assumes no prior knowledge of the demand distribution. In particular, we choose the three equidistant breakpoints $\{x(\overline{\uvec}-\underline{\uvec})/4\, \colon \, x = 1, 2, 3\}$ for LIFT3 within the bounding box of the empirical domain $\U=\bigcup_{g\in[G]}\U_g$. 
Figure~\ref{fig:inventory_dd_performance} reports the out-of-sample performance of affine and piecewise affine policies over $10,{}000$ uncertainty realizations and $200$ randomly generated instances with demand correlation $\alpha=0.25$ for different training set sizes $G$ and radii $\epsilon$, and Figure~\ref{fig:inventory_cv_performance} reports the out-of-sample performances for different training set sizes $G$ and serial demand correlations $\alpha$ if the radii $\epsilon$ in our piecewise affine policies are selected via 5-fold cross-validation. 
While our hyperparameter tuning is purely based on the objective values, Table~\ref{tab:inventory_dd_runtime} shows that both policies fulfill the service level constraint with high probability. 
Both figures demonstrate that we can expect significant performance improvements when we employ our piecewise affine policies instead of affine ones. 
Table~\ref{tab:inventory_dd_runtime} shows that these improvements again have to be balanced against increases in computation times, which increase by around a factor of $10$. 
Both policies' training times increase sub-quadratically with the size $G$ of the training set. 
The effect of $\alpha$ on the training times is negligible and thus omitted.

{
\singlespacing
\bibliography{bibliography}
}
\clearpage

\section*{E-Companion}

\appendix

\renewcommand{\thesection}{EC.\arabic{section}}

\section{Proof of Lemma \ref{lem:rectangular_hull}}
\label{appendix:proof:lem:rectangular_hull}
\input{proofs/proof_pro_rectangular_hull}

\section{Proof of Theorem \ref{thm:negative_SP}}
\label{appendix:proof:thm:negative_SP}
\input{proofs/proof_thm_no_improvements}

\section{Proof of Observation~\ref{obs:distance_formalization}}
\label{appendix:proof:obs:distance_formalization}
\input{proofs/proof_pro_distance_formalization}

\section{Proof of Observation~\ref{obs:conic-reformulation}}
\label{appendix:proof:obs:conic-reformulation}
\input{proofs/proof_pro_counterpart}

\section{Proof of Proposition~\ref{pro:np_hard}}
\label{appendix:proof:pro:np_hard}
\input{proofs/proof_pro_np_hard}

\section{Proof of Theorem~\ref{thm:symmetric}}
\label{appendix:proof:thm:symmetric}
\input{proofs/proof_pro_symmetric}

\section{Proof of Proposition~\ref{pro:pballs}}
\label{appendix:proof:pro:pballs}
\input{proofs/proof_pro_pballs}

\section{Proof of Proposition~\ref{pro:intersections}}
\label{appendix:proof:pro:intersections}
\input{proofs/proof_pro_intersections}

\section{Proof of Proposition~\ref{pro:bounding_by_symmetric}}
\label{appendix:proof:pro:bounding_by_symmetric}
\input{proofs/proof_pro_bounding_by_symmetry}

\section{Proof of Proposition~\ref{pro:bounds}}
\label{appendix:proof:pro:bounds}
\input{proofs/proof_pro_bounds}

\section{Dynamic Cut Generation}
\label{appendix:exp:dyn-cut}
The experiments in Section~\ref{sec:numerical} do not use the full dynamic cut generation procedure described in Section~\ref{sec:solving_subproblem}; 
instead, they limit our policies to square cuts $[z_j\unitvec, -z_j\unitvec]$ for all $j\in[\numuef/2]$. 
Figure~\ref{fig:inventory_service_cg_performance} shows that dynamically adding more cuts when computing these policies may improve objective values by up to $5\%$ for small instances and by up to $2\%$ for medium-sized instances.
However, Table~\ref{tab:inventory_service_cg_runtime} indicates that these improvements come at the cost of significant increases in solution times. 

\begin{figure}[htb]
    \centering
    \includegraphics[width=.8\textwidth]{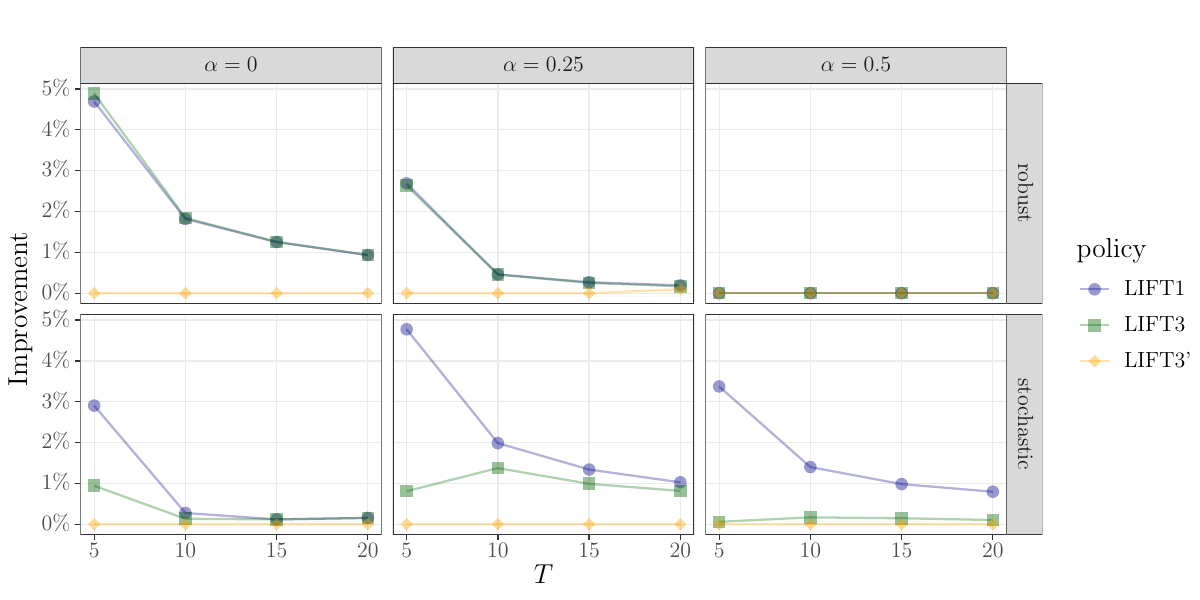}
    \caption{Relative objective improvements of different piecewise affine policies with dynamic cut generation over their respective counterparts as used in Section~\ref{sec:numerical} for different serial demand correlations $\alpha$ and time horizons $T$.
    }
    \label{fig:inventory_service_cg_performance}
\end{figure}

\begin{table}[htb]
\centering
\begin{tabular}{rrrrrrr}
  \hline
T & LIFT1 & LIFT1 CG & LIFT3 & LIFT3 CG & LIFT3' & LIFT3' CG \\ 
  \hline
  5 & 0.01 & 0.11 & 0.01 & 0.14 & 0.02 & 0.03 \\ 
   10 & 0.04 & 1.20 & 0.09 & 4.18 & 0.08 & 3.43 \\ 
   15 & 0.12 & 6.82 & 0.52 & 23.39 & 0.61 & 28.32 \\ 
   20 & 0.31 & 18.61 & 1.18 & 58.08 & 1.59 & 62.44 \\ 
   \hline
\end{tabular}
\caption{Average solution times (in seconds) for the different policies from Figure~\ref{fig:inventory_service_cg_performance} without and with dynamic cut generation (CG).}
    \label{tab:inventory_service_cg_runtime}
\end{table}

\end{document}

%% file: makros.tex
\DeclareMathOperator*{\st}{subject\;to}

\DeclareMathOperator{\pas}{\textnormal{-a.s.}}
\DeclareMathOperator*{\minimize}{minimize}
\DeclareMathOperator*{\maximize}{maximize}
\DeclareMathOperator{\case}{\textnormal{if }}
\DeclareMathOperator{\otherwise}{\textnormal{otherwise}}

\DeclareMathOperator{\argmax}{argmax}
\DeclareMathOperator{\convex}{conv}

\newcommand{\norm}[1]{\left\lVert#1\right\rVert}

\newcommand{\numc}{K} 
\newcommand{\numu}{N} 
\newcommand{\numue}{I} 
\newcommand{\numuef}{J} 
\newcommand{\numuf}{{N'}} 
\newcommand{\nums}{G} 

\newcommand{\baseproblem}{\mathcal{P}}

\newcommand{\lbaseproblem}{\baseproblem'}

\newcommand{\aff}{\mathrm{Aff\,}}

\newcommand{\lproblem}{\mathcal{P}'}

\newcommand{\PU}{\mathbb{P}}
\newcommand{\PUF}{{\mathbb{P}'}}
\newcommand{\PUFO}{\overline{\mathbb{P}}'}

\newcommand{\U}{\Xi} 
\newcommand{\UE}{\Theta} 
\newcommand{\UF}{\Xi'} 
\newcommand{\UFO}{\overline{\Xi}{}^\star} 
\newcommand{\UFOG}{\overline{\Xi}{}'} 
\newcommand{\UFOV}{\overline{\Xi}} 
\newcommand{\EBOX}{[\underline{\uevec}, \overline{\uevec}]}
\newcommand{\FBOX}{\convex(\fold(\EBOX))}
\newcommand{\EBOXg}{[\underline{\uevec}_g, \overline{\uevec}_g]}

\newcommand{\FBOXg}{\convex(\fold(\EBOXg))}

\newcommand{\uvec}{\boldsymbol{\uscal}} 
\newcommand{\uscal}{\xi} 
\newcommand{\altuvec}{\hat{\uvec}} 
\newcommand{\uevec}{\boldsymbol{\uescal}} 
\newcommand{\uescal}{\theta} 
\newcommand{\ufvec}{\boldsymbol{\xi}'} 
\newcommand{\ufscal}{\xi'} 
\newcommand{\altufvec}{\boldsymbol{\xi}^\star} 
\newcommand{\altufvecT}{\boldsymbol{\xi}^{\star\intercal}} 
\newcommand{\altufscal}{\xi^\star} 
\newcommand{\altaltufvec}{\boldsymbol{\xi}''} 
\newcommand{\altaltufscal}{\xi''} 
\newcommand{\suvec}{\tilde{\uvec}} 
\newcommand{\suscal}{\tilde{\uscal}} 
\newcommand{\sufvec}{\tilde{\uvec}'} 

\newcommand{\upvec}{\boldsymbol{\xi}'} 

\newcommand{\uvecbase}{\boldsymbol{\phi}} 
\newcommand{\uscalbase}{\phi} 

\newcommand{\embed}{\boldsymbol{E}} 
\newcommand{\fold}{F} 
\newcommand{\flatten}{F^+} 
\newcommand{\lift}{L} 
\newcommand{\retr}{R} 

\newcommand{\dist}{d}
\newcommand{\fdist}{d'}
\newcommand{\maxdist}{\overline{\dist}}

\newcommand{\X}{\mathcal{F}}
\newcommand{\XF}{\X'}

\newcommand{\avec}{\boldsymbol{a}}
\newcommand{\bvec}{\boldsymbol{b}}
\newcommand{\cvec}{\boldsymbol{c}}
\newcommand{\dvec}{\boldsymbol{d}}

\newcommand{\svec}{\boldsymbol{s}}
\newcommand{\vvec}{\boldsymbol{v}}
\newcommand{\xvec}{\boldsymbol{x}}
\newcommand{\xfvec}{\boldsymbol{x}'}
\newcommand{\yvec}{\boldsymbol{y}}
\newcommand{\zvec}{\boldsymbol{z}}

\newcommand{\amat}{\boldsymbol{A}}
\newcommand{\bmat}{\boldsymbol{B}}

\newcommand{\smat}{\boldsymbol{S}}

\newcommand{\vmat}{\boldsymbol{V}}
\newcommand{\wmat}{\boldsymbol{W}}
\newcommand{\xmat}{\boldsymbol{X}}
\newcommand{\ymat}{\boldsymbol{Y}}
\newcommand{\zmat}{\boldsymbol{Z}}

\newcommand{\gammavec}{\boldsymbol{\gamma}}
\newcommand{\deltavec}{\boldsymbol{\delta}}

\newcommand{\unitvec}{\boldsymbol{e}}
\newcommand{\nullvec}{\boldsymbol{0}}
\newcommand{\unitytary}{\boldsymbol{1}}

%% file: figures/folding_visualization_staggered.tex
\begin{tikzpicture}
\color{black}
\begin{scope}[shift={(0,0)}, scale=.7]

    \coordinate (e-origin) at (-.5, -1.3);
    \coordinate (e-end) at (8, -1.3);

    \coordinate (e-z0) at (0, -1.3);
    \coordinate (e-z1) at (2.5, -1.3);
    \coordinate (e-z2) at (5, -1.3);
    \coordinate (e-z3) at (7.5, -1.3);

    \coordinate (e-vec) at (6.5, -1.3);
    \coordinate (f-scal1) at (6.5, 0.2);
    \coordinate (f-scal2) at (6.5, 1.9);
    \coordinate (f-scal3) at (6.5, 3);

    \draw[thick, ->] (e-origin) -- (e-end) node[right] {$\uescal_{i}$};

    \draw (e-z0) ++(0,.2) -- ++(0,-.4) node[below] {$z_{i0}$};
    \draw (e-z1) ++(0,.2) -- ++(0,-.4) node[below] {$z_{i1}$};
    \draw (e-z2) ++(0,.2) -- ++(0,-.4) node[below] {$z_{i2}$};
    \draw (e-z3) ++(0,.2) -- ++(0,-.4) node[below] {$z_{i3}$};
    \draw (e-vec) ++(0,.2) -- ++(0,-.4) node[below] {$\uescal_{i}$};

    \draw[fill=black] (e-vec) circle (.08);

    \foreach \x in {1,...,3} {
    \begin{scope}[shift={(0,1.7*\x-3)}]
    \draw[thick] (-.5,0) -- (2.5*\x-2.5,0) -- (2.5*\x,1.5) -- (8,1.5);
    \draw[->] (-.5,0) -- (-.5,1.6) node[below left] {$\ufscal_{i\x}$};
    \draw[thick, ->] (-.5,0) -- (8,0);
    \end{scope}
    }
    \foreach \x in {0,...,3} {
    \draw[dashed, opacity=.3] (e-z\x) --++ (0,5);
    }
    \draw[dotted] (e-vec) -- (f-scal3);
    \foreach \x in {1,...,3} {
    \draw[fill=black] (f-scal\x) circle (.08);
    \draw (f-scal\x) node[below] {\quad$\fold_{i\x}(\uescal_i)$};
    }
\end{scope}
\end{tikzpicture}

%% file: figures/folding_visualization_box.tex
\begin{tikzpicture}
\color{black}
\begin{scope}[shift={(8,0)}, scale=.7]

    \coordinate (f-origin) at (0,0,0);
    \coordinate (f-u1-end) at (3.2,0,0);
    \coordinate (f-u2-end) at (0,3.2,0);
    \coordinate (f-u3-end) at (0,0,-3.2);
    
    \coordinate (f-z0) at (f-origin);
    \coordinate (f-z1) at (2.5,0,0);
    \coordinate (f-z2) at (2.5,2.5,0);
    \coordinate (f-z3) at (2.5, 2.5,-2.5);

    \coordinate (f-z2-1) at (0,2.5,0);
    \coordinate (f-z3-1) at (0,2.5,-2.5);
    \coordinate (f-z3-2) at (2.5,0,-2.5);
    \coordinate (f-z3-12) at (0,0,-2.5);

    \coordinate (f-vec) at (2.5,2.5,-1.5);

    \coordinate (e-origin) at (-.5, -1.3);
    \coordinate (e-end) at (8, -1.3);

    \coordinate (e-z0) at (0, -1.3);
    \coordinate (e-z1) at (2.5, -1.3);
    \coordinate (e-z2) at (5, -1.3);
    \coordinate (e-z3) at (7.5, -1.3);

    \coordinate (e-vec) at (6.5, -1.3);

    \draw[thick, ->] (e-origin) -- (e-end) node[right] {$\uescal_{i}$};
 
    \draw[thick,->] (f-origin) -- (f-u1-end) node[anchor=north] {$\ufscal_{i1}$};
    \draw[thick,->] (f-origin) -- (f-u2-end) node[anchor=north east] {$\ufscal_{i2}$};
    \draw[thick,->] (f-origin) -- (f-u3-end) node[anchor=south ] {$\ufscal_{i3}$};

    \draw (e-z0) ++(0,.2) -- ++(0,-.4) node[below] {$z_{i0}$};
    \draw (e-z1) ++(0,.2) -- ++(0,-.4) node[below] {$z_{i1}$};
    \draw (e-z2) ++(0,.2) -- ++(0,-.4) node[below] {$z_{i2}$};
    \draw (e-z3) ++(0,.2) -- ++(0,-.4) node[below] {$z_{i3}$};
    \draw (e-vec) ++(0,.2) -- ++(0,-.4) node[below] {$\uescal_{i}$};

    \draw[line width=4, draw opacity=.3] (e-z0) -> (e-vec);
    \draw[fill=black] (e-vec) circle (.08);
    
    \draw[dashed, opacity=.3] (f-z1) -- (f-z2) -- (f-z2-1);
    \draw[dashed, opacity=.3] (f-z3-1) -- (f-z3) -- (f-z3-2) -- (f-z3-12) -- (f-z3-1);
    \draw[dashed, opacity=.3] (f-z2) -- (f-z3);
    \draw[dashed, opacity=.3] (f-z1) -- (f-z3-2);
    \draw[dashed, opacity=.3] (f-z2-1) -- (f-z3-1);

    \draw[line width=4, draw opacity=.3] (f-z0) -- (f-z1) -- (f-z2) -- (f-vec);
    \draw[fill=black] (f-vec) circle (.08);

    \draw[dotted, ->, opacity=.5, thick, shorten >=3pt] (e-z0) to node[midway, left, opacity=1] {$\fold_i(z_{i0})$} (f-z0);
    \draw[dotted, ->, opacity=.5, thick, shorten >=3pt] (e-z1) to node[midway, left, opacity=1] {$\fold_i(z_{i1})$} (f-z1);
    \draw[dotted, ->, opacity=.5, thick, shorten >=3pt] (e-z2) to[bend right] node[pos=.4, left, opacity=1] {$\fold_i(z_{i2})$} (f-z2);
    \draw[dotted, ->, opacity=.5, thick, shorten >=3pt] (e-z3) to[bend right] node[pos=.65, right, opacity=1] {$\fold_i(z_{i3})$} (f-z3);
    
    \draw[dotted, ->, opacity=.5, thick, shorten >=3pt] (e-vec) to[bend right] node[pos=.5, left, opacity=1] {$\fold_i(\uescal_i)$} (f-vec);
\end{scope}
\end{tikzpicture}

%% file: figures/test_support.tex
\begin{tikzpicture}
\color{black}

\newcommand{\makebase}[1]{
\coordinate (z1) at (2,-.2);
\draw[thin, dashed, opacity=.5] (z1) -- ++(0,4.1);
\draw (z1) node[below]{$\mu$};

\coordinate (z2) at (-.2,2);
\draw[thin, dashed, opacity=.5] (z2) -- ++(4,0);
\draw (z2) node[left]{$\mu$};

\fill [black, opacity=.5] (2,2) circle (.05);

\draw[thick, ->] (-0.3,0) -- (4,0) node[below] {$\uscal_1$};
\draw[thick, ->] (0,-.3) -- (0,4) node[above] {$\uscal_2$};

\draw (2, 4.6) node {$\alpha = #1$};
}

\begin{scope}[shift={(0,0)}]
\makebase{0}

\filldraw[thick, blue, fill opacity=.02] (3.5, 1.38) -- (2.62, 0.5) -- (1.38, 0.5) -- (0.5, 1.38) -- (0.5, 2.62) -- (1.38, 3.5) -- (2.62, 3.5) -- (3.5, 2.62) -- cycle;
\begin{scope}[shift={(2,2)}]
  \filldraw[thick,green,rotate=-45,fill opacity=.15] (0,0) ellipse [x radius=1.5, y radius=1.5];
\end{scope}
\end{scope}

\begin{scope}[shift={(5,0)}]
\makebase{0.25}

\filldraw[thick, blue, fill opacity=.02] (3.5, 1.1) -- (2.86, 0.45) -- (1.14, 0.45) -- (0.5, 1.1) -- (0.5, 2.9) -- (1.14, 3.55) -- (2.86, 3.55) -- (3.5, 2.9) -- cycle;
\begin{scope}[shift={(2,2)}]
  \filldraw[thick,green,rotate=-41.4,fill opacity=.15] (0,0) ellipse [x radius=1.32, y radius=1.7];
\end{scope}
\end{scope}

\begin{scope}[shift={(10,0)}]
\makebase{0.5}

\filldraw[thick, blue, fill opacity=.02] (3.5, 0.8) -- (3.02, 0.32) -- (0.98, 0.32) -- (0.5, 0.8) -- (0.5, 3.2) -- (0.98, 3.68) -- (3.02, 3.68) -- (3.5, 3.2) -- cycle;
\begin{scope}[shift={(2,2)}]
  \filldraw[thick,green,rotate=-38,fill opacity=.15] (0,0) ellipse [x radius=1.17, y radius=1.92];
\end{scope}
\end{scope}
\end{tikzpicture}

%% file: proofs/proof_pro_rectangular_hull.tex
Lemma~\ref{lem:rectangular_hull} generalizes Lemma 4.3 of \citet{Georghiou2015}, which covers the special case where $\nums=1$ and thus $[\underline{\uevec}_1, \overline{\uevec}_1]=\EBOX$. 
The key difference in our setting is that $\EBOXg\neq\EBOX$ in general. 
In this case, the outer approximation of \citet{Georghiou2015}, which utilizes $\FBOX$ instead of $\FBOXg$, contains realizations $\ufvec$ that are not contained in the convex hull $\FBOXg$. In particular, components $\ufscal_{ij}$ corresponding to intervals $[ z_{i,j-1}, z_{ij} ]$ that are (partially) outside of the box $\EBOXg$, that is, $z_{i,j-1}<\underline{\uescal}_{gi}$ or $z_{ij}>\overline{\uescal}_{gi}$, may be smaller or larger than the respective components of any realization in the convex hull, respectively, that is, $\ufscal_{ij} < \altaltufscal_{ij}$ (when $z_{i,j-1}<\underline{\uescal}_{gi}$) or $\ufscal_{ij} > \altaltufscal_{ij}$ (when $z_{ij}>\overline{\uescal}_{gi}$) for all $\altaltufvec\in\FBOXg$.
Lemma~\ref{lem:rectangular_hull} addresses this issue by tightening the formulation for these components.

\begin{proof}[Proof of Lemma~\ref{lem:rectangular_hull}]
We prove the lemma by first showing that $\FBOXg$ decomposes by components $i\in[\numue]$ and then dealing with each component separately. 
By rectangularity of $\EBOXg$, we find 
$$
\FBOXg 
= 
\convex(\fold (\bigtimes_{i\in[\numue]} [\underline{\uescal}_{gi}, \overline{\uescal}_{gi}])) 
\overset{\textnormal{(a)}}{=} 
\convex(\bigtimes_{i\in[\numue]}\fold_i([\underline{\uescal}_{gi}, \overline{\uescal}_{gi}]))
\overset{\textnormal{(b)}}{=} 
\bigtimes_{i\in[\numue]}\convex(\fold_i([\underline{\uescal}_{gi}, \overline{\uescal}_{gi}])),
$$
where (a) follows from $\fold$ decomposing by components $\uescal_i$, and 
(b) follows from standard convexity properties of Cartesian products. 

For the remainder of the proof, we fix $i\in[\numue]$ and consider $\convex(\fold_i([\underline{\uescal}_{gi}, \overline{\uescal}_{gi}]))$. 
Our arguments are based on the intuition that $\fold_i$ maps to edges of the hyperrectangle spanned by the breakpoints visualized in Figure~\ref{fig:visualization:fold}~(b). 
Let $\hat{Z} = \{z_{ij} \,\colon\, \underline{\uescal}_{gi} \leq z_{ij} \leq \overline{\uescal}_{gi},\, j\in[\numuef_i]\} \cup \{\underline{\uescal}_{gi}, \overline{\uescal}_{gi}\}$ be the union of all breakpoints contained in $[\underline{\uescal}_{gi}, \overline{\uescal}_{gi}]$ and the bounds of the interval. 
Geometrically, $\fold_i(\hat{Z})$ coincides with the extreme points of $\fold_i([\underline{\uescal}_{gi}, \overline{\uescal}_{gi}])$, and thus $\convex(\fold_i([\underline{\uescal}_{gi}, \overline{\uescal}_{gi}])) = \convex(\fold_i(\hat{Z}))$. 
Let $H\colon [0,1]^{\numuef_i}\to \mathbb{R}^{\numuef_i}$ be the affine transformation defined component-wise via 
$
H_j(\zeta_j) = \fold_{ij}(\underline{\uescal}_i) + \zeta_j (\fold_{ij}(\overline{\uescal}_i) - \fold_{ij}(\underline{\uescal}_i))
$,
and let $\unitvec^j = \sum_{j'\in[j]}\unitvec_{j'}$ be the vector with ones in the first $j$ components and zero otherwise. 
One readily verifies that
$$
H(\unitvec^j) = \begin{cases}
    \fold_i(\overline{\uescal}_{gi}) & \case z_{ij} \geq \overline{\uescal}_{gi} \\
    \fold_i(z_{ij}) & \case \underline{\uescal}_{gi} \leq z_{ij} \leq \overline{\uescal}_{gi} \\
    \fold_i(\underline{\uescal}_{gi}) & \case z_{ij} \leq \underline{\uescal}_{gi}.
\end{cases}
$$
Thus, $\fold_i(\hat{Z}) = H(\{\unitvec^j\,\colon\, j\in [\numuef_i]\})$ and consequently
$$
\convex(\fold_i(\hat{Z}))
=
\convex(H(\{\unitvec^j\,\colon\, j\in [\numuef_i]\})) 
=
H(\convex(\{\unitvec^j\,\colon\, j\in [\numuef_i]\})),
$$
where the second identity follows from $H$ being affine. 
Note that
\begin{equation*}
\convex(\{\unitvec^j\,\colon\, j\in [\numuef_i]\}) = \{\bm{\zeta}\in[0,1]^{\numuef_i}\,\colon\,\zeta_j\geq\zeta_{j+1}\,\forall j\in [\numuef_i-1]\},
\end{equation*}
where the ``$\subseteq$" direction is immediate and the ``$\supseteq$" direction follows from $\bm{\zeta} = \sum_{j\in [\numuef_i]_{0}} (\zeta_j - \zeta_{j+1}) \unitvec^j$, where we use $\zeta_0=1$ and $\zeta_{\numuef_i+1} = 0$ for notational convenience.
Substituting $\ufvec = H(\bm{\zeta})$ and noting that $H^{-1}_j(\ufscal_j) = \frac{\ufscal_j - \fold_{ij}(\underline{\uescal}_{gi})}{\fold_{ij}(\overline{\uescal}_{gi}) - \fold_{ij}(\underline{\uescal}_{gi})}$ yields
\begin{align*}
&H(\convex(\{\unitvec^j\,\colon\, j\in [\numuef_i]\})) = 
\\&\quad\quad
\left\{
\begin{aligned}
\ufvec\in\mathbb{R}^{\numuef_i} \colon 
& \fold_{ij}(\underline{\uescal}_{gi}) \leq \ufscal_j \leq \fold_{ij}(\overline{\uescal}_{gi}) && \forall j \in [\numuef_i]\\
& \frac{\ufscal_j - \fold_{ij}(\underline{\uescal}_{gi})}{\fold_{ij}(\overline{\uescal}_{gi}) - \fold_{ij}(\underline{\uescal}_{gi})} \geq \frac{\ufscal_{j+1} - \fold_{i,j+1}(\underline{\uescal}_{gi})}{\fold_{i,j+1}(\overline{\uescal}_{gi}) - \fold_{i,j+1}(\underline{\uescal}_{gi})} && \forall j\in[\numuef_i-1]
\end{aligned}
\right\},
\end{align*}
which is the $i$\textsuperscript{th} component of the right-hand-side in the statement of the lemma.
\end{proof}

%% file: proofs/proof_thm_no_improvements.tex
We prove a slightly more general version of Theorem \ref{thm:negative_SP}, where the constraint matrices $\amat_g$ may depend on the uncertainty realizations. 
In particular, we study the problem
\begin{equation}
\tag{$\baseproblem_{\mathrm{A}} (\PU, \U)$}
\label{eq:P:general}
\begin{aligned}
    \minimize_{\xvec\in\X} \quad& \mathbb{E}_\PU \left[ \cvec(\suvec)^\intercal\xvec(\suvec)\right]\\
    \st \quad& \amat_g(\uvec) \xvec(\uvec) \leq \bvec_g(\uvec) && \forall g\in[\nums], \uvec\in \U_g
\end{aligned}  
\end{equation}
and its lifted version
\begin{equation}
\tag{$\lbaseproblem_{\mathrm{A}} (\PUF, \UF)$}
\begin{aligned}
    \minimize_{\xfvec\in\XF} \quad& \mathbb{E}_\PUF \left[ \cvec(\retr(\sufvec))^\intercal\xfvec(\sufvec)\right]\\
    \st \quad& \amat_g(\retr(\ufvec)) \xfvec(\ufvec) \leq \bvec_g(\retr(\ufvec)) && \forall g\in[\nums], \ufvec\in \UF_g.
\end{aligned}  
\end{equation}
These problems are significantly more general than the problems studied in the main part of the paper and, among others, also include robust optimization with uncertain objectives $\cvec(\uvec)$ via the epigraph reformulation from Example~\ref{ex:RP}. 
In contrast to $\baseproblem(\PU, \U)$ studied in the main paper, affine relaxations of $\baseproblem_{\mathrm{A}}(\PU, \U)$ no longer yield tractable counterparts in general \citep[Example~3.1]{BenTal2004}. 
Thus, finding optimal affine policies for these problems is generally intractable. 
From a theoretical perspective, it is nonetheless interesting to see that the limitations of lifting policies using outer approximations $\UFOG$ extend to this more general setting. 
{
\renewcommand{\thetheorem}{\ref{thm:negative_SP}$'$}
\begin{theorem}
\label{thm:negative_SP:strong}

For $\nums = 1$, there exists a subset $\overline{\Psi}{}'\subseteq\UFOG$ with $\retr(\overline{\Psi}{}') = \U$ such that every feasible solution $\xfvec$ in $\aff\lbaseproblem_{\mathrm{A}}(\PUF, \UFOG)$ has a corresponding feasible solution $\xvec$ in $\aff\baseproblem_{\mathrm{A}}(\PU, \U)$ with $\xfvec(\ufvec) = \xvec(\retr(\ufvec))$ for all $\ufvec\in\overline{\Psi}{}'$.
\end{theorem}
}

\begin{remark}
Theorem~\ref{thm:negative_SP:strong} still holds when we replace the requirement $G=1$ with the weaker requirement $\EBOXg = \EBOX$ for all $g\in [G]$.
As we are not aware of any interesting problems with $G>1$ that meet this weaker requirement, we focus on the relevant cases where $G=1$.
\end{remark}

\begin{proof}[Proof of Theorem~\ref{thm:negative_SP:strong}]
Consider the affine approximation
$$
\overline{\fold}_{ij}(\uescal_i)
\; = \;
(z_{ij} - z_{i,j-1}) \cdot \frac{\uescal_i - \underline{\uescal}_i}{\overline{\uescal}_i-\underline{\uescal}_i} \qquad \forall i \in [\numue], j \in [\numuef_i]
$$
of the folding operator $\fold$ and the associated affine approximation $\overline{\lift} = \overline{\fold}\circ\embed$ of the lifting operator $\lift$. We prove the statement of the theorem in three steps: (i) we show that $\overline{\Psi}{'} = \overline{L}(\U)$ satisfies $\overline{\Psi}{'}\subseteq \UFOG$; (ii) we verify that $\retr(\overline{\Psi}{'})=\U$; and (iii) we show that for every feasible solution $\xfvec$ in $\aff\lbaseproblem_{\mathrm{A}}(\PUF, \UFOG)$, the corresponding solution $\xvec=\xfvec\circ\overline{\lift}$ is feasible in $\aff\baseproblem_{\mathrm{A}}(\PU, \U)$ and satisfies $\xfvec(\ufvec) = \xvec(\retr(\ufvec))$ for all $\ufvec\in\overline{\Psi}{}'$. Since $G = 1$, we drop the subscript $g$ throughout the proof.

In view of Step~(i), we show that $\overline{\Psi}{'}\subseteq \UFOG$ by verifying that $\overline{\Psi}{'}$ is contained in each of the two sets on the right-hand side of~\eqref{eq:outer_general}. We have $\overline{\Psi}{'} \subseteq \{ \ufvec \, \colon \,
\fold^+ (\ufvec) \in\UE \}$ whenever $\fold^+(\overline{\Psi}{'}) \subseteq \UE$, that is, whenever $\fold^+ \circ \overline{\fold} \circ \embed \U \subseteq \UE$. To this end, we note that $\fold^+$ is a left-inverse of $\overline{\fold}$ since
\begin{equation*}
    \flatten_i(\overline{\fold}(\uevec)) 
    \; \overset{\text{(a)}}{=} \;
    \underline{\uescal}_i + \sum_{j\in[\numuef_i]} (z_{ij} - z_{i,j-1}) \cdot \frac{\uescal_i - \underline{\uescal}_i}{\overline{\uescal}_i-\underline{\uescal}_i}
    \; \overset{\text{(b)}}{=} \;
    \underline{\uescal}_i + (\overline{\uescal}_i-\underline{\uescal}_i) \cdot \frac{\uescal_i - \underline{\uescal}_i}{\overline{\uescal}_i-\underline{\uescal}_i}
    \; = \;
    \uescal_i
    \qquad \forall i\in[\numue], \uevec \in \UE,
\end{equation*}
where (a) uses the definitions of $\flatten$ and $\overline{\fold}$ and (b) contracts the sum. The set $\overline{\Psi}{'}$ is therefore contained in the first set on the right-hand side of~\eqref{eq:outer_general} whenever $\embed \U \subseteq \UE$, which holds by construction. To see that $\overline{\Psi}{'} = \overline{\fold}(\UE)$ is also contained in the second set on the right-hand side of~\eqref{eq:outer_general}, we show that any $\overline{\fold}(\uevec)$, $\uevec\in\UE$, satisfies the conditions in the representation of $\FBOX$ offered by Lemma~\ref{lem:rectangular_hull}. To see that $\fold_{ij}(\underline{\uescal}_i) \leq \overline{\fold}_{ij}(\uescal_i) \leq \fold_{ij}(\overline{\uescal}_i)$, we observe that
\begin{equation*}
    \fold_{ij}(\underline{\uescal}_i)
    \; \overset{\text{(a)}}{=} \;
    0 
    \; \overset{\text{(b)}}{\leq} \;
    (z_{ij} - z_{i,j-1}) \cdot \frac{\uescal_i - \underline{\uescal}_i}{\overline{\uescal}_i-\underline{\uescal}_i}
    \; \overset{\text{(c)}}{\leq} \;
    (z_{ij} - z_{i,j-1})
    \; \overset{\text{(d)}}{=} \;
    \fold_{ij}(\overline{\uescal}_i)
    \qquad \forall \uevec \in \UE, i\in[\numue], j\in[\numuef_i],
\end{equation*}
where (a) and (d) follow from the definition of $\fold_{ij}$ on the bounds $\underline{\uescal}_i$ and $\overline{\uescal}_i$ respectively; 
(b) and (c) follow from $z_{ij}\geq z_{i,j-1}$ and $\underline{\uescal}_i\leq \uescal_i \leq \overline{\uescal}_i$. To see that the second condition in the representation of $\FBOX$ is satisfied as well, we note that
\begin{equation*}
\begin{aligned}
    \frac{\overline{\fold}_{ij}(\uescal_i) - \fold_{ij}(\underline{\uescal}_i)}{\fold_{ij}(\overline{\uescal}_i) - \fold_{ij}(\underline{\uescal}_i)} 
    \overset{\text{(a)}}{=} 
    \frac{\overline{\fold}_{ij}(\uescal_i)}{z_{ij}-z_{i,j-1}}
    \; \overset{\text{(b)}}{=} \;
    \frac{z_{ij} - z_{i,j-1}}{z_{ij}-z_{i,j-1}} \cdot
    \frac{\uescal_i - \underline{\uescal}_i}{\overline{\uescal}_i-\underline{\uescal}_i}
    \; \overset{\text{(c)}}{=} \;
    \frac{\uescal_i - \underline{\uescal}_i}{\overline{\uescal}_i-\underline{\uescal}_i}
    \overset{\text{(d)}}{=}
    \frac{\overline{\fold}_{i,j+1}(\uescal_i) - \fold_{i,j+1}(\underline{\uescal}_i)}{\fold_{i,j+1}(\overline{\uescal}_i) - \fold_{i,j+1}(\underline{\uescal}_i)},
\end{aligned}
\end{equation*}
where (a) is due to the definition of $\fold_{ij}$ and the bounds $\underline{\uescal}_i$ and $\overline{\uescal}_i$,
(b) follows from the definition of $\overline{\fold}$, 
(c) simplifies the term,
and (d) follows by replacing $j$ with $j+1$ in our chain of equations (a)--(c). This concludes Step~(i).

As for Step~(ii), we note that $\retr = \embed^+ \circ \fold^+$ and that Step~(i) has shown that $\fold^+$ is a left-inverse of $\overline{\fold}$. Thus, $\retr$ is a left-inverse of $\overline{\lift}$, which implies that $\retr(\overline{\Psi}{'})=\retr(\overline{\lift}(\U))=\U$. 

In view of Step~(iii), finally, fix any $\xfvec$ feasible in $\aff\lbaseproblem_{\mathrm{A}}(\PUF, \UFOG)$ and set $\xvec=\xfvec\circ\overline{\lift}$. 
Let $\ufvec\in\overline{\Psi}{'}$. 
By definition of $\overline{\Psi}{}'=\overline{\lift}(\U)$, there exists some $\uvec\in \U$ such that $\ufvec = \overline{\lift}(\uvec)$. With this we find
$$
\xvec(\retr(\ufvec)) 
\overset{\text{(a)}}{=} \xfvec(\overline{\lift}(\retr(\ufvec))) 
\overset{\text{(b)}}{=} \xfvec(\overline{\lift}(\retr(\overline{\lift}(\uvec)))) 
\overset{\text{(c)}}{=} \xfvec(\overline{\lift}(\uvec)) 
\overset{\text{(d)}}{=} \xfvec(\ufvec),
$$
where (a) is due to the definition of $\xvec$, (b) and (d) follow from $\ufvec=\overline{\lift}(\uvec)$, 
and (c) holds since $\retr$ is a left-inverse of $\overline{\lift}$.
As both $\xfvec$ and $\overline{\lift}$ are affine, so is $\xvec$. 
Further, $\xvec$ is feasible in $\aff\baseproblem_{\mathrm{A}}(\PU, \U)$ since
$$
\amat(\uvec) \xvec(\uvec) 
\overset{\text{(a)}}{=} \amat(\retr(\overline{\lift}(\uvec)) \xfvec (\overline{\lift}(\uvec)) 
\overset{\text{(b)}}{\leq} \bvec(\retr(\overline{\lift}(\uvec))) 
\overset{\text{(c)}}{=} \bvec(\uvec) \qquad \forall \uvec \in \U, 
$$
where (a) is due to the definition of $\xvec$ and $\retr$ being a left-inverse of $\overline{\lift}$, 
(b) holds since $\xfvec$ is feasible in $\aff\lbaseproblem_{\mathrm{A}}(\PUF, \UFOG)$, and 
(c) follows from $\retr$ being a left-inverse of $\overline{\lift}$.
\end{proof}

%% file: proofs/proof_pro_distance_formalization.tex
\begin{proof}[Proof of Observation~\ref{obs:distance_formalization}]
We prove the inequality~\eqref{eq:the-mother-of-all-distances} by showing that for all $\ufvec\in\UFOG$, we have
\begin{align*}
\dist(\fold^+(\ufvec), [\zvec^-, \zvec^+]) 
&= \min_{\uevec'\in[\zvec^-, \zvec^+]} \norm{\fold^+(\ufvec) - \uevec'}_1 = \sum_{i\in[\numue]} \min_{\uescal_i'\in[z_i^-, z_i^+]} \lvert \fold_i^+(\ufvec) - \uescal_i'\rvert \\
&= \sum_{i\in[\numue]} (\fold_i^+(\ufvec) - z_i^+)_+ + (z_i^- - \fold_i^+(\ufvec))_+ \\
&\overset{\text{(i)}}{=} 
\sum_{i\in[\numue]}\left(\sum_{j\in \numuef^+_i}\fold_{ij}(\fold_i^+(\ufvec)) + z^-_i - \underline{\uescal}_i-\sum_{j\in \numuef^-_i}\fold_{ij}(\fold_i^+(\ufvec))\right)\\
&\overset{\text{(ii)}}{\leq} 
\sum_{i\in[\numue]}\left(\sum_{j\in \numuef^+_i}\ufscal_{ij} + z^-_i - \underline{\uescal}_i-\sum_{j\in \numuef^-_i}\ufscal_{ij}\right).
\end{align*}
Here, the first equality holds by construction of $d$, and the subsequent two equalities are immediate. In the remainder of the proof, we show that (i) and (ii) hold. 
Since $\fold_{ij}(\fold_i^+(\ufvec)) = \ufscal_{ij}$ for all $\ufvec\in\UF$, our above argument also implies that~\eqref{eq:the-mother-of-all-distances} is tight for all $\ufvec\in\UF$.

As for (i), recall the following property of $\fold$,
\begin{equation}
\label{eq:proof_distance_formalization:helper-a}
    \underline{\uescal}_i + \sum_{j' \in [j]} \fold_{ij'}(\uescal_i) = \min \{ z_{ij}, \uescal_i \} \quad\quad \forall i\in[\numue], j\in[\numuef_i], \uevec\in\UE,
\end{equation}
and note that $\fold^+(\ufvec)\in \UE$ for any $\ufvec\in\UFOG$. 
With this we find
\begin{align*}
    (\fold_i^+(\ufvec) - z_i^+)_+
    &= \fold_i^+(\ufvec) - \min \{ \fold_i^+(\ufvec), z_i^+ \}
    \overset{\text{(a)}}{=} \min\{\fold_i^+(\ufvec), \overline{\uescal}_i\} - \min \{ \fold_i^+(\ufvec), z_i^+ \}
    \\&\overset{\text{(b)}}{=} \sum_{j\in [\numuef_i]}\fold_{ij}(\fold_i^+(\ufvec)) + \underline{\uescal}_i
    - \sum_{j\in[\numuef_i], z_{ij}\leq z^+_i}\fold_{ij}(\fold_i^+(\ufvec)) - \underline{\uescal}_i
    \overset{\text{(c)}}{=} \sum_{j\in \numuef^+_i}\fold_{ij}(\fold_i^+(\ufvec)),
\end{align*}
where (a) holds since $\fold^+(\ufvec)\in \UE \subseteq \EBOX$, (b) follows from $\overline{\uescal}_i = z_{i \numuef_i}$ and applying \eqref{eq:proof_distance_formalization:helper-a} to each of the minima, and (c) is due to $\{j\in[\numuef_i]\,\colon\, z_{ij}\leq z^+_i\} = [\numuef_i]\setminus\numuef^+_i$. 
We further find 
\begin{equation*}
    (z_i^- - \fold_i^+(\ufvec))_+
    = z_i^- - \min \{ z_i^-, \fold_i^+(\ufvec) \}
    \overset{\text{(a)}}{=} z_i^- - \sum_{j\in \numuef^-_i}\fold_{ij}(\fold_i^+(\ufvec)) - \underline{\uescal}_i,
\end{equation*}
where (a) holds due to \eqref{eq:proof_distance_formalization:helper-a}.

To prove (ii), we will show that 
\begin{equation}\label{eq:part-b}
    \sum_{j \in \numuef^-_i} \ufscal_{ij} \leq
    \sum_{j \in \numuef^-_i} \fold_{ij}(\fold_i^+(\ufvec))
    \quad \text{and} \quad
    \sum_{j \in \numuef^+_i} \ufscal_{ij} \geq
    \sum_{j \in \numuef^+_i} \fold_{ij}(\fold_i^+(\ufvec))
    \qquad \forall i\in[\numue], \ufvec\in\UFOG.
\end{equation}
In view of the first inequality in~\eqref{eq:part-b}, we show the more general property
\begin{equation}
\label{eq:proof_distance_formalization:helper-b}
    \sum_{j' \in [j]} \ufscal_{ij'} \overset{\text{(a)}}{\leq} \min\{ \fold^+_i(\ufvec), z_{ij} \} - \underline{\uescal}_i \overset{\text{(b)}}{=}
    \sum_{j' \in [j]} \fold_{ij'}(\fold_i^+(\ufvec)) \qquad \forall i\in[\numue], j\in[\numuef_i], \ufvec\in\UFOG,
\end{equation}
where the identity (b) follows from \eqref{eq:proof_distance_formalization:helper-a}. 
For (a), we observe that
\begin{equation*}
    \fold^+_i(\ufvec) - \underline{\uescal}_i \overset{\text{(c)}}{=} \sum_{j \in [\numuef_i]} \ufscal_{ij} \overset{\text{(d)}}{\geq} \sum_{j' \in [j]} \ufscal_{ij'} 
    \quad \text{and} \quad
    z_{ij} - \underline{\uescal}_i
    \overset{\text{(e)}}{=} \sum_{j' \in [j]} z_{ij'}-z_{i,j'-1}
    \overset{\text{(f)}}{\geq}
    \sum_{j' \in [j]} \ufscal_{ij'},
\end{equation*}
where (c) uses the definition of $\fold^+$,
(d) holds since $\ufscal_{ij}\geq \fold_{ij}(\underline{\uescal}_i) = 0$,
(e) follows from contracting the sum, and
(f) holds since $\ufscal_{ij'} \leq \fold_{ij'}(\overline{\uescal}_i) = z_{ij'}-z_{i,j'-1}$.
The second inequality in~\eqref{eq:part-b}, finally, holds since
\begin{equation*}
    \begin{aligned}
        \sum_{j \in \numuef^+_i} \ufscal_{ij} &= 
        \sum_{j \in [\numuef_i]} \ufscal_{ij} - \sum_{j \in [\numuef_i]\setminus \numuef^+_i} \ufscal_{ij} \overset{\text{(a)}}{=} 
        \sum_{j \in [\numuef_i]} \fold_{ij}(\fold_i^+(\ufvec))
        - \sum_{j \in [\numuef_i]\setminus \numuef^+_i} \ufscal_{ij} \\&\overset{\text{(b)}}{\geq} 
        \sum_{j \in [\numuef_i]} \fold_{ij}(\fold_i^+(\ufvec)) - \sum_{j \in [\numuef_i]\setminus \numuef^+_i} \fold_{ij}(\fold_i^+(\ufvec)) = 
        \sum_{j \in \numuef^+_i} \fold_{ij}(\fold_i^+(\ufvec)),
    \end{aligned}
\end{equation*}
where (a) follows from the definition of $\fold$ and $\fold^+$ and (b) follows from \eqref{eq:proof_distance_formalization:helper-b}.
\end{proof}

%% file: proofs/proof_pro_counterpart.tex
\begin{proof}[Proof of Observation~\ref{obs:conic-reformulation}]
The proofs of \ref{obs:conic-reformulation-UFOG} and \ref{obs:conic-reformulation-UFO} largely follow the same arguments. 
In the following, we explicitly show \ref{obs:conic-reformulation-UFO} and point out which parts need to be adjusted for the proof of \ref{obs:conic-reformulation-UFOG}. 
Our proof employs the primal-worst-dual-best strong duality argument of \citet{Zhen2025}, which extends the duality argument of \citet{Beck2009} to problem classes that encompass $\aff \lproblem(\PUF, \UFO)$. 
Under this scheme, the dual of $\aff \lproblem(\PUF, \UFO)$ is

\begin{equation}
\label{eq:duality:proof:rawdual}
\begin{aligned}
	\maximize_{\ymat, \yvec, \zmat, \zvec, \ufvec, \bm{s}} \quad&
	-\sup_{\xmat, \xvec} \left(\langle\xmat,\ymat\rangle + \langle\xvec,\yvec\rangle 
	- \mathbb{E}_\PUF \left[\cvec(\retr(\sufvec))^\intercal(\xmat\sufvec + \xvec)\right]\right) \\&
    \begin{aligned}
	-\sum_{g\in[\nums], k\in[\numc_g]}\sup_{\xmat, \xvec} \Big(
	    &\langle\xmat,\zmat_{gk}\rangle + \langle\xvec,\zvec_{gk}\rangle 
	\\[-2ex]&
    - s_{gk} \left(\amat_{gk}(\xmat\ufvec_{gk}+\xvec) - b_{gk}(\retr(\ufvec_{gk}))\right)
	\Big)\end{aligned}\span\span\\
	\st\quad
	&\ymat+\sum_{g\in[\nums], k\in[\numc_g]} \zmat_{gk} = \nullvec\\
	&\yvec+\sum_{g\in[\nums], k\in[\numc_g]} \zvec_{gk} = \nullvec\\
	&\ufvec_{gk}\in\UFO_g, \quad s_{gk}\geq 0 &&\forall g\in[\nums], k\in[\numc_g],
\end{aligned}
\end{equation}
where we use the affine representation $\xfvec(\ufvec)=\xmat\ufvec + \xvec$.

By assumption, each set $\UFO_g$ is bounded, and $\aff \lproblem(\PUF, \UFO)$ affords a Slater point. 
Thus, strong duality holds by Theorem 4 (ii) of \citet{Zhen2025}. 
Together with the assumed feasibility of $\aff \lproblem(\PUF, \UFO)$, this, in particular, implies that the dual has a finite optimal value. 
The suprema in the objective function only yield finite values, when $\ymat$, $\yvec$, $\zmat$, and $\zvec$ omit certain conditions. 
In particular, for the first term of the objective function in \eqref{eq:duality:proof:rawdual}, we find
\begin{align*}
	&\sup_{\xmat, \xvec} \left(\langle\xmat,\ymat\rangle + \langle\xvec,\yvec\rangle 
	- \mathbb{E}_\PUF \left[\cvec(\retr(\sufvec))^\intercal(\xmat\sufvec + \xvec)\right]\right)
	\\\overset{\text{(a)}}{=}&
	\sup_{\xmat, \xvec} \left(\langle\xmat,\ymat-\mathbb{E}_{\PUF}[c(\retr(\sufvec))\sufvec^\intercal]\rangle + \langle\xvec,\yvec-\mathbb{E}_{\PUF}[\cvec(\retr(\sufvec))]\rangle \right)
	\\\overset{\text{(b)}}{=}&
	\begin{cases}
		0 & \case \ymat=\mathbb{E}_{\PUF}[c(\retr(\sufvec))\sufvec^\intercal] \text{ and } \yvec=\mathbb{E}_{\PUF}[\cvec(\retr(\sufvec))] \\
		\infty & \otherwise,
	\end{cases}
\end{align*}
where (a) follows from the linearity of the expectation operator, and
(b) holds since $\xmat$ and $\xvec$ are free variables.
For the second term of the objective function in \eqref{eq:duality:proof:rawdual}, we find
\begin{align*}
	&\sup_{\xmat, \xvec} \left(\langle\xmat,\zmat_{gk}\rangle + \langle\xvec,\zvec_{gk}\rangle 
	- s_{gk} \left(\amat_{gk}(\xmat\ufvec_{gk}+\xvec) - b_{gk}(\retr(\ufvec_{gk}))\right)
	\right)
	\\\overset{\text{(a)}}{=}&
	\sup_{\xmat, \xvec} \left(\langle\xmat,\zmat_{gk}-s_{gk} \amat^\intercal_{gk} \ufvec^\intercal_{gk} \rangle + \langle\xvec,\zvec_{gk}-s_{gk}\amat^\intercal_{gk}\rangle 
	+ s_{gk} b_{gk}(\retr(\ufvec_{gk}))
	\right)
	\\\overset{\text{(b)}}{=}&
	\begin{cases}
		s_{gk} b_{gk}(\retr(\ufvec_{gk})) &\case \zmat_{gk}=s_{gk} \amat^\intercal_{gk} \ufvec^\intercal_{gk} \text{ and } \zvec_{gk}=s_{gk}\amat^\intercal_{gk}\\
		\infty &\otherwise,
	\end{cases}
\end{align*}
where (a) reorders terms and 
(b) holds since $\xmat$ and $\xvec$ are free variables. 
Using this, we substitute $\ymat$, $\yvec$, $\zmat$, and $\zvec$ in~\eqref{eq:duality:proof:rawdual} to arrive at the equivalent problem 
\begin{equation}
\label{eq:duality:proof:refineddual}
\begin{aligned}
	\maximize_{\ufvec, \bm{s}} \quad 
	& - \sum_{g\in[\nums],k\in[\numc_g]} s_{gk}b_{gk}( \retr( \ufvec_{gk})) \\
	\st \quad & \sum_{g\in[\nums],k\in[\numc_g]} s_{gk}\amat_{gk}^\intercal  \ufvec^\intercal_{gk} = - \mathbb{E}_{\PUF}[\cvec(\retr(\sufvec))\sufvec^\intercal]\\
	&\sum_{g\in[\nums],k\in[\numc_g]} s_{gk} \amat_{gk}^\intercal = -\mathbb{E}_{\PUF}[\cvec(\retr(\sufvec))] && \\
	& \ufvec_{gk}\in \UFO_g, \quad s_{gk}\geq 0 &&\forall g\in[\nums], k\in[\numc_g].
\end{aligned}
\end{equation}%
Problem~\eqref{eq:duality:proof:refineddual} contains bi-linear terms in $\ufvec$ and $\bm{s}$ and is thus non-convex in general. 
Next, we eliminate these bi-linear terms via another variable substitution. 
In particular, we reformulate the problem to have only bi-linearities of the form $s_{gk}\ufvec_{gk}$. 
For the term $s_{gk}b_{gk}(\retr(\ufvec_{gk}))$ in the objective of \eqref{eq:duality:proof:refineddual}, note that $b_{gk}\circ \retr$ is an affine function as both $b_{gk}$ and $\retr$ are affine. 
Thus, $b_{gk}(\retr(\cdot))-b_{gk}(\retr(\nullvec))$ is linear, which implies that $s_{gk}b_{gk}(\retr(\ufvec_{gk})) = b_{gk}(\retr(s_{gk}\ufvec_{gk}))+(s_{gk}-1)b_{gk}(\retr(\nullvec))$ for all $\ufvec_{gk}\in \UFO_g$ and $s_{gk}\geq 0$. 
For the term $s_{gk}\amat^\intercal_{gk}\ufvec^\intercal_{gk}$ in the first constraint of \eqref{eq:duality:proof:refineddual}, we find that $s_{gk}\amat^\intercal_{gk}\ufvec^\intercal_{gk} = \amat^\intercal_{gk}(s_{gk}\ufvec_{gk})^\intercal$ by the linearity of matrix multiplications and transpositions. 
Substituting $\altufvec_{gk} = s_{gk}\ufvec_{gk}$, where $\altufvec_{gk} \in \{s_{gk}\ufvec \,\colon\, \ufvec\in \UFO_g\} = s_{gk}\UFO_g$, we find that \eqref{eq:duality:proof:refineddual} is equivalent to 
\begin{equation*}
\begin{aligned}
	\maximize_{\altufvec, \bm{s}} \quad 
	& - \sum_{g\in[\nums],k\in[\numc_g]} b_{gk}(\retr(\altufvec_{gk}))+(s_{gk}-1)b_{gk}(\retr(\nullvec)) \\
	\st \quad & \sum_{g\in[\nums],k\in[\numc_g]} \amat_{gk}^\intercal  \altufvecT_{gk} = - \mathbb{E}_{\PUF}[\cvec(\retr(\sufvec))\sufvec^\intercal]\\
	&\sum_{g\in[\nums],k\in[\numc_g]} s_{gk} \amat_{gk}^\intercal = -\mathbb{E}_{\PUF}[\cvec(\retr(\sufvec))] && \\
	& \altufvec_{gk}\in s_{gk}\UFO_g, \quad s_{gk}\geq 0 &&\forall g\in[\nums], k\in[\numc_g].
\end{aligned}
\end{equation*}
We conclude the proof by showing that the condition $\altufvec_{gk}\in s_{gk}\UFO_g$ is conic representable for conic representable sets $\U_g = \{\uvec\in\mathbb{R}^\numu\;\colon\; \vmat_g\uvec\succeq_{\mathcal{K}_g}\dvec_g\}$. 
Indeed, using the definitions of $\UFO_g$ and $\U_g$, we find that
\begin{equation}
\label{eq:duality:proof:sUFOraw}
\begin{aligned}
	s\UFO_g = \Big\{
			s\ufvec\;\colon\;& \ufvec\in\mathbb{R}^\numuf_+,\; \uvec\in\mathbb{R}^\numu,
            \\& 
			\vmat_g\uvec  \succeq_{\mathcal{K}_g} \dvec_g 
			,\;
			\flatten(\ufvec) = \embed\uvec
			\\&\fold_{ij}(\underline{\uescal}_{gi}) \leq \ufscal_{ij} \leq \fold_{ij}(\overline{\uescal}_{gi}) \;&&
				\forall i\in[\numue],j\in[\numuef_i]
			\\&\frac{\ufscal_{ij} - \fold_{ij}(\underline{\uescal}_{gi})}
			{\fold_{ij}(\overline{\uescal}_{gi}) - \fold_{ij}(\underline{\uescal}_{gi})}
			\geq
			\frac{\ufscal_{i,j+1} - \fold_{i,j+1}(\underline{\uescal}_{gi})}
			{\fold_{i,j+1}(\overline{\uescal}_{gi}) - \fold_{i,j+1}(\underline{\uescal}_{gi})}
			\; &&\begin{aligned}
				\forall	& i\in[\numue],j\in[\numuef_i] \\[-1ex]
				&\mathrm{ with }\, \underline{\uescal}_{gi} < z_{ij} < \overline{\uescal}_{gi}
			\end{aligned}
			\\&
			\fdist(\ufvec, [\zvec^-, \zvec^+]) \leq \maxdist_g ([\zvec^-, \zvec^+]) && \forall [\zvec^-, \zvec^+]\in \mathcal{Z}
		\Big\}
\end{aligned}
\end{equation}
for any $s\geq 0$. 
Scaling all conditions in \eqref{eq:duality:proof:sUFOraw} by $s$ and using the linearity of matrix multiplications, \eqref{eq:duality:proof:sUFOraw} is equivalent to  
\begin{equation}
\label{eq:duality:proof:sUFOscaled}
\begin{aligned}
	\Big\{
			s\ufvec\;\colon\;& \ufvec\in\mathbb{R}^\numuf_+,\; \uvec\in\mathbb{R}^\numu,
            \\& 
			\vmat_g (s\uvec)  \succeq_{\mathcal{K}_g} s\dvec_g 
			,\;
			s\flatten(\ufvec) = \embed (s\uvec)
			\\& s\fold_{ij}(\underline{\uescal}_{gi}) \leq s\ufscal_{ij} \leq s\fold_{ij}(\overline{\uescal}_{gi}) \;&&
				\forall i\in[\numue],j\in[\numuef_i]
			\\&\frac{s\ufscal_{ij} - s\fold_{ij}(\underline{\uescal}_{gi})}
			{\fold_{ij}(\overline{\uescal}_{gi}) - \fold_{ij}(\underline{\uescal}_{gi})}
			\geq
			\frac{s\ufscal_{i,j+1} - s\fold_{i,j+1}(\underline{\uescal}_{gi})}
			{\fold_{i,j+1}(\overline{\uescal}_{gi}) - \fold_{i,j+1}(\underline{\uescal}_{gi})}
			\; &&\begin{aligned}
				\forall	& i\in[\numue],j\in[\numuef_i] \\[-1ex]
				&\mathrm{ with }\, \underline{\uescal}_{gi} < z_{ij} < \overline{\uescal}_{gi}
			\end{aligned}
			\\&
			s\fdist(\ufvec, [\zvec^-, \zvec^+]) \leq s\maxdist_g ([\zvec^-, \zvec^+]) && \forall [\zvec^-, \zvec^+]\in \mathcal{Z}
		\Big\}.
\end{aligned}
\end{equation}
Using that $\fold^+$ is affine, we find that $s\fold^+(\ufvec) = \fold^+(s \ufvec) + (s-1)\fold^+(\nullvec)$. 
Similarly, $\fdist(\cdot, [\zvec^-, \zvec^+])$ being affine implies that $s\fdist(\ufvec, [\zvec^-, \zvec^+]) = \fdist(s\ufvec, [\zvec^-, \zvec^+]) + (s-1)\fdist(\nullvec, [\zvec^-, \zvec^+])$. 
Thus, substituting $\altuvec = s\uvec$ and $\altufvec = s\ufvec$ in \eqref{eq:duality:proof:sUFOscaled} yields the equivalent formulation 
\begin{equation}
\label{eq:duality:proof:sUFOfinal}
\begin{aligned}
	\Big\{
			\altufvec\;\colon\;& \altufvec\in\mathbb{R}^\numuf_+,\; \altuvec\in\mathbb{R}^\numu,
            \\& 
			\vmat_g \altuvec  \succeq_{\mathcal{K}_g} s\dvec_g 
			,\;
			\fold^+(\altufvec) + (s-1)\fold^+(\nullvec) = \embed \altuvec
			\\& s\fold_{ij}(\underline{\uescal}_{gi}) \leq \altufscal_{ij} \leq s\fold_{ij}(\overline{\uescal}_{gi}) \;&&
				\forall i\in[\numue],j\in[\numuef_i]
			\\&\frac{\altufscal_{ij} - s\fold_{ij}(\underline{\uescal}_{gi})}
			{\fold_{ij}(\overline{\uescal}_{gi}) - \fold_{ij}(\underline{\uescal}_{gi})}
			\geq
			\frac{\altufscal_{i,j+1} - s\fold_{i,j+1}(\underline{\uescal}_{gi})}
			{\fold_{i,j+1}(\overline{\uescal}_{gi}) - \fold_{i,j+1}(\underline{\uescal}_{gi})}
			\; &&\begin{aligned}
				\forall	& i\in[\numue],j\in[\numuef_i] \\[-1ex]
				&\mathrm{ with }\, \underline{\uescal}_{gi} < z_{ij} < \overline{\uescal}_{gi}
			\end{aligned}
			\\&
			\fdist(\altufvec, [\zvec^-, \zvec^+]) + (s-1)\fdist(\nullvec, [\zvec^-, \zvec^+]) \leq s\maxdist_g ([\zvec^-, \zvec^+]) && \forall [\zvec^-, \zvec^+]\in \mathcal{Z}
		\Big\},
\end{aligned}
\end{equation}
which is manifestly conic representable. 
Finally, note that the last condition of~\eqref{eq:duality:proof:sUFOfinal} resembles the affine cuts~\eqref{eq:counterpart_dual_cuts} since
\begin{align*}
    &
    \fdist(\altufvec, [\zvec^-, \zvec^+]) + (s-1)\fdist(\nullvec, [\zvec^-, \zvec^+])
    \\
    \overset{\text{(a)}}{=}
    &
    \sum_{i\in[\numue]} \left(
    \sum\limits_{j \in \numuef^+_i} \altufscal_{ij} 
    + z^-_i - \underline{\uescal}_i
    -\sum\limits_{j \in \numuef^-_i} \altufscal_{ij} \right)
    + (s-1) (z^-_i - \underline{\uescal}_i)
    \\
    \overset{\text{(b)}}{=}
    &
    s \sum_{i\in[\numue]} \left(
    \sum\limits_{j \in \numuef^+_i} \frac{1}{s}\altufscal_{ij} 
    + z^-_i - \underline{\uescal}_i
    -\sum\limits_{j \in \numuef^-_i} \frac{1}{s}\altufscal_{ij} \right)
    \\
    \overset{\text{(c)}}{=}
    &
    s \fdist(\altufvec / s, [\zvec^-, \zvec^+]),
\end{align*}
where (a) and (c) follow from the definition of $\fdist$ and (b) follows from simplifying the term.

The proof for \ref{obs:conic-reformulation-UFOG} follows the same line of arguments if we replace $\UFO$ with $\UFOG$ and drop the constraints in our representation of $s\UFO_g$ that pertain to the cuts $\fdist(\ufvec, [\zvec^-, \zvec^+]) \leq \maxdist_g ([\zvec^-, \zvec^+])$, $[\zvec^-, \zvec^+] \in \mathcal{Z}$.
\end{proof}

%% file: proofs/proof_pro_np_hard.tex
\begin{proof}[Proof of Proposition~\ref{pro:np_hard}]
We show strong NP-hardness of the separation problem~\eqref{eq:subproblem} via a reduction from \textsc{Minimum Vertex Cover}, which is known to be strongly NP-hard \citep{Karp1972}. 
To this end, consider an undirected connected graph $(\mathcal{V}, \mathcal{E})$ with $n$ vertices and $m$ edges, and let $\wmat\in\mathbb{R}^{n\times m}$ be the incidence matrix that satisfies $W_{ij} = 1$ if the $i$\textsuperscript{th} vertex is incident to the $j$\textsuperscript{th} edge; $W_{ij} = 0$ otherwise. The optimal value of \textsc{Minimum Vertex Cover} is then given by the optimal value of the binary linear program
\begin{equation}
\label{eq:proof_np_hard:mvc}
    \min\{\unitvec^\intercal \yvec \,\colon\,
    \wmat^\intercal \yvec \geq \unitvec
    \,,\, \yvec \in \{0,1\}^n\}.
\end{equation}
In the remainder of the proof, we define a valid instance to the separation problem~\eqref{eq:subproblem} whose optimal value coincides with that of problem~\eqref{eq:proof_np_hard:mvc}, modulo some affine transformation.

Consider the instance of problem~\ref{eq:P} with $G=1$ support set of the form $\U = \{\uvec\in\mathbb{R}_+^\numu\,\colon\, \overline{\wmat}\uvec\leq2\unitvec\}$, $\numu=m+n$ and $\overline{\wmat} = (\wmat, \unitytary)$, as well as the embedding $\embed = \unitytary$ satisfying $\numue=\numu$ and $\UE = \U$. We omit the other problem data as it is irrelevant for the separation problem~\eqref{eq:subproblem}. We choose the bounding box $\EBOX = [-2\unitvec, 2\unitvec]$, which is not tight but allows for a more intuitive proof; the extension to a tight support is tedious but straightforward. Our folding operator $F$ employs a single breakpoint at $0$ in every dimension, that is, we have
$\numuef_i=2$ and $z_{i0} = -2 < z_{i1} = 0 < z_{i2} = 2$ for all $i\in[\numue]$. We wish to separate the lifted vector $\ufvec\in\mathbb{R}^{2\numue}$ with components $\ufscal_{ij} = 1$ for all $i\in[\numue], j\in[\numuef_i]$ and the scaling factor $s=1$. Throughout the proof, we omit the indices $g$ and $k$ for better readability.

To confirm that our choice of $\U$, $\UE$, $\mathcal{Z}$, $\ufvec$, and $s$ induces a valid instance of the separation problem~\eqref{eq:subproblem}, we show that $\ufvec\in\UFOG$. By the definition of $\UFOG$ in~\eqref{eq:outer_general} and the representation of $\FBOXg$ in Lemma \ref{lem:rectangular_hull}, this holds whenever the following three conditions are satisfied:
\begin{enumerate}
    \item $\fold^+ (\ufvec) \in\UE$: This holds since $\fold^+(\ufvec)=\nullvec\in\UE$ by definition of $\fold^+$, $\ufvec$ and $\UE$.
    \item $\fold_{ij}(\underline{\uescal}_i)\leq \ufscal_{ij}\leq \fold_{ij}(\overline{\uescal}_i)$ for all $i\in[I]$ and $j\in[\numuef_i]$: This holds since $\fold_{ij}(\underline{\uescal}_i) = 0$ and $\fold_{ij}(\overline{\uescal}_i) = 2$ by construction of $Z$, as well as $\ufscal_{ij} = 1$, for all $i\in[\numue]$ and $j\in[\numuef_i]$.
    \item $\frac{\ufscal_{ij} - \fold_{ij}(\underline{\uescal}_i)} {\fold_{ij}(\overline{\uescal}_i) - \fold_{ij}(\underline{\uescal}_i)} \geq \frac{\ufscal_{i,j+1} - \fold_{i,j+1}(\underline{\uescal}_i)}{\fold_{i,j+1}(\overline{\uescal}_i) - \fold_{i,j+1}(\underline{\uescal}_i)}$ for all $i\in[I]$ and $j\in[\numuef_i]$ with $\underline{\uescal}_i<z_{ij}<\overline{\uescal}_i$: This holds since $\fold_{ij}(\underline{\uescal}_i) = 0$ and $\fold_{ij}(\overline{\uescal}_i) = 2$ by construction of $Z$, as well as $\ufscal_{ij} = 1$, for all $i\in[\numue]$ and $j\in[\numuef_i]$.
\end{enumerate}

We now prove the main claim by transforming our instance of the separation problem to the \textsc{Minimum Vertex Cover} instance through a series of equivalent problems. To this end, we first reformulate the distances $\fdist$ and $\maxdist$ in the separation problem. For $\fdist$, we find that
\begin{align*}
\fdist(\ufvec/s, [\zvec^-, \zvec^+]) 
    \overset{\text{(a)}}{=}&
    \sum_{i\in[\numue]} \left(
    \sum\limits_{j \in \numuef^+_i} \ufscal_{ij} 
    + z^-_i - \underline{\uescal}_i
    -\sum\limits_{j \in \numuef^-_i} \ufscal_{ij} \right)\\
    =&
    \sum_{i\in[\numue]} \left(
    \sum\limits_{j \in \numuef^+_i} 1 
    + z^-_i + 2
    -\sum\limits_{j \in \numuef^-_i} 1 \right)
    \overset{\text{(b)}}{=}
    \sum_{i\in[\numue]} \left(
    \frac{z_i^- -z_i^+}{2}  + 2
    \right),
\end{align*}
where (a) immediately follows from the definition of $\fdist$ and $s=1$, while (b) is due to $\lvert\numuef^+_i\rvert = 1 - \frac{z^+_i}{2}$ and $\lvert\numuef^-_i\rvert = \frac{z^-_i}{2}+1$. For $\maxdist$, we find that
\begin{align*}
    \maxdist([\zvec^-, \zvec^+])&=\max_{\uevec\in\UE}\min_{\uevec'\in[\zvec^-, \zvec^+]} \norm{\uevec - \uevec'}_1 \\&\overset{\text{(a)}}{=} \max_{\uevec\in\UE}\sum_{i\in[\numue]} \min_{\uescal'\in[z_i^-, z_i^+]} \lvert\uescal_i-\uescal'\rvert \overset{\text{(b)}}{=} 
    \max_{\uevec\in\UE} \sum_{i\in[\numue]} \begin{cases}
        z_i^--\uescal_i & \case z_i^- = z_i^+ = 2\\
        \uescal_i - z_i^+ & \case z_i^+\leq 0\\
        0 & \otherwise,
    \end{cases}
\end{align*}
where (a) follows from the properties of the $1$-norm and 
(b) is due to a case distinction and $\UE\subseteq [0,2]^\numue$. 
Combining both reformulations, we equivalently express the separation problem as
{\allowdisplaybreaks
\begin{align}
    &\max_{[\zvec^-, \zvec^+]\in\mathcal{Z}}
    \fdist(\ufvec/s, [\zvec^-, \zvec^+]) 
    - 
    \maxdist([\zvec^-, \zvec^+])
    \nonumber\\
    \overset{\text{(a)}}{=}&
    \max_{[\zvec^-, \zvec^+]\in\mathcal{Z}}
    \sum_{i\in[\numue]} \left(
    \frac{z_i^- -z_i^+}{2}  + 2
    \right)
    - 
    \max_{\uevec\in\UE} \sum_{i\in[\numue]} \begin{cases}
        2-\uescal_i & \case z_i^- = z^+_i = 2\\
        \uescal_i - z_i^+ & \case z_i^+\leq 0\\
        0 & \otherwise
    \end{cases}  \nonumber\\
    \overset{\text{(b)}}{=}&
    \max_{[\zvec^-, \zvec^+]\in\mathcal{Z}}
    \min_{\uevec\in\UE} \sum_{i\in[\numue]} 
    \begin{cases}
        \frac{z_i^- -z_i^+}{2}
        +\uescal_i & \case z_i^- = z^+_i = 2\\
    \frac{z_i^- -z_i^+}{2}  + 2 +
     z_i^+ - \uescal_i & \case z_i^+\leq 0\\
    \frac{z_i^- -z_i^+}{2}  + 2
     & \otherwise
    \end{cases} \nonumber\\
    \overset{\text{(c)}}{=}&
    \max_{\zvec^+ \in \{0,2\}^\numue}
    \min_{\uevec\in\UE} \sum_{i\in[\numue]} 
    \begin{cases}
    -\frac{z_i^+}{2} + 2 + z_i^+ - \uescal_i & \case z_i^+=0\\
    -\frac{z_i^+}{2} + 2
     & \case z_i^+=2
    \end{cases} \nonumber\\
    \overset{\text{(d)}}{=}&
    \max_{\zvec^+ \in \{0,2\}^\numue}
    \min_{\uevec\in\UE} \sum_{i\in[\numue]} 
    \left(-\frac{z_i^+}{2} + 2 - \uescal_i\left(1-\frac{z_i^+}{2}\right)\right) \nonumber\\
    \overset{\text{(e)}}{=}&
    \max_{\overline{\deltavec} \in \{0,1\}^\numue}
    \min_{\uevec\in\UE} \sum_{i\in[\numue]} 
    \left(1+\overline{\delta}_i - \uescal_i\overline{\delta}_i\right) \nonumber\\
    \overset{\text{(f)}}{=}&
    \max\{\unitvec^\intercal\overline{\deltavec} + 
    \min\{-\overline{\deltavec}^\intercal\uevec \,\colon\,
    \overline{\wmat} \uevec \leq 2 \unitvec \,,\, \uevec \geq \nullvec \}
    \,\colon\,
    \overline{\deltavec} \in \{0,1\}^\numue\} + \numue \nonumber\\
    \overset{\text{(g)}}{=}&
    \max\{\unitvec^\intercal\overline{\deltavec} - 
    \max\{\overline{\deltavec}^\intercal\uevec \,\colon\,
    \overline{\wmat} \uevec \leq 2 \unitvec \,,\, \uevec \geq \nullvec \}
    \,\colon\,
    \overline{\deltavec} \in \{0,1\}^\numue\} + \numue. \label{eq:proof_np_hard:subproblem}
\end{align}
}%
Here, 
(a) follows from the earlier reformulations of $\fdist$ and $\maxdist$, 
(b) merges the two terms, and
(c) holds since there always exists an optimal solution $[\zvec^-, \zvec^+]$ satisfying $\zvec^- = \nullvec$. Indeed, if $z^-_i = 2$ for some $i \in [I]$, implying that $z^+_i = 2$ as well (case~1), then setting $z^-_i = 0$ (and thus switching to case~3) weakly increases the objective value. Likewise, if $z^-_i = -2$ for some $i \in [I]$, then the expressions in cases~2 and~3 imply that increasing $z^-_i$---and, possibly in case~2, $z^+_i$---to $0$ increases the objective value.
Equation (d) holds by construction, and (e) substitutes $\overline{\delta}_i = 1-z^+_i/2$. Equation (f) reformulates the problem and replaces $\UE$ with its definition, and (g) negates the inner optimization problem.

Dualizing the inner maximization in \eqref{eq:proof_np_hard:subproblem} yields 
\begin{align}
   &\max\{\unitvec^\intercal\overline{\deltavec} - \min\{2\unitvec^\intercal\yvec \,\colon\,
    \overline{\wmat}^\intercal \yvec \geq \overline{\deltavec}\,,\,
    \yvec \geq \nullvec\}
    \,,\, 
    \overline{\deltavec} \in \{0,1\}^\numue\} + I \nonumber\\
    \overset{\text{(a)}}{=}&-
    \min\{2\unitvec^\intercal\yvec - \unitvec^\intercal\deltavec - \unitvec^\intercal\gammavec \,\colon\,
    \wmat^\intercal \yvec \geq \gammavec
    \,,\, 
    \yvec \geq \deltavec
    \,,\, 
    \yvec \geq \nullvec
    \,,\, 
    (\gammavec, \deltavec) \in \{0,1\}^{(m+n)}\} + I.\label{eq:proof_np_hard:mvcrel2}
\end{align}
Strong LP duality and finiteness of the optimal objective values hold since $\uevec=\nullvec$ is a primal feasible solution and $\yvec=\unitvec$ is a dual feasible solution for any $\overline{\deltavec}$. Equation (a) negates the maximum and replaces $\overline{\deltavec}$ with $(\bm{\gamma}, \bm{\delta})$. Thus, the values of~\eqref{eq:proof_np_hard:subproblem} and~\eqref{eq:proof_np_hard:mvcrel2} coincide.

We next claim that there is always an optimal solution $(\yvec, \gammavec, \deltavec)$ to problem~\eqref{eq:proof_np_hard:mvcrel2} with $\gammavec=\unitvec$. Indeed, assume that $\gamma_e = 0$ for some edge $e=\{ v, w \} \in\mathcal{E}$, and observe that $2y_v - \delta_v - \gamma_e \geq 0$ since $y_v \geq \delta_v$. By setting $y_v = \delta_v = \gamma_e = 1$, we obtain another feasible solution whose contribution $2y_v - \delta_v - \gamma_e = 0$ to the objective function of~\eqref{eq:proof_np_hard:mvcrel2} is weakly less than before. Fixing $\gammavec=\unitvec$ simplifies the optimization problem~\eqref{eq:proof_np_hard:mvcrel2} to
\begin{equation}
\label{eq:proof_np_hard:mvcrel}
   - \min\{2\unitvec^\intercal\yvec - \unitvec^\intercal\deltavec \,\colon\,
    \wmat^\intercal \yvec \geq \unitvec
    \,,\, 
    \yvec \geq \deltavec
    \,,\, 
    \yvec \geq \nullvec
    \,,\, 
    \deltavec \in \{0,1\}^n\} + I + m.
\end{equation}

We conclude the proof by arguing that (\ref{eq:proof_np_hard:mvc}) attains the same optimal value as the minimization problem inside (\ref{eq:proof_np_hard:mvcrel}), that is, that $\eqref{eq:proof_np_hard:mvc} = -\eqref{eq:proof_np_hard:mvcrel} + I + m$. To this end, observe first that any solution $\yvec$ feasible in~\eqref{eq:proof_np_hard:mvc} can be augmented to a feasible solution
$(\yvec, \deltavec)$ to~\eqref{eq:proof_np_hard:mvcrel} that attains the same objective value in the minimization problem inside~\eqref{eq:proof_np_hard:mvcrel} by setting $\deltavec = \yvec$. Conversely, we claim that there is always an optimal solution $(\yvec, \deltavec)$ to~\eqref{eq:proof_np_hard:mvcrel} with $\yvec \in \{0,1\}^n$; the partial solution $\yvec$ then attains the same objective value in (\ref{eq:proof_np_hard:mvc}) as $(\yvec, \deltavec)$ does in the minimization problem inside~\eqref{eq:proof_np_hard:mvcrel}. Assume to the contrary that all optimal solutions $(\yvec, \deltavec)$ to~\eqref{eq:proof_np_hard:mvcrel} are non-binary, and fix an optimal solution $(\yvec, \deltavec)$ to~\eqref{eq:proof_np_hard:mvcrel} with the fewest number of non-binary components in $\yvec$. Fix any $v \in \mathcal{V}$ with $y_v \in (0, 1)$. Then there must be $\{ v, w \} \in \mathcal{E}$ such that $\wmat^\intercal \yvec \geq \unitvec$ imposes the covering condition $y_v + y_w \geq 1$, which implies that either $y_v \geq 1/2$ or $y_w \geq 1/2$ or both. Assume without loss of generality that $y_v \geq 1/2$. In that case, setting $y_v = \delta_v = 1$ results in another feasible solution with fewer non-binary components that is at least as good as the previous one. This leads to a contradiction, and we thus conclude that there is indeed always an optimal solution $(\yvec, \deltavec)$ to~\eqref{eq:proof_np_hard:mvcrel} with $\yvec \in \{0,1\}^n$.
\end{proof}

%% file: proofs/proof_pro_symmetric.tex
We structure the proof of Theorem~\ref{thm:symmetric} into multiple lemmas. 
We first use the permutation invariance of Assumptions~\ref{ass:symmetric_Z} and~\ref{ass:symmetric_concave} to identify analytical representations for the distance function $\maxdist$. 
We split this simplification into three lemmas. 
Lemma~\ref{lem:rotational_solution} generalizes Lemma~3 of \citet{BenTal2020} to show that for each $i\in[\numue]$, $\UE$ contains realizations with one-norm $\eta(i)$ that are constant on their first $i$ components and zero on their remaining components. 
Lemma~\ref{lem:symmetric_solution} shows that there always exists an optimal solution to the separation problem that is symmetric around the origin. 
Finally, Lemma~\ref{lem:symmetric_dist} combines these two results to derive an analytical representation for $\maxdist$. 
Next, we use this representation to identify a small subset of solutions guaranteed to contain an optimal solution for the separation problem in Lemma~\ref{lem:reduced_symmetric_solution_structure}. 
We conclude the proof of part~\ref{thm:symmetric_algo} of Theorem~\ref{thm:symmetric} by describing how to efficiently enumerate this subset in Algorithm~\ref{alg:symmetric} and prove its validity in Lemma~\ref{lem:symmetric_alg}. 
For part~\ref{thm:symmetric_full} of Theorem~\ref{thm:symmetric}, Lemma~\ref{lem:symmetric_full} shows that the additional property in the statement of the theorem further reduces the set of possible solutions to the set of squares. 
We omit the index $k\in [\numc]$ of the separation problem considered and use $\altufvec = \ufvec/s$ for increased readability throughout the proof. 
Further, we assume that $\EBOX$ is a tight bounding box for $\UE$, that is $\underline{\uescal}_i\unitvec_i, \overline{\uescal}_i\unitvec_i\in\UE$ for all $i\in[\numue]$.
The extension to general bounding boxes $\EBOX$ is straightforward, but it requires additional tedious case distinctions leading to a less intuitive proof.

\begin{lemma}
\label{lem:rotational_solution}
    Let $\UE$ satisfy property \ref{ass:symmetric_U} of Assumption~\ref{ass:symmetric_concave}. 
    Then for each $i\in[\numue]$ there exists a $\lambda(i)$ such that
    $
        \sum_{i'\in[i]} \lambda(i) \unitvec_{i'} \in \UE,
    $
    and
    $
        i \lambda(i) = \eta(i).
    $
\end{lemma}
\begin{proof}[Proof of Lemma~\ref{lem:rotational_solution}]
Fix any $\uevec \in \argmax_{\uevec \in \UE} \sum_{i'\in[i]} \lvert \uescal_{i'} \rvert$. 
Let $\hat{\uevec}$ be the vector with absolute values of $\uevec$ in each component, that is, $\hat{\uescal}_{i'} = \lvert\uescal_{i'}\rvert$. 
By axis symmetry of $\UE$, we have $\hat{\uevec} \in \UE$. 
Thus, we assume w.l.o.g.~that $\uevec\geq 0$. 
The remainder of the proof follows from Lemma~3 of \citet{BenTal2020}, which considers the special case of Lemma~\ref{lem:rotational_solution} where $\UE$ is non-negative.
\end{proof}

\begin{lemma}
    \label{lem:symmetric_solution}
    Let $Z$ satisfy Assumption~\ref{ass:symmetric_Z} and let $\UE$ satisfy property \ref{ass:symmetric_U} of Assumption~\ref{ass:symmetric_concave}.
    Then there exists an optimal solution $[\zvec^-,\zvec^+]$ for the separation problem that satisfies
    $
    \zvec^- = - \zvec^+.
    $
\end{lemma}
\begin{proof}[Proof of Lemma \ref{lem:symmetric_solution}]
We split the proof into two parts. 
First, we show that there always exists an optimal solution $[\zvec^-,\zvec^+]$ to the separation problem that satisfies $\zvec^-\leq \nullvec \leq \zvec^+$. 
We then use this property to prove the actual claim. 
We prove both parts using contradiction arguments. 

Let $[\zvec^-,\zvec^+]$ be an optimal solution to the separation problem with a minimal number of indices $i\in[\numue]$ such that $z^+_i<0$. 
Assume there was such an index $i^\star$. 
Define $\hat{\zvec}{}^+$ via
$$
\hat{z}{}^+_i = \begin{cases}
    z^+_i & \case i\neq i^\star\\
    0 & \case i=i^\star.
\end{cases}
$$
Then $[\zvec^-,\hat{\zvec}{}^+]\in\mathcal{Z}$, as $z^-_i \leq z^+_i \leq \hat{z}{}^+_i$ for all $i\in[\numue]$. 
Fix any $\uevec\in\argmax_{\uevec\in\UE}\dist(\uevec, [\zvec^-, \hat{\zvec}^+])$ maximizing the distance from the new rectangle. 
By $z^-_{i^\star} \leq z^+_{i^\star} < \hat{z}^+_{i^\star} = 0$ we have $\lvert \uescal_{i^\star} \rvert-\hat{z}{}^+_{i^\star} > \lvert \uescal_{i^\star} \rvert + z^-_{i^\star}$, which by maximality of $\uevec$ and symmetry of $\UE$ implies $\uescal_{i^\star} \geq 0$.
With this we find
\begin{align*}
    \maxdist([\zvec^-,\hat{\zvec}{}^+])
    &\overset{\text{(a)}}{=} \sum_{i\in[\numue]} 
    \left(\uescal_i - \hat{z}{}^+_i\right)_+ +
    \left(z^-_i - \uescal_i\right)_+\\
    &\overset{\text{(b)}}{=}
    \sum_{i\in[\numue]} 
    \left(\uescal_i - z^+_i\right)_+ +
    \left(z^-_i - \uescal_i\right)_+
    +z^+_{i^\star}
    \overset{\text{(c)}}{\leq}
    \maxdist([\zvec^-,\zvec^+]) + z^+_{i^\star},
\end{align*}
Here, (a) follows from the choice of $\uevec$,
(b) is due to $\uescal_{i^\star}\geq 0$, 
and
(c) holds since $\uevec$ is a feasible, though not necessarily optimal, solution for the maximization problem represented by $\maxdist([\zvec^-,\zvec^+])$.
Similarly we find
\begin{align*}
    \fdist(\altufvec, [\zvec^-,\hat{\zvec}{}^+]) 
    &\overset{\text{(a)}}{=} \sum_{i\in[\numue]} 
    \sum\limits_{j \in\hat{\numuef}{}_i^+} \altufscal_{ij} 
    + z^-_i - \underline{\uescal}_i - \sum\limits_{j \in \numuef_i^-} \altufscal_{ij}\\
    &\overset{\text{(b)}}{=} \sum_{i\in[\numue]} 
    \sum\limits_{j \in\numuef_i^+} \altufscal_{ij} 
    + z^-_i - \underline{\uescal}_i - \sum\limits_{j \in \numuef_i^-} \altufscal_{ij}
    - \sum\limits_{j \in \numuef_{i^\star}^+\setminus\hat{\numuef}{}_{i^\star}^+} \altufscal_{i^\star j} \\
    &\overset{\text{(c)}}{=} 
    \fdist(\altufvec, [\zvec^-,\zvec^+]) 
    - \sum\limits_{j \in \numuef_{i^\star}^+\setminus\hat{\numuef}{}_{i^\star}^+} \altufscal_{i^\star j}
    \\&\overset{\text{(d)}}{\geq} 
    \fdist(\altufvec, [\zvec^-,\zvec^+]) 
    + z^+_{i^\star},
\end{align*}
where we use $\hat{\numuef}{}_i^+=\{j\in[\numuef_i] \colon \hat{z}^+_i < z_{ij}\}$ analogously to $\numuef^+_i$. 
Here, (a) and (c) are due to the definition of $\fdist$,
(b) follows from $z^+_{i^\star}<\hat{z}^+_{i^\star}$ and splitting the sum, and
(d) follows from $\altufscal_{ij} \leq \fold_{ij}(\overline{\uescal}_i) \leq z_{j}-z_{j-1}$ by Lemma~\ref{lem:rectangular_hull}, contracting the sum, and $\hat{z}{}^+_{i^\star}=0$.
Combining the newly found inequalities for $\maxdist$ and $\fdist$ yields $\fdist(\altufvec, [\zvec^-,\hat{\zvec}{}^+]) - \maxdist([\zvec^-,\hat{\zvec}{}^+]) \geq \fdist(\altufvec, [\zvec^-,\zvec^+]) - \maxdist([\zvec^-,\zvec^+])$.
This contradicts the assumption that $[\zvec^-,\zvec^+]$ had a minimal number of indices $i\in[\numue]$ such that $z^+_i<0$.
The case for indices $i\in[\numue]$ with $z^-_i>0$ follows analogously, and we thus conclude that the separation problem always has an optimal solution $[\zvec^-,\zvec^+]$ with $\zvec^-\leq \nullvec \leq \zvec^+$. 

We now use this result to prove the claim of the lemma. 
Let $[\zvec^-,\zvec^+]$ be an optimal solution to the separation problem that satisfies $\zvec^-\leq \nullvec \leq \zvec^+$ with a minimal number of indices $i\in[\numue]$ such that $z^+_{i}>-z^-_{i}$.
Assume there was such an index $i^\star$.
Define $\hat{\zvec}{}^+$ via
$$
\hat{z}{}^+_i = \begin{cases}
    z^+_i & \case i\neq i^\star\\
    -z^-_i & \case i=i^\star.
\end{cases}
$$
By $z^-_{i^\star} \leq 0 \leq -z^-_{i^\star} = \hat{z}{}^+_{i^\star}$, we have $[\zvec^-, \hat{\zvec}{}^+]\in\mathcal{Z}$. 
Fix any $\uevec\in\argmax_{\uevec\in\UE}\dist(\uevec, [\zvec^-, \hat{\zvec}^+])$ maximizing the distance from the new rectangle. 
Then by $\hat{z}^+_{i^\star} = -z^-_{i^\star}$ we have $\lvert \uescal_{i^\star} \rvert-\hat{z}{}^+_{i^\star} = \lvert \uescal_{i^\star} \rvert + z^-_{i^\star}$.
By the symmetry of $\UE$ we can assume w.l.o.g.~that $\uescal_{i^\star} \leq 0$.
We thus have 
$$
\left(\uescal_{i^\star} - \hat{z}{}^+_{i^\star}\right)_+ + 
\left(z^-_{i^\star} - \uescal_{i^\star}\right)_+
= 
\left(z^-_{i^\star} - \uescal_{i^\star}\right)_+
= 
\left(\uescal_{i^\star} - z^+_{i^\star}\right)_+ + 
\left(z^-_{i^\star} - \uescal_{i^\star}\right)_+.
$$
With this we find that
\begin{align*}
    \maxdist([\zvec^-,\hat{\zvec}{}^+])
    &= \sum_{i\in[\numue]} 
    \left(\uescal_i - \hat{z}{}^+_i\right)_+ +
    \left(z^-_i - \uescal_i\right)_+\\
    &=
    \sum_{i\in[\numue]} 
    \left(\uescal_i - z^+_i\right)_+ +
    \left(z^-_i - \uescal_i\right)_+
    \leq
    \maxdist([\zvec^-,\zvec^+]).
\end{align*}
Using a similar argument as above for $\fdist$, we additionally find that
\begin{align*}
    \fdist(\altufvec, [\zvec^-,\hat{\zvec}{}^+]) 
    &\overset{\text{(a)}}{=} \sum_{i\in[\numue]} 
    \sum\limits_{j \in\hat{\numuef}{}_i^+} \altufscal_{ij} 
    + z^-_i - \underline{\uescal}_i - \sum\limits_{j \in \numuef_i^-} \altufscal_{ij}\\
    &\overset{\text{(b)}}{=} \sum_{i\in[\numue]} 
    \sum\limits_{j \in\numuef_i^+} \altufscal_{ij} 
    + z^-_i - \underline{\uescal}_i - \sum\limits_{j \in \numuef_i^-} \altufscal_{ij}
    + \sum\limits_{j \in \hat{\numuef}{}_{i^\star}^+\setminus\numuef_{i^\star}^+} \altufscal_{i^\star j} \\
    &\overset{\text{(c)}}{=} 
    \fdist(\altufvec, [\zvec^-,\zvec^+]) 
    + \sum\limits_{j \in \hat{\numuef}{}_{i^\star}^+\setminus\numuef_{i^\star}^+} \altufscal_{i^\star j}
    \\&\overset{\text{(d)}}{\geq} 
    \fdist(\altufvec, [\zvec^-,\zvec^+]).
\end{align*}
Here, (a) and (c) are due to the definition of $\fdist$,
(b) follows from $\hat{z}{}^+_{i^\star}\leq z^+_{i^\star}$ and splitting the sum, and
(d) holds since $\altufvec\geq \nullvec$.
Combining the inequalities for $\maxdist$ and $\fdist$ yields $\fdist(\altufvec, [\zvec^-,\hat{\zvec}{}^+]) - \maxdist([\zvec^-,\hat{\zvec}{}^+]) \geq \fdist(\altufvec, [\zvec^-,\zvec^+]) - \maxdist([\zvec^-,\zvec^+])$.
This contradicts the assumption that $[\zvec^-,\zvec^+]$ had a minimal number of violated indices. 
The case for indices $i$ with $z^+_{i}<-z^-_{i}$ follows analogously. 
We thus conclude that the separation problem always has an optimal solution $[\zvec^-,\zvec^+]$ with $\zvec^- = - \zvec^+$, which completes the proof.
\end{proof}

\begin{lemma}
\label{lem:symmetric_dist}
    Let $Z$ satisfy Assumption~\ref{ass:symmetric_Z} and let $\UE$ satisfy property \ref{ass:symmetric_U} of Assumption~\ref{ass:symmetric_concave}. 
    Let $[\zvec^-,\zvec^+]\in\mathcal{Z}$ with $\zvec^-=-\zvec^+$ and $\zvec^+$ being non-descending, that is, $\forall i\in[\numue-1] \colon z^+_i \leq z^+_{i+1}$.
    Then 
    $
    \maxdist([\zvec^-,\zvec^+]) = \max_{i\in[\numue]}\{\eta(i) - \sum_{i'\in[i]} z^+_{i'}\}.
    $
\end{lemma}
\begin{proof}[Proof of Lemma \ref{lem:symmetric_dist}]
We prove the claim by showing each of the inequalities ``$\leq$" and ``$\geq$" separately. 
For the ``$\geq$" direction, recall that by Lemma~\ref{lem:rotational_solution}, there is some $\lambda(i)$ such that $\sum_{i'\in[i]}\lambda(i)\unitvec_{i'}\in\UE$, and $i\lambda(i) = \eta(i)$ for each $i\in[\numue]$.
Then
$$
\maxdist([\zvec^-, \zvec^+]) 
\overset{\text{(a)}}{\geq} \max_{i\in[\numue]}\sum_{i'\in[i]}(\lambda(i) - z^+_{i'})_+ 
\overset{\text{(b)}}{\geq}
\max_{i\in[\numue]}\sum_{i'\in[i]}(\lambda(i) - z^+_{i'})
\overset{\text{(c)}}{=} \max_{i\in[\numue]}(\eta(i) - \sum_{i'\in[i]} z^+_{i'}),
$$
where (a) follows from $\sum_{i'\in[i]}\lambda(i)\unitvec_{i'}\in\UE$ and the definition of $\maxdist$, (b) follows from $\zvec^+\geq\nullvec$ and relaxing the $(\cdot)_+$ operator, and (c) follows from $i\lambda(i) = \eta(i)$. 
This concludes the ``$\geq$" direction.

For the ``$\leq$" direction, fix any $\uevec\in\argmax_{\uevec\in\UE}\dist(\uevec, [\zvec^-,\zvec^+])$. 
By the symmetry of $\UE$ and $\zvec^+ = - \zvec^-$, we assume w.l.o.g.~that $\uevec\geq \nullvec$.
First, we show that there always exists a maximizing $\uevec$ with non-ascending ordering, that is, $\uescal_i \geq \uescal_{i+1}$ for all $i\in[\numue-1]$. 
To see this, assume there were two indices $i_1 \leq i_2$ with $\uescal_{i_1} \leq \uescal_{i_2}$. 
We show $\left(\uescal_{i_1} - z^+_{i_1}\right)_+ +
\left(\uescal_{i_2} - z^+_{i_2}\right)_+ 
\leq
\left(\uescal_{i_2} - z^+_{i_1}\right)_+  +
\left(\uescal_{i_1} - z^+_{i_2}\right)_+$, which implies that swapping the values of $\uescal_{i_1}$ and $\uescal_{i_2}$ does not decrease $\dist(\uevec, [\zvec^-,\zvec^+])$.

For $\uescal_{i_1} \geq z^+_{i_1}$ the claim follows from
\begin{align*}
\left(\uescal_{i_1} - z^+_{i_1}\right)_+ +
\left(\uescal_{i_2} - z^+_{i_2}\right)_+ 
&\overset{\text{(a)}}{=}
\uescal_{i_1} - z^+_{i_1} +
\left(\uescal_{i_2} - \uescal_{i_1} + \uescal_{i_1} - z^+_{i_2}\right)_+ \\
&\overset{\text{(b)}}{\leq}
\uescal_{i_1} + \uescal_{i_2} - \uescal_{i_1} - z^+_{i_1} +
\left(\uescal_{i_1} - z^+_{i_2}\right)_+ 
\overset{\text{(c)}}{=}
\left(\uescal_{i_2} - z^+_{i_1}\right)_+  +
\left(\uescal_{i_1} - z^+_{i_2}\right)_+, 
\end{align*} 
where (a) follows from $\uescal_{i_1} \geq z^+_{i_1}$, 
(b) is due to $\uescal_{i_1}\leq \uescal_{i_2}$ and the piecewise linearity of $(\cdot)_+$, and
(c) holds since $z^+_{i_1}\leq \uescal_{i_1} \leq \uescal_{i_2}$. 

For $\uescal_{i_1} \leq z^+_{i_1}$ the claim follows from 
$$
\left(\uescal_{i_1} - z^+_{i_1}\right)_+ +
\left(\uescal_{i_2} - z^+_{i_2}\right)_+ 
\overset{\text{(a)}}{=}
\left(\uescal_{i_2} - z^+_{i_2}\right)_+ 
\overset{\text{(b)}}{\leq}
\left(\uescal_{i_2} - z^+_{i_1}\right)_+ 
\overset{\text{(c)}}{=}
\left(\uescal_{i_2} - z^+_{i_1}\right)_+  +
\left(\uescal_{i_1} - z^+_{i_2}\right)_+,
$$
where (a) follows from $\uescal_{i_1} \leq z^+_{i_1}$, 
(b) holds since $z^+_{i_1}\leq z^+_{i_2}$, and 
(c) is due to $\uescal_{i_1}\leq z^+_{i_1} \leq z^+_{i_2}$. 
Thus, by the permutation invariance of $\UE$, we can w.l.o.g assume that $\uevec$ is in non-ascending order. 

Let $i^\star\in[\numue]$ be the last index with $\uescal_{i^\star} \geq z^+_{i^\star}$. 
Such an index always exists by $\overline{\uescal}_1\geq z^+_{1}$ and $\overline{\uescal}_1\unitvec_1\in\UE$.
By the ordering of $\uevec$ and $\zvec^+$, we have $\uescal_{i} \geq z^+_{i}$ for all $i\leq i^\star$. 
We conclude the proof by showing that $\maxdist([\zvec^-, \zvec^+]) \leq \eta(i^\star) - \sum_{i\in[i^\star]} z^+_i$.
In particular, we find
\begin{align*}
    \maxdist([\zvec^-, \zvec^+]) &\overset{\text{(a)}}{=}
    \sum_{i\in\numue} 
    \left(\uescal_i - z^+_i\right)_+ +
    \left(z^-_i - \uescal_i\right)_+\\
    &\overset{\text{(b)}}{=}
    \sum_{i\in\numue} 
    \left(\uescal_i - z^+_i\right)_+ 
    \overset{\text{(c)}}{=}
    \sum_{i\in[i^\star]} 
    \uescal_i - z^+_i
    \overset{\text{(d)}}{\leq}
    \eta(i^\star) -
    \sum_{i\in[i^\star]} z^+_i,
\end{align*}
where (a) follows from the optimality of $\uevec$,
(b) holds since $\uevec\geq\nullvec$, 
(c) is due to the choice of $i^\star$, and 
(d) follows from the definition of $\eta$ and $\uevec\in\UE$.
\end{proof}

\begin{lemma}
\label{lem:reduced_symmetric_solution_structure}
    Let $Z$ satisfy Assumption~\ref{ass:symmetric_Z} and let $\UE$ satisfy Assumption~\ref{ass:symmetric_concave}. 
    Then there always exists an index pair $(i, j) \in [\numue] \times [\numuef]$, $j\geq\frac{\numuef}{2}$, such that there is an optimal solution $[\zvec^-,\zvec^+]$ for the separation problem that satisfies
    $z^+_{\pi(i')} = \begin{cases}
        z_{j} &\case i'\leq i\\
        z_{j+1} &\case i' > i
    \end{cases}$ and $\zvec^- = -\zvec^+$, where $\pi$ is an non-ascending ordering of $\{\delta_{i'j}\}_{i'\in[\numue]}$ with $\delta_{i'j}= z_{j+1}- z_{j} + \altufscal_{i',j+1} - \altufscal_{i',\numuef - j}$. 
\end{lemma}
\begin{proof}[Proof of Lemma \ref{lem:reduced_symmetric_solution_structure}]
The proof proceeds in two parts. 
We first prove the characteristics of Lemma~\ref{lem:reduced_symmetric_solution_structure} for a potentially arbitrary ordering of $\{\delta_{ij}\}_{i\in[\numue]}$. 
Then, we use the resulting solution to construct another solution with the same objective value, following a non-ascending ordering of $\{\delta_{ij}\}_{i\in[\numue]}$.

For the first part of the proof, with a potentially arbitrary ordering of $\{\delta_{ij}\}_{i\in[\numue]}$, let $[\zvec^-,\zvec^+]$ with $\zvec^+=-\zvec^-$ be an optimal solution to the separation problem. 
By Lemma~\ref{lem:symmetric_solution}, such a solution always exists. 
Let $\sigma$ be an ordering of the indices such that $z^+_{\sigma(i)}\leq z^+_{\sigma(i+1)}$ for all $i\in[\numue-1]$.
Then,
$$
\maxdist([\zvec^-,\zvec^+]) \overset{\text{(a)}}{=} \maxdist([\sigma(\zvec^-),\sigma(\zvec^+)]) \overset{\text{(b)}}{=} \max_{i\in[\numue]}\eta(i) - \sum_{i'\in[i]} z^+_{\sigma(i')}.
$$
Here (a) follows from permutational invariance of $\UE$ and
(b) is due to Lemma \ref{lem:symmetric_dist}.
Let $i^\star\in[\numue]$ be the maximal index such that $\maxdist([\zvec^-,\zvec^+]) = \eta(i^\star) - \sum_{i'\in[i^\star]} z^+_{\sigma(i')}$ and let $j^\star\in[\numuef]$ be the maximal index such that 
\begin{align}
     z_{j^\star} & \leq \eta(i^\star) - \eta(i^\star-1) \label{eq:jstarub}\\
     \text{and} \quad z_{j^\star} & \leq z^+_{\sigma(i^\star+1)} \label{eq:jstarub2}.
\end{align}
With this we construct $\hat{\zvec}{}^+$ via
\begin{equation*}
  \hat{z}{}^+_{\sigma(i)} =\begin{cases}
    z_{j^\star} & \case i \leq i^\star \text{ or } z_{j^\star} > \eta(i^\star+1) - \eta(i^\star) \\
    z_{j^\star+1} & \case i > i^\star \text{ and } z_{j^\star} \leq \eta(i^\star+1) - \eta(i^\star),
\end{cases}  
\end{equation*}
we define $\hat{\zvec}{}^- = -\hat{\zvec}{}^+$, and we claim that
\begin{equation}
\label{eq:reduced_symmetric_solution_structure:sigmaoptimal}
\fdist(\altufvec, [\hat{\zvec}{}^-, \hat{\zvec}{}^+]) - \maxdist([\hat{\zvec}{}^-, \hat{\zvec}{}^+])\geq \fdist(\altufvec, [\zvec^-, \zvec^+]) - \maxdist([\zvec^-, \zvec^+]).
\end{equation}

To prove \eqref{eq:reduced_symmetric_solution_structure:sigmaoptimal}, we first simplify $\maxdist([\hat{\zvec}{}^-, \hat{\zvec}{}^+])$ and show that
\begin{equation}
    \label{eq:betaoptforjhat}
    \maxdist([\hat{\zvec}{}^-, \hat{\zvec}{}^+]) = \eta(i^\star) - \sum_{i\in[i^\star]} \hat{z}{}^+_{\sigma(i)} = \eta(i^\star) - i^\star z_{j^\star}.
\end{equation}
By Lemma \ref{lem:symmetric_dist}, there exists an index $i'$ such that $\maxdist([\hat{\zvec}{}^-, \hat{\zvec}{}^+]) = \eta(i') - \sum_{i\in[i']} \hat{z}^+_{\sigma(i)}$. 
We claim that $i' = i^\star$, which immediately implies \eqref{eq:betaoptforjhat}.  

First, consider the case where $i' > i^\star$. 
Then
\begin{align*}
    \eta(i') - \sum_{i\in[i']} \hat{z}^+_{\sigma(i)} 
    &\overset{\text{(a)}}{=} 
    \eta(i^\star) - \sum_{i\in[i^\star]} \hat{z}^+_{\sigma(i)} + 
    \eta(i') - \eta(i^\star) -\sum_{i\in [i']\setminus[i^\star]} \hat{z}^+_{\sigma(i)}
    \\&\overset{\text{(b)}}{\leq}
    \eta(i^\star) - \sum_{i\in[i^\star]} \hat{z}^+_{\sigma(i)} + 
    (i'-i^\star) (\eta(i^\star+1) - \eta(i^\star) - \hat{z}{}^+_{\sigma(i^\star+1)})
    \overset{\text{(c)}}{<}
    \eta(i^\star) - \sum_{i\in[i^\star]} \hat{z}^+_{\sigma(i)},
\end{align*}
where (a) follows from adding a smart zero and splitting the sum,
(b) is due to the non-increasing differences of $\eta$ and the definition of $\hat{\zvec}^+$,
and (c) holds since $\hat{z}{}^+_{\sigma(i^\star+1)}>\eta(i^\star+1)-\eta(i^\star)$, which follows from a case distinction in the definition of $\hat{\zvec}{}^+$. 
By this definition, the claim is immediate for the case $\hat{z}{}^+_{\sigma(i^\star+1)} = z_{j^\star}$. 
For the case $\hat{z}{}^+_{\sigma(i^\star+1)} = z_{j^\star+1}$, we consider another case distinction over \eqref{eq:jstarub} and \eqref{eq:jstarub2}. 
If \eqref{eq:jstarub} is binding, then the claim follows from the non-increasing differences of $\eta$. 
If \eqref{eq:jstarub2} is binding, then the claim follows from $z_{j^\star+1} > z^+_{\sigma(i^\star+1)}$ and 
$
z^+_{\sigma(i^\star+1)} > \eta(i^\star+1)-\eta(i^\star),
$
which holds since $\eta(i^\star+1) - \sum_{i\in[i^\star+1]} z^+_{\sigma(i)} < \eta(i^\star) - \sum_{i\in[i^\star]} z^+_{\sigma(i)}$ by the optimality and maximality of the index $i^\star$. 

Next, consider the case where $i' < i^\star$. 
Then
\begin{align*}
    \eta(i') - \sum_{i\in[i']} \hat{z}^+_{\sigma(i)} 
    &\overset{\text{(a)}}{=} 
    \eta(i^\star) - \sum_{i\in[i^\star]} \hat{z}^+_{\sigma(i)} + 
    \eta(i') - \eta(i^\star) + \sum_{i\in[i^\star]\setminus[i']} \hat{z}^+_{\sigma(i)}
    \\&\overset{\text{(b)}}{\leq}
    \eta(i^\star) - \sum_{i\in[i^\star]} \hat{z}^+_{\sigma(i)} + 
    (i^\star-i') (z_{j^\star} - (\eta(i^\star) - \eta(i^\star-1)))
    \overset{\text{(c)}}{\leq}
    \eta(i^\star) - \sum_{i\in[i^\star]} \hat{z}^+_{\sigma(i)},
\end{align*}
where (a) follows from adding a smart zero and splitting the sum,
(b) is due to the non-increasing differences of $\eta$ and the definition of $\hat{\zvec}^+$, and
(c) follows from (\ref{eq:jstarub}). 

To conclude our proof of \eqref{eq:reduced_symmetric_solution_structure:sigmaoptimal}, we show the following relational property between $\hat{\zvec}{}^+$ and $\zvec^+$:
\begin{align}
    z^+_{\sigma(i)} &\leq \hat{z}^+_{\sigma(i)} & \case i \leq i^\star \label{eq:leqk}\\
    \text{and} \quad z^+_{\sigma(i)} &\geq \hat{z}^+_{\sigma(i)} & \case i \geq i^\star+1. \label{eq:geqkp1}
\end{align}
For (\ref{eq:leqk}) it is sufficient to show that $z^+_{\sigma(i^\star)} \leq z_{j^\star}$, as $z^+_{\sigma(i)} \leq z^+_{\sigma(i^\star)}$ for all $i\leq i^\star$ by definition of $\sigma$ and $\hat{z}^+_{\sigma(i)} = z_{j^\star}$ for all $i\leq i^\star$ by definition of $\hat{\zvec}{}^+$, respectively. 
Consider the case where (\ref{eq:jstarub}) is the binding constraint for $j^\star$.
Then by the maximality of $i^\star$, we find
$
\eta(i^\star) - \sum_{i\in[i^\star]} z^+_{\sigma(i)} \geq \eta(i^\star-1) - \sum_{i\in[i^\star-1]} z^+_{\sigma(i)},
$ 
which implies that
$
z^+_{\sigma(i^\star)} \leq \eta(i^\star) - \eta(i^\star-1).
$
This implies (\ref{eq:leqk}) as $j^\star$ is the maximal index such that $z_{j^\star} \leq \eta(i^\star) - \eta(i^\star-1)$. 
Now consider the case where (\ref{eq:jstarub2}) is the binding constraint for $j^\star$. 
Then 
$
z^+_{\sigma(i^\star)} \leq z^+_{\sigma(i^\star+1)} = z_{j^\star} 
$
by the definition of $\sigma$. 

For (\ref{eq:geqkp1}) it is sufficient to show $z^+_{\sigma(i^\star+1)} \geq \hat{z}{}^+_{\sigma(i^\star+1)}$,  
as $z^+_{\sigma(i)} \geq z^+_{\sigma(i^\star+1)}$ for all $i\geq i^\star+1$ by definition of $\sigma$ and $\hat{z}^+_{\sigma(i)} = \hat{z}^+_{\sigma(i^\star+1)}$ for all $i\geq i^\star+1$ by definition of $\hat{\zvec}{}^+$, respectively. 
Consider the case $z_{j^\star} > \eta(i^\star+1) - \eta(i^\star)$.
Then 
$
\hat{z}^+_{\sigma(i^\star+1)} = z_{j^\star} \leq z^+_{\sigma(i^\star+1)}
$
by (\ref{eq:jstarub2}). 
Now consider the case $z_{j^\star} \leq \eta(i^\star+1) - \eta(i^\star)$. 
Then $\hat{z}^+_{\sigma(i^\star+1)} = z_{j^\star+1}$. 
Further, we have $z^+_{\sigma(i^\star+1)} > \eta(i^\star+1) - \eta(i^\star) \geq z_{j^\star}$ by optimality and maximality of $i^\star$. 
Thus $z^+_{\sigma(i^\star+1)} \geq z_{j^\star + 1} = \hat{z}^+_{\sigma(i^\star+1)}$.

Next, we find that
\begin{align*}
    &\fdist(\altufvec, [\hat{\zvec}{}^-, \hat{\zvec}{}^+]) - \maxdist([\hat{\zvec}{}^-, \hat{\zvec}{}^+]) \\
    \overset{\text{(a)}}{=} & 
    \sum_{i\in[\numue]} \left(
    \sum\limits_{j \in\hat{\numuef}{}^+_i} \altufscal_{ij} 
    + \hat{z}^-_i - \underline{\uescal}_i - \sum\limits_{j\in\hat{\numuef}{}^-_i} \altufscal_{ij} \right) - \eta(i^\star) + \sum_{i\in[i^\star]} \hat{z}{}^+_{\sigma(i)}
    \\\overset{\text{(b)}}{=}&
    \sum_{i\in[\numue]} \left(
    \sum\limits_{j \in\numuef^+_i} \altufscal_{ij} 
    + z^-_i - \underline{\uescal}_i - \sum\limits_{j\in\numuef^-_i} \altufscal_{ij} \right) - \eta(i^\star) + \sum_{i\in[i^\star]} \hat{z}{}^+_{\sigma(i)}
    \displaybreak[0]\\& +\sum_{i\in[\numue]\setminus[i^\star]} 
    \left(
    \sum\limits_{j \in\hat{\numuef}{}^+_{\sigma(i)}\setminus\numuef^+_{\sigma(i)}} \altufscal_{\sigma(i) j} 
    + \hat{z}{}^-_{\sigma(i)} - z^-_{\sigma(i)} - \sum\limits_{j \in\hat{\numuef}{}^-_{\sigma(i)}\setminus\numuef^-_{\sigma(i)}} \altufscal_{\sigma(i) j} \right)
    \\& -\sum_{i\in[i^\star]} \left(
    \sum\limits_{j \in\numuef^+_{\sigma(i)}\setminus\hat{\numuef}{}^+_{\sigma(i)}} \altufscal_{\sigma(i) j} 
    - \hat{z}{}^-_{\sigma(i)} + z^-_{\sigma(i)} - \sum\limits_{j \in\numuef^-_{\sigma(i)}\setminus \hat{\numuef}{}^-_{\sigma(i)}} \altufscal_{\sigma(i) j} \right)
    \displaybreak[0]\\\overset{\text{(c)}}{=}&
    \sum_{i\in[\numue]} \left(
    \sum\limits_{j \in\numuef^+_i} \altufscal_{ij} 
    + z^-_i - \underline{\uescal}_i - \sum\limits_{j\in\numuef^-_i} \altufscal_{ij} \right) - \eta(i^\star) + \sum_{i\in[i^\star]} z^+_{\sigma(i)}
    \\& +\sum_{i\in[\numue]\setminus[i^\star]} 
    \left(
    \sum\limits_{j \in\hat{\numuef}{}^+_{\sigma(i)}\setminus\numuef^+_{\sigma(i)}} \altufscal_{\sigma(i) j} 
    + \hat{z}{}^-_{\sigma(i)} - z^-_{\sigma(i)} - \sum\limits_{j \in\hat{\numuef}{}^-_{\sigma(i)}\setminus\numuef^-_{\sigma(i)}} \altufscal_{\sigma(i) j} \right)
    \\& -\sum_{i\in[i^\star]} \left(
    \sum\limits_{j \in\numuef^+_{\sigma(i)}\setminus\hat{\numuef}{}^+_{\sigma(i)}} \altufscal_{\sigma(i) j} 
    - \sum\limits_{j \in\numuef^-_{\sigma(i)}\setminus \hat{\numuef}{}^-_{\sigma(i)}} \altufscal_{\sigma(i) j} \right)
    \displaybreak[0]\\\overset{\text{(d)}}{\geq} &
    \sum_{i\in[\numue]} \left(
    \sum\limits_{j \in\numuef^+_i} \altufscal_{ij} 
    + z^-_i - \underline{\uescal}_i - \sum\limits_{j\in\numuef^-_i} \altufscal_{ij} \right) - \eta(i^\star) + \sum_{i\in[i^\star]} z^+_{\sigma(i)} 
    \\=& \fdist(\altufvec, [\zvec^-, \zvec^+]) - \maxdist([\zvec^-, \zvec^+]),
\end{align*}
where $\hat{\numuef}{}^+_i = \{j\in[\numuef_i] \colon \hat{z}{}^+_i < z_{ij}\}$ and $\hat{\numuef}{}^-_i = \{j\in[\numuef_i] \colon \hat{z}{}^-_i \geq z_{ij}\}$ analogous to $\numuef^+_i$ and $\numuef^-_i$. 
Here, (a) is due to (\ref{eq:betaoptforjhat}) and the definition of $\fdist$,
(b) follows from (\ref{eq:leqk}) and (\ref{eq:geqkp1}) as well as separating the sum,
(c) follows from $\zvec^+ = -\zvec^-$ and $\hat{\zvec}^+ = -\hat{\zvec}^-$ and combining the sums over $[i^\star]$ in the first and last term, and
(d) holds since the second row is non-negative by $\hat{z}{}^-_{\sigma(i)} - z^-_{\sigma(i)} = \sum_{j \in\hat{\numuef}{}^-_{\sigma(i)}\setminus\numuef^-_{\sigma(i)}} z_{j+1}-z_j$ and $0\leq \altufscal_{ij}\leq z_{j} - z_{j-1}$ for all $i\in[\numue],j\in[\numuef_i]$, and the third row is non-negative by $\altufscal_{ij}\geq \altufscal_{i,\numuef-j}$ $i\in[\numue],j\in[\numuef/2]$, and $\numuef^+_i = \{\numuef-j\colon j\in\numuef^-_i\}$ and $\hat{\numuef}{}^+_i = \{\numuef-j\colon j\in\hat{\numuef}{}^-_i\}$ by $\zvec^+ = -\zvec^-$ and $\hat{\zvec}^+ = -\hat{\zvec}^-$, respectively. 
Thus $[\hat{\zvec}{}^-, \hat{\zvec}{}^+]$ is an optimal solution to \eqref{eq:subproblem}.

To conclude the proof, we show that there is an optimal solution $[\overline{\zvec}{}^-, \overline{\zvec}{}^+]$ following the ordering $\pi$ of $\{\delta_{ij^\star}\}_{i\in[\numue]}$. 
Define $\overline{\zvec}{}^+$ via 
$\overline{z}{}^+_{\pi(i)} = \hat{z}{}^+_{\sigma(i)}$, $i \in[\numue]$, and $\overline{\zvec}{}^-=-\overline{\zvec}{}^+$.
If $z_{j^\star} > \eta(i^\star+1) - \eta(i^\star)$, then we have $\overline{z}{}^+_{\pi(i)} = \hat{z}{}^+_{\sigma(i)}=z_{j^\star}$ for all $i \in[\numue]$, which implies that $\overline{\zvec}{}^+=\hat{\zvec}{}^+$, making $[\overline{\zvec}{}^-, \overline{\zvec}{}^+]$ an optimal solution to \eqref{eq:subproblem}.
Now consider the case where $z_{j^\star} \leq \eta(i^\star+1) - \eta(i^\star)$.
As $\maxdist$ is permutation invariant, $\maxdist([\overline{\zvec}{}^-, \overline{\zvec}{}^+]) = \maxdist([\hat{\zvec}{}^-, \hat{\zvec}{}^+])$.
For $\fdist$, we find that
\begin{align*}
    \fdist(\altufvec, [\overline{\zvec}{}^-, \overline{\zvec}{}^+]) 
    =&
    \sum_{i\in[\numue]} \left(
    \sum\limits_{j \in \overline{\numuef}{}^+_i} \altufscal_{ij} 
    + \overline{z}{}^-_i - \underline{\uescal}_i - \sum\limits_{j \in \overline{\numuef}{}^-_i} \altufscal_{ij} \right)
    \\\overset{\text{(a)}}{=}&
    \sum_{i\in[\numue]} \left(
    \sum\limits_{j \in [\numuef]\setminus[j^\star+1]} \altufscal_{ij} 
    - z_{j^\star+1} - \underline{\uescal}_i - \sum\limits_{j \in [\numuef-(j^\star+1)]} \altufscal_{ij} \right)
    \\&+ 
    \sum_{i\in[i^\star]} \left( \altufscal_{\pi(i), j^\star+1} - z_{j^\star} + z_{j^\star+1}- \altufscal_{\pi(i), \numuef-j^\star} \right)
    \\\overset{\text{(b)}}{=}&
    \sum_{i\in[\numue]} \left(
    \sum\limits_{j \in [\numuef]\setminus[j^\star+1]} \altufscal_{ij} 
    - z_{j^\star+1} - \underline{\uescal}_i - \sum\limits_{j \in [\numuef-(j^\star+1)]} \altufscal_{ij} \right)
    + \sum_{i\in[i^\star]} \delta_{\pi(i) j^\star}
    \\\overset{\text{(c)}}{\geq}&
    \sum_{i\in[\numue]} \left(
    \sum\limits_{j \in [\numuef]\setminus[j^\star+1]} \altufscal_{ij} 
    - z_{j^\star+1} - \underline{\uescal}_i - \sum\limits_{j \in [\numuef-(j^\star+1)]} \altufscal_{ij} \right)
    + \sum_{i\in[i^\star]} \delta_{\sigma(i) j^\star}
    \\\overset{\text{(d)}}{=}& \fdist(\altufvec, [\hat{\zvec}{}^-, \hat{\zvec}{}^+]).
\end{align*}
Here, (a) follows from splitting the sums and the definition of $[\overline{\zvec}{}^-, \overline{\zvec}{}^+]$,
(b) is due to the definition of $\delta_{ij^\star}$,
(c) holds since $\pi$ is a non-ascending ordering of $\{\delta_{ij^\star}\}_{i\in[\numue]}$, 
and (d) follows from combining the arguments from (a) and (b). 
Thus, $[\overline{\zvec}{}^-, \overline{\zvec}{}^+]$ is an optimal solution to \eqref{eq:subproblem} that satisfies all properties claimed by the lemma.
\end{proof}

Lemma~\ref{lem:reduced_symmetric_solution_structure} provides us with the key structural property to efficiently find an optimal solution for the separation problem, as it guarantees that one of the $\frac{\numue \numuef}{2}$ solution candidates with the properties specified by the lemma is optimal. 
Using Algorithm~\ref{alg:symmetric}, we can efficiently enumerate all of these candidates.

\begin{algorithm}[h]
\caption{Permutation invariant cut separation \label{alg:symmetric}}
\DontPrintSemicolon
\KwData{ A scaled incumbent solution $\altufvec = \ufvec / s$ to problem $\aff \lproblem(\PUF, \UFO)$}
\KwResult{An optimal solution $[\zvec^-, \zvec^+]$ for the separation problem}
    $o^\star = o = 0$ \;
  \For{$j = \numuef-1$ \KwTo $ \frac{\numuef}{2} $}{
    Set $\delta_{ij}= z_{j+1}- z_{j} + \altufscal_{i,j+1} - \altufscal_{i,\numuef - j}$ for all $i\in[\numue]$\;
    Let $\pi$ be a non-ascending order of the indices of $\{\delta_{ij}\}_{i\in[\numue]}$ \;
    \For{$i = 1$ \KwTo $\numue$}{
    $o = o + \delta_{\pi(i) j} - \begin{cases}
        z_{j+1} - z_j & \case \eta(i) - \eta(i-1) \geq z_{j+1}\\
        \eta(i) - \eta(i-1) - z_j &\case z_{j+1} > \eta(i) - \eta(i-1) \geq z_{j} \\
        0 & \text{otherwise}
    \end{cases}$ \label{alg:update}\;
    \If{$o > o^\star$}{
    $j^\star=j, i^\star=i, \pi^\star = \pi, o^\star = o$ \;
    }
    }
    }
    \Return $[\zvec^-,\zvec^+]$ with $z^+_{\pi^\star(i)} = \begin{cases}
        z_{j^\star} &\case i\leq i^\star\\
        z_{j^\star+1} &\case i > i^\star
    \end{cases}\quad\forall i\in[\numue]$ and $\zvec^- = -\zvec^+$\;
\end{algorithm}

\begin{lemma}
    \label{lem:symmetric_alg}
    Let $Z$ satisfy Assumption~\ref{ass:symmetric_Z} and let $\UE$ satisfy Assumption~\ref{ass:symmetric_concave}. 
    Then, Algorithm~\ref{alg:symmetric} determines an optimal solution for the separation problem in time $\mathcal{O}(\numue\numuef \log(\numue))$.
\end{lemma}
\begin{proof}[Proof of Lemma \ref{lem:symmetric_alg}]
For the runtime complexity note that each operation in the inner loop of Algorithm~\ref{alg:symmetric} can be executed in $\mathcal{O}(1)$, and thus the entire inner loop runs in time $\mathcal{O}(\numue)$. 
Sorting the values $\{\delta_{ij}\}_{i\in[\numue]}$ for a single index $j$ takes time $\mathcal{O}(\numue\log(\numue))$ using standard sorting algorithms such as merge sort. 
Thus, the entire runtime of Algorithm~\ref{alg:symmetric} is bounded by $\mathcal{O}(\numue\numuef\log(\numue))$.

To prove the optimality of the algorithm, we show that it iterates through all solutions of the form introduced by Lemma~\ref{lem:reduced_symmetric_solution_structure}.
In particular, we show that $o$ maintains the objective value of the solution $[\zvec^-, \zvec^+]$ induced by $i$, $j$ and $\pi$ via the construction in Lemma~\ref{lem:reduced_symmetric_solution_structure}, that is,  $o=\fdist(\altufvec, [\zvec^-, \zvec^+]) - \maxdist(\altufvec, [\zvec^-, \zvec^+])$. 
Throughout the remainder of the proof, we use the abbreviations $\fdist_{ij}=\fdist(\altufvec, [\zvec^-, \zvec^+])$ and $\maxdist_{ij}=\maxdist(\altufvec, [\zvec^-, \zvec^+])$. 
Note that we omit the explicit dependence of $[\zvec^-, \zvec^+]$ on $i$ and $j$ as well as the dependence of $\pi$ on $j$. 
We prove the result of the lemma by showing that
\begin{align}
    &\fdist_{ij} = \fdist_{i-1,j}+\delta_{\pi(i)j} &&\forall i\in[\numue], j\in[\numuef]\label{eq:reduced_symmetric_solution_structure:fdist}\\
    \text{and } &\maxdist_{ij} = \maxdist_{i-1,j} + \begin{cases}
        z_{j+1} - z_j & \case \eta(i) - \eta(i-1) \geq z_{j+1}\\
        \eta(i) - \eta(i-1) - z_j &\case z_{j+1} > \eta(i) - \eta(i-1) \geq z_{j} \\
        0 & \otherwise
    \end{cases} && \forall i\in[\numue], j\in[\numuef].\label{eq:reduced_symmetric_solution_structure:maxdist}
\end{align}
By construction, we have $\fdist_{\numue j}=\fdist_{0,j+1}$ and $\maxdist_{\numue j}=\maxdist_{0,j+1}$. 
Thus, combining \eqref{eq:reduced_symmetric_solution_structure:fdist} and \eqref{eq:reduced_symmetric_solution_structure:maxdist} immediately implies the claim. 

Equation \eqref{eq:reduced_symmetric_solution_structure:fdist} follows from the construction of $[\zvec^-, \zvec^+]$ in Lemma~\ref{lem:reduced_symmetric_solution_structure}. 
Specifically, we find
\begin{align*}
    \fdist_{ij} \overset{\text{(a)}}{=}&
    \sum_{i'\in[i]}\left(\sum_{j'\in [\numuef]\setminus[j]}\altufscal_{\pi(i')j'}-z_j-\underline{\uescal}_{i'}-\sum_{j'\in[\numuef-j]}\altufscal_{\pi(i')j'} \right) 
    \\&+ 
    \sum_{i'\in[\numue]\setminus[i]}\left(\sum_{j'\in [\numuef]\setminus[j+1]}\altufscal_{\pi(i')j'}-z_{j+1}-\underline{\uescal}_{i'}-\sum_{j'\in[\numuef-(j+1)]}\altufscal_{\pi(i')j'} \right)
    \\\overset{\text{(b)}}{=}&
    \sum_{i'\in[i-1]}\left(\sum_{j'\in [\numuef]\setminus[j]}\altufscal_{\pi(i')j'}-z_j-\underline{\uescal}_{i'}-\sum_{j'\in[\numuef-j]}\altufscal_{\pi(i')j'} \right) 
    \\&+ 
    \sum_{i'\in[\numue]\setminus[i-1]}\left(\sum_{j'\in [\numuef]\setminus[j+1]}\altufscal_{\pi(i')j'}-z_{j+1}-\underline{\uescal}_{i'}-\sum_{j'\in[\numuef-(j+1)]}\altufscal_{\pi(i')j'} \right) 
    \\&+
    (\altufscal_{\pi(i), j+1} + z_{j+1} - z_j - \altufscal_{\pi(i), \numuef-j})
    \overset{\text{(c)}}{=}\fdist_{i-1,j}+\delta_{\pi(i)j},
\end{align*}
where (a) and (c) follow from the construction of $[\zvec^-, \zvec^+]$ and the definition of $\delta_{ij}$ and (b) follows from shifting the sums.

To show equation \eqref{eq:reduced_symmetric_solution_structure:maxdist}, we first simplify the expression for $\maxdist_{ij}$ and prove the claim via a case distinction. 
Recall that Lemma~\ref{lem:symmetric_dist} implies the existence of an index $i^\star$ for each $i\in[\numue]$ and $j\in[\numuef], j\geq\frac{\numuef}{2}$, such that
\begin{equation}
\maxdist_{ij} = \eta(i^\star)-\sum_{i'\in[i^\star]} z^+_{i'} \overset{\text{(a)}}{=} \eta(i^\star)- i^\star z_{j+1} + \min\{i, i^\star\} (z_{j+1} - z_j), \label{eq:reduced_symmetric_solution_structure:compact_maxdist}
\end{equation}
where (a) follows from the construction of $\zvec^+$. 
Let $i^\star\in[\numue]$ be the maximal index satisfying~\eqref{eq:reduced_symmetric_solution_structure:compact_maxdist} and let $i^-\in [\numue]$ be the maximal index such that 
$
\maxdist_{i-1,j} = \eta(i^-) - i^- z_{j+1} + \min\{i-1,i^-\} (z_{j+1} - z_{j}).
$ 
By construction of $i^\star$, we have
\begin{align*}
	&\eta(i)-\eta(i-1) \overset{\text{(a)}}{\leq} \eta(i^\star+1) - \eta(i^\star) \overset{\text{(b)}}{<} z_j && \case i^\star < i \\
	&\eta(i+1)-\eta(i) \overset{\text{(a)}}{\leq} \eta(i^\star+1) - \eta(i^\star) \overset{\text{(b)}}{<} z_{j+1} && \case i^\star \leq i \\
	&\eta(i)-\eta(i-1) \overset{\text{(a)}}{\geq} \eta(i^\star) - \eta(i^\star-1) \overset{\text{(c)}}{\geq} z_j && \case i^\star \geq i \\
	&\eta(i+1)-\eta(i) \overset{\text{(a)}}{\geq} \eta(i^\star) - \eta(i^\star-1) \overset{\text{(c)}}{\geq} z_{j+1} && \case i^\star > i,
\end{align*}
where the inequalities (a) are due to the non-increasing differences of $\eta$, the inequalities (b) follow from $i^\star$ being a maximal index satisfying \eqref{eq:reduced_symmetric_solution_structure:compact_maxdist}, and the inequalities (c) are due to the optimality of $i^\star$, that is, $\eta(i^\star)-\sum_{i'\in[i^\star]} z^+_{i'} =  \max_{\hat{i}\in[\numue]}(\eta(\hat{i})-\sum_{i'\in[\hat{i}]} z^+_{i'})$ by Lemma~\ref{lem:symmetric_dist}. 
Negating the above implications yields
\begin{subequations}
\begin{align}
	\hspace{4cm}
    \eta(i) - \eta(i-1) & \geq z_j &\Rightarrow&&&  i^\star \geq i \label{eq:kstarlbweak} \hspace{4cm}\\
    \eta(i+1) - \eta(i) & \geq z_{j+1} &\Rightarrow&&&  i^\star > i \label{eq:kstarlbstrong}\\
    \eta(i) - \eta(i-1) & <  z_j \hspace{-.5cm}&\Rightarrow&&&  i^\star < i \label{eq:kstarubstrong}\\
    \eta(i+1) - \eta(i) & <  z_{j+1} &\Rightarrow&&&  i^\star \leq i. \label{eq:kstarubweak}
\end{align}
\end{subequations}
Using the same argument, we can show that the implications \eqref{eq:kstarlbweak}--\eqref{eq:kstarubweak} hold for $i^-$ when substituting $i$ by $i-1$. 
We next consider each of the three cases in the case distinction of~\eqref{eq:reduced_symmetric_solution_structure:maxdist} separately.

\emph{Case 1.} Let $\eta(i) - \eta(i-1) \geq z_{j+1}$. 
We find that
\begin{align*}
    \maxdist_{ij} &= \eta(i^\star) - i^\star z_{j+1} + \min\{i,i^\star\} (z_{j+1} - z_{j}) 
    \\&\overset{(a)}{=} \eta(i^\star) - i^\star z_{j+1} + i (z_{j+1} - z_{j})
    \\&\overset{(b)}{=} \eta(i^-) - i^- z_{j+1} + (i-1) (z_{j+1} - z_{j}) + (z_{j+1} - z_{j})
    \\&\overset{(c)}{=} \eta(i^-) - i^- z_{j+1} + \min\{i-1,i^-\} (z_{j+1} - z_{j}) + (z_{j+1} - z_{j})
    = \maxdist_{j, i-1} + (z_{j+1} - z_{j}),
\end{align*}
where (a) and (c) follow from $i^\star\geq i$ by (\ref{eq:kstarlbweak}) and $i^-\geq i$ by (\ref{eq:kstarlbstrong}), respectively. 
For (b), we use
\begin{align*}
    \eta(i^-) - i^- z_{j+1} + (i-1) (z_{j+1} - z_{j}) &\geq \eta(i^\star) - i^\star z_{j+1} + (i-1) (z_{j+1} - z_{j}) \\
    \intertext{and}
    \eta(i^\star) - i^\star z_{j+1} + i (z_{j+1} - z_{j}) &\geq \eta(i^-) - i^- z_{j+1} + i (z_{j+1} - z_{j})
\end{align*}
that follow from $i^\star, i^-\geq i$ and the optimality of $i^-$ and $i^\star$, respectively.

\emph{Case 2.} Let $z_{j+1} > \eta(i) - \eta(i-1) \geq z_{j}$.
Then $i^\star\leq i$ by (\ref{eq:kstarubweak}) and the non-increasing differences of $\eta$,
$i^-\leq i-1$ by (\ref{eq:kstarubweak}),
$i^\star\geq i$ by (\ref{eq:kstarlbweak}), and
$i^-\geq i-1$ by (\ref{eq:kstarlbweak}) and the non-increasing differences of $\eta$.
We thus have $i^\star = i$ and $i^-=i-1$. 
With this, we find that
\begin{align*}
    \maxdist_{ij} &= \eta(i^\star) - i^\star z_{j+1} + \min\{i,i^\star\} (z_{j+1} - z_{j}) 
    \\&\overset{(a)}{=} \eta(i) - i z_{j+1} + i (z_{j+1} - z_{j})
    \\&\overset{(b)}{=} \eta(i-1) - (i-1) z_{j+1} + (i-1) (z_{j+1} - z_{j}) 
    + \eta(i) - \eta(i-1) - z_{j}
    \\&\overset{(c)}{=} \eta(i^-) - i^- z_{j+1} + \min\{i-1,i^-\} (z_{j+1} - z_{j})
    + \eta(i) - \eta(i-1) - z_{j}
    \\&= \maxdist_{j, i-1} + \eta(i) - \eta(i-1) - z_{j},
\end{align*}
where (a) and (c) follow from $i^\star= i$ and $i^-= i-1$, respectively, 
and (b) follows from adding $\eta(i-1) - \eta(i-1)$.

\emph{Case 3.} Let $\eta(i) - \eta(i-1) < z_{j}$.
We find that
\begin{align*}
    \maxdist_{ij} &= \eta(i^\star) - i^\star z_{j+1} + \min\{i,i^\star\} (z_{j+1} - z_{j}) 
    \\&\overset{(a)}{=} \eta(i^\star) - i^\star z_{j+1} + i^\star (z_{j+1} - z_{j})
    \\&\overset{(b)}{=} \eta(i^-) - i^- z_{j+1} + i^- (z_{j+1} - z_{j}) 
    \\&\overset{(c)}{=} \eta(i^-) - i^- z_{j+1} + \min\{i-1,i^-\} (z_{j+1} - z_{j})
    = \maxdist_{j, i-1},
\end{align*}
where (a) and (c) follow from $i^\star\leq i-1$ by (\ref{eq:kstarubstrong}) and $i^-\leq i-1$ by (\ref{eq:kstarubweak}), respectively. 
For (b), we use
\begin{align*}
    \eta(i^-) - i^- z_{j+1} + i^- (z_{j+1} - z_{j}) &\geq \eta(i^\star) - i^\star z_{j+1} + i^\star (z_{j+1} - z_{j}) \\
    \intertext{and}
    \eta(i^\star) - i^\star z_{j+1} + i^\star (z_{j+1} - z_{j}) &\geq \eta(i^-) - i^- z_{j+1} + i^- (z_{j+1} - z_{j})
\end{align*}
that follow from $i^\star, i^-\leq i-1$ and the optimality of $i^-$ and $i^\star$, respectively.
\end{proof}

\begin{lemma}
\label{lem:symmetric_full}
Let the Assumptions~\ref{ass:symmetric_Z} and~\ref{ass:symmetric_concave} both hold.
If for every $j \in [\numuef]$ there is at most one $i \in [\numue]$ such that $\eta(i) - \eta(i-1) \in [z_{j-1}, z_j)$, then the separation problem~\eqref{eq:subproblem} has a square solution $[z_j \bm{e}, -z_j \bm{e}]$ for some $j \in [\numuef / 2]$.
\end{lemma}
\begin{proof}[Proof of Lemma~\ref{lem:symmetric_full}]
The proof follows from Step~\ref{alg:update} of Algorithm~\ref{alg:symmetric}. 
For any iteration $i, j$ of the algorithm, let $o_{ij}$ be the value of $o$ after Step~\ref{alg:update} and recall that $o_{0j}=o_{\numue,j+1}$. 
Let $i^\star$, $j^\star$ and $\pi$ be an optimal solution returned by the algorithm. 
We claim that
$$
o_{i^\star j^\star} \leq \begin{cases}
    o_{0 j^\star} & \case \eta(i^\star)-\eta(i^\star-1) \geq z_{j^\star +1}\\
    o_{I j^\star} & \case \eta(i^\star)-\eta(i^\star-1) < z_{j^\star +1}
\end{cases}
$$

For the first case, note that by the non-increasing differences of $\eta$, each $i\leq i^\star$ satisfies $\eta(i)-\eta(i-1)\geq z_{j^\star+1}$. 
We thus have
$$
o_{i^\star j^\star} \overset{\text{(a)}}{=} o_{0 j^\star} + \sum_{i\in [i^\star]} \altufscal_{\pi(i), j^\star+1} - \altufscal_{\pi(i), \numuef-j^\star}
\overset{\text{(b)}}{\leq} o_{0 j^\star},
$$
where (a) follows from Step~\ref{alg:update}. 
For (b) we use
$$
\altufscal_{\pi(i), j^\star+1} = (z_{j^\star+1} - z_{j^\star})\frac{\altufscal_{\pi(i), j^\star+1}}{z_{j^\star+1} - z_{j^\star}} \overset{\text{(c)}}{\leq} (z_{j^\star+1} - z_{j^\star})\frac{\altufscal_{\pi(i), \numuef-j^\star}}{z_{\numuef - j^\star} - z_{\numuef - (j^\star+1)}} \overset{\text{(d)}}{=} \altufscal_{\pi(i), \numuef-j^\star},
$$
where (c) follows from $\altufvec\in\UFOG$, $j^\star\geq \numuef/2$, and Lemma~\ref{lem:rectangular_hull} and (d) holds since $z_{j} = - z_{\numuef -j},\; j\in[\numuef]$, by Assumption~\ref{ass:symmetric_Z}.

For the second case, note that the non-increasing differences of $\eta$ and the assumption that there is at most one $i$ with $\eta(i)-\eta(i-1) \in[z_{j^\star}, z_{j^\star+1})$ imply that each $i > i^\star$ satisfies $\eta(i)-\eta(i-1) < z_{j^\star}$.
We thus have
$$
o_{I j^\star} \overset{\text{(a)}}{=} o_{i^\star j^\star} + \sum_{i\in [i^\star]} \altufscal_{\pi(i), j^\star+1} + z_{j^\star+1} - z_{j^\star} - \altufscal_{\pi(i), \numuef-j^\star}
\overset{\text{(b)}}{\geq} o_{i^\star j^\star},
$$
where (a) is due to Step~\ref{alg:update} and (b) follows from $0 \leq \altufscal_{\pi(i), j^\star+1} - \altufscal_{\pi(i), \numuef-j^\star} \leq z_{j^\star+1} - z_{j^\star}$ by $\altufvec\in\UFOG$, $j^\star\geq \numuef/2$, Lemma~\ref{lem:rectangular_hull}, and Assumption~\ref{ass:symmetric_Z}.
\end{proof}

%% file: proofs/proof_pro_pballs.tex
\begin{proof}[Proof of Proposition~\ref{pro:pballs}]
Property~\ref{ass:symmetric_U} of Assumption~\ref{ass:symmetric_concave} immediately follows from the permutation invariance of the norm $\norm{\cdot}_{p}$. 
Property~\ref{ass:concave} of Assumption~\ref{ass:symmetric_concave}, on the other hand, follows from the closed-form expression of $\eta(i)$. 
Indeed, Lemma~\ref{lem:rotational_solution} implies that $\eta(i)=i\lambda(i)$, where $\lambda(i)=i^{-\frac{1}{p}}\delta$ is the maximal value satisfying $\norm{\sum_{i'\in[i]} \lambda(i)\unitvec_{i'}}_{p} \leq \delta$. 
We thus have $\eta(i)=i^{1-\frac{1}{p}}\delta$. 
Since $p\geq 1$, the exponent of $i$ in this expression is always between zero and one. This makes $\eta$ a concave function of $i$, which implies non-increasing differences. 
\end{proof}

%% file: proofs/proof_pro_intersections.tex
\begin{proof}[Proof of Proposition~\ref{pro:intersections}]
We first show the permutation invariance of $\UE$.
Let $\uvec\in\UE=\bigcap_{h\in[n]}\UE_h$ and $\smat\in \mathcal{S}$. 
Then $\smat\uvec\in\UE_h$ for all $h\in[n]$ by the symmetry and permutation invariance of each $\UE_h$. 
Thus $\smat\uvec\in \bigcap_{h\in[n]}\UE_h = \UE$.

We next show that $\eta(i)=\min_{h\in[n]}\eta_h(i)$, which implies property~\ref{ass:concave} of Assumption~\ref{ass:symmetric_concave}, as the minimum preserves non-increasing differences. 
Indeed, Lemma~\ref{lem:rotational_solution} implies $\eta_h(i)=i\lambda_h(i)$ for all $h\in[n], i\in[\numue]$, where $\lambda_h(i)$ is the maximal value such that $\lambda_h(i)\sum_{i'\in[i]}\unitvec_{i'}\in \UE_h$. 
Analogously, $\eta(i)=i\lambda(i)$, where $\lambda(i)$ is the maximal value such that $\lambda(i)\sum_{i'\in[i]}\unitvec_{i'}\in \UE=\bigcap_{h\in[n]}\UE_h$. 
The intersection immediately implies that $\lambda(i)\leq \min_{h\in[n]}\lambda_h(i)$. 
By convexity and the fact that $\nullvec\in\U_h$, we further observe that $\min_{h'\in[n]}\lambda_{h'}(i) \sum_{i'\in[i]} \unitvec_{i'} \in\U_h$ for all $h\in[n]$. 
We thus have $\min_{h\in[n]}\lambda_{h}(i) \sum_{i'\in[i]} \unitvec_{i'} \in\bigcap_{h\in[n]}\UE_h = \UE$, which implies that $\lambda(i)= \min_{h\in[n]}\lambda_h(i)$ and $\eta(i)=\min_{h\in[n]}\eta_h(i)$. 
\end{proof}

%% file: proofs/proof_pro_bounding_by_symmetry.tex
\begin{proof}[Proof of Proposition~\ref{pro:bounding_by_symmetric}]
We use the reformulations offered by Observation~\ref{obs:conic-reformulation} to show the inequalities from the statement of the proposition. 
In particular, we show that (i) any optimal solution to the reformulation of $\aff\lproblem(\PU, \UFO)$ is a feasible solution to the reformulation of $\aff\lproblem(\PU, \UFOV{}^\mathrm{A})$ and (ii) any optimal solution to the reformulation of $\aff\lproblem(\PU, \UFOV{}^\mathrm{A})$ is a feasible solution to the reformulation of $\aff\lproblem(\PU, \UFOV{}^\mathrm{P})$. 

(i) $\aff\lproblem(\PU, \UFO) \leq \aff\lproblem(\PU, \UFOV{}^\mathrm{A})$: Note that the reformulations only differ in the affine cuts~\eqref{eq:counterpart_dual_cuts}. 
More specifically, the reformulation of $\aff\lproblem(\PU, \UFOV{}^\mathrm{A})$ only adds the subsets $\mathcal{Z}^\mathrm{A}_g\subseteq\mathcal{Z}$ of cuts for each $g\in[\nums]$ found by the separation algorithm. 
Thus, the solution space of $\aff\lproblem(\PU, \UFO)$ is a subset of the solution space of $\aff\lproblem(\PU, \UFOV{}^\mathrm{A})$, which immediately implies the inequality. 

(ii) $\aff\lproblem(\PU, \UFOV{}^\mathrm{A}) \leq \aff\lproblem(\PU, \UFOV{}^\mathrm{P})$: Let $(\ufvec, \svec)$ be a feasible solution of $\aff\lproblem(\PU, \UFOV{}^\mathrm{A})$. 
Note that the reformulations only differ in the last set of constraints in \eqref{eq:counterpart_dual} and the affine cuts~\eqref{eq:counterpart_dual_cuts}, which separate along indices $g\in[\nums], k\in[\numc_g]$. 
For the remainder of the proof, we fix one index pair $(g,k)$. 
For $s_{gk}=0$, the constraints of interest trivially hold. 
Thus, we w.l.o.g.~assume that $s_{gk}>0$ and define $\altufvec=\frac{1}{s_{gk}}\ufvec_{gk}$. 
Then, the last constraint of \eqref{eq:counterpart_dual} is equivalent to $\altufvec\in\UFOG_g$, which implies that $\fold^+(\altufvec)\in \UE_g$ and $\altufvec\in\FBOXg$. 
Using $\UE_g\subseteq\UE^\mathrm{P}_g$ and consequently $\EBOXg\subseteq[\underline{\uevec}{}^\mathrm{P}_g, \overline{\uevec}{}^\mathrm{P}_g]$, we find that $\fold^+(\altufvec)\in \UE^\mathrm{P}_g$ and $\altufvec\in\convex(\fold([\underline{\uevec}{}^\mathrm{P}_g, \overline{\uevec}{}^\mathrm{P}_g]))$, which implies that $(\ufvec, \svec)$ satisfies all constraints in the reformulation \eqref{eq:counterpart_dual} of $\aff\lproblem(\PU, \UFOV{}^\mathrm{P})$. 

For the cuts~\eqref{eq:counterpart_dual_cuts} we recall that $\maxdist_g$ depends on the embedded support set $\UE_g$. 
We denote by $\maxdist{}^\mathrm{P}_g([\zvec^-,\zvec^+])=\max_{\uevec\in\UE^\mathrm{P}_g} d(\uevec, [\zvec^-,\zvec^+])$ the maximal distance to the permutation invariant $\UE^\mathrm{P}_g$. 
Note that $\maxdist_g([\zvec^-,\zvec^+])\leq\maxdist{}^\mathrm{P}_g([\zvec^-,\zvec^+])$ by $\UE_g\subseteq\UE^\mathrm{P}_g$. 
Thus, by 
$$\fdist(\altufvec, [\zvec^-,\zvec^+])\leq \maxdist_g([\zvec^-,\zvec^+])\leq\maxdist{}^\mathrm{P}_g([\zvec^-,\zvec^+])$$
all cuts induced via $[\zvec^-,\zvec^+]\in \mathcal{Z}^\mathrm{A}_g$ are also satisfied by the reformulation of $\aff\lproblem(\PU, \UFOV{}^\mathrm{P})$. 
By Lemma~\ref{lem:reduced_symmetric_solution_structure} we can w.l.o.g.~restrict the reformulation of $\aff\lproblem(\PU, \UFOV{}^\mathrm{P})$ to a subset $\mathcal{Z}^\mathrm{P}\subseteq\mathcal{Z}$ satisfying the specific structure described in the lemma. 
Further, any violated cut in the reformulation of $\aff\lproblem(\PU, \UFOV{}^\mathrm{P})$ induced by some $[\zvec^-,\zvec^+]\in \mathcal{Z}^\mathrm{P}$ is guaranteed to be found by Algorithm~\ref{alg:symmetric}. 
Note that the order in which cuts are enumerated by Algorithm~\ref{alg:symmetric} solely depends on the current solution $\altufvec$, and recall that after each iteration of Step~\ref{alg:update} of the algorithm we have $o=\fdist(\altufvec, [\zvec^-,\zvec^+])-\max_{i^\star\in[\numue]}\{\eta_g(i^\star)-\sum_{i'\in[i^\star]}z^+_{i'}\}$, where $\eta_g(i)=\max_{\uevec\in\UE_g}\sum_{i'\in[i]}\lvert\uevec_{i'}\rvert$, which follows from the proof of Lemma~\ref{lem:symmetric_alg}.
Denote by $o^\mathrm{P}$ the value for the permutation invariant case, and let $\eta_g^\mathrm{P}(i)$ be the maximal $l_1$-norm of the first $i$ components of $\UE_g^\mathrm{P}$. 
We then have
$$\eta_g(i)=\max_{\uevec\in\UE}\sum_{i'\in[i]}\lvert\uevec_{i'}\rvert\leq \max_{\uevec\in\UE^\mathrm{P}}\sum_{i'\in[i]}\lvert\uevec_{i'}\rvert=\eta_g^\mathrm{P}(i),$$
which implies that $o\geq o^\mathrm{P}$. 
Thus, any violated cut for the reformulation of $\aff\lproblem(\PU, \UFOV{}^\mathrm{P})$ found by Algorithm~\ref{alg:symmetric} is also found as a violated cut for the reformulation of $\aff\lproblem(\PU, \UFOV{}^\mathrm{A})$. 
Consequently, at termination any feasible solution to the reformulation of $\aff\lproblem(\PU, \UFOV{}^\mathrm{A})$ is also a feasible solution to the reformulation of $\aff\lproblem(\PU, \UFOV{}^\mathrm{P})$. 
\end{proof}

%% file: proofs/proof_pro_bounds.tex
The proof of Proposition~\ref{pro:bounds} builds on the piecewise affine policies designed by \citet{BenTal2020} and \citet{Thomae2024}. 
To keep our work self-contained, we briefly summarize their key results in the context of our problem setting. 

\paragraph{Piecewise Affine Policies Via Domination.}
The key observation for piecewise affine policies via domination is that for any two uncertainty realizations $\uvec$ and $\hat{\uvec}$, where $\hat{\uvec}$ dominates $\uvec$ in the sense that $\uvec \leq \hat{\uvec}$, Assumption~\ref{ass:approximation}~\ref{ass:approximation:rhs} implies that $\bvec_g(\hat{\uvec})\leq \bvec_g(\uvec)$ for all $g\in[G]$. 
Thus, any fixed solution $\xvec$ that is feasible for $\hat{\uvec}$ is also feasible for $\uvec$. 
We use this property to construct approximate policies by replacing the uncertainty support $\U$ via a dominating support $\hat{\U}$, such that for each $\uvec\in \U$ there is some $\hat{\uvec}\in \hat{\U}$ with $\uvec\leq \hat{\uvec}$. 
Assumption~\ref{ass:approximation:UE} extends the domination properties described above to the embedded support $\UE$. 
Analogous to the dominating polytopes in the original works, we define a \emph{dominating polytope} $\hat{\UE} = \convex(\bm{v}_0, \bm{v}_1, \dots, \bm{v}_\numue)$ for $\UE$ such that $\bm{v}_0\in \mathbb{R}^\numue$, $\beta(\bm{v}_0)\in \mathbb{R}_+$, $\bm{v}_i = \bm{v}_0 + \rho \unitvec_i$ with $\rho\geq 0$ and $\sum_{i\in [\numue]} (\uescal_i - v_{0i})_+ \leq \rho$ for all $\uevec\in\UE$. 
Note that the last property implies that $\hat{\UE}$ dominates $\UE$  \citep[\emph{cf.}][Condition~(5)]{Thomae2024}. 
When replacing the embedded support with a dominating polytope $\hat{\UE}$, all worst-case realizations are contained within the vertices $\bm{v}_0, \dots, \bm{v}_\numue$. 
Thus, the dominating problem is equivalent to
\begin{equation}
\tag{$\hat{\baseproblem}(\hat{\UE})$}
\label{eq:hatP}
\begin{aligned}
    \minimize_{\xvec_{0}, \dots, \xvec_{\numue}} \quad& \cvec^\intercal\xvec_{00}\\
    \st \quad& \amat_g \xvec_{i} \leq \bvec_g(\embed^+ \bm{v}_{i}) && \forall g\in[\nums], i\in[\numue]_0 \\
    & \xvec_{it} = \xvec_{i' t} && \forall i,i'\in [\numue]_0, t\in [T]_0 \;\textnormal{with}\; \bm{v}_i^t = \bm{v}_{i'}^t,
\end{aligned}
\end{equation}
which follows from straightforward minor modifications to Lemma~2 of \cite{Thomae2024}. 
Intuitively, \ref{eq:hatP} determines one solution $\xvec_i$ to the embedded problem for each vertex $\bm{v}_i$. 
The last set of constraints couples the solutions $\xvec_i$ to ensure nonanticipativity. 
Note that the here-and-now decisions $\xvec_{i0}$ are identical for all $i\in [\numue]_0$ as $\bm{v}_i^0$ is the empty vector, and thus $\bm{v}_i^0 = \bm{v}_{i'}^0$ holds for any $i,i'\in [I]_0$. 
The solution $(\xvec_0, \dots,\xvec_\numue)$ to \ref{eq:hatP} induces a feasible piecewise affine solution $\xvec(\uvec)$ to $\baseproblem(\PU, \U)$ with the same objective value via $\xvec(\uvec) = \xvec_0 + \frac{1}{\rho}\sum_{i\in[\numue]} \left([\embed(\uvec)]_i - v_{0i}\right)_+ (\xvec_i - \xvec_0)$. 
This policy is further guaranteed to be within a factor $\beta (\bm{v}_0) + \rho$ of optimal solutions to $\baseproblem(\PU, \U)$.
Since these guarantees are central to our proof of Proposition~\ref{pro:bounds}, we first formalize them by adapting Theorem~1 of \citet{Thomae2024} to the dominating problem \ref{eq:hatP}. 

\begin{lemma}
\label{lem:dominatingpolicybound}
Let $\baseproblem(\PU, \U)$ be a problem instance that satisfies Assumption~\ref{ass:approximation}, let $\embed$ be such that Assumption~\ref{ass:approximation:UE} holds, and let $\hat{\UE} = \convex(\bm{v}_0, \bm{v}_1, \dots, \bm{v}_\numue)$ be a dominating polytope for $\UE$.
Then,
$$
\baseproblem(\PU, \U) \leq \hat{\baseproblem}(\hat{\UE}) \leq (\beta (\bm{v}_0) + \rho) \baseproblem(\PU, \U).
$$
\end{lemma}
\begin{proof}[Proof of Lemma~\ref{lem:dominatingpolicybound}]
First, note that either both problems are bounded or both problems are unbounded. 
This follows, as unboundedness of either of the problems implies the existence of a feasible ray of decisions $\xvec$ with negative objective value, implying that both problems are unbounded (\emph{cf.}~Part~1 of the original proof). 
For the remainder of the proof, we assume both problems are bounded and show each of the inequalities separately. 

For the first inequality, let $(\xvec_0, \dots,\xvec_\numue)$ be an optimal solution to \ref{eq:hatP} and consider the piecewise affine policy $\xvec(\uvec) = \xvec_0 + \frac{1}{\rho}\sum_{i\in[\numue]} \left([\embed(\uvec)]_i - v_{0i}\right)_+ (\xvec_i - \xvec_0)$ on $\baseproblem(\PU, \U)$. 
The nonanticipativity constraints of \ref{eq:hatP} imply that the here-and-now decisions of $\xvec(\uvec)$ are given by $\xvec(\uvec)_0 = \xvec_{00}$. 
Thus, $\mathbb{E}_{\PU} [\cvec(\suvec)^\intercal \xvec(\suvec)] = \cvec^\intercal\xvec_{00} = \hat{\baseproblem}(\hat{\U})$ by part \ref{ass:approximation:obj} of Assumption~\ref{ass:approximation}. 
Further, $\xvec(\uvec)$ is feasible for $\baseproblem(\PU, \U)$ as for any $g\in[G]$ and $\uvec\in\U_g$, we have
\begin{align*}
\amat_g \xvec(\uvec) 
& \overset{\textnormal{(a)}}{=} 
\amat_g \left(\left(1-\frac{1}{\rho}\sum_{i\in[\numue]} \left([\embed(\uvec)]_i - v_{0i}\right)_+\right) \xvec_0 + \frac{1}{\rho}\sum_{i\in[\numue]} \left([\embed(\uvec)]_i - v_{0i}\right)_+\xvec_i  \right)
\\
& \overset{\textnormal{(b)}}{\leq}
\left(\left(1-\frac{1}{\rho}\sum_{i\in[\numue]} \left([\embed(\uvec)]_i - v_{0i}\right)_+\right) \bvec_g(\embed^+\bm{v}_0) + \frac{1}{\rho}\sum_{i\in[\numue]} \left([\embed(\uvec)]_i - v_{0i}\right)_+\bvec_g(\embed^+\bm{v}_i)  \right)
\\
& \overset{\textnormal{(c)}}{=}
\bvec_g\left(\embed^+\left(\bm{v}_0 + \frac{1}{\rho}\sum_{i\in[\numue]} \left([\embed(\uvec)]_i - v_{0i}\right)_+(\bm{v}_i - \bm{v}_0)  \right)\right)
\\
& \overset{\textnormal{(d)}}{=}
\bvec_g(\embed^+(\max\{\embed(\uvec), \vvec_0\})) \overset{\textnormal{(e)}}{\leq} \bvec_g(\embed^+(\embed(\uvec))) \overset{\textnormal{(f)}}{=} \bvec_g(\uvec).
\end{align*}
Here, (a) follows from the definition of $\xvec$, and (b) follows from $\amat_g \xvec_i \leq \bvec_g(\embed^+ \bm{v}_i)$ for all $i\in [\numue]_0$ by feasibility of $(\xvec_0, \dots, \xvec_\numue)$ in \ref{eq:hatP}, linearity, and the factors pre-multiplying $\bvec_g$ being non-negative. 
In particular, the choice of $\rho$ implies that $1-\frac{1}{\rho}\sum_{i\in[\numue]}\left([\embed(\uvec)]_i - v_{0i}\right)_+ \geq 0$. 
Finally, (c) follows from the linearity of $\bvec_g$ and $\embed^+$, 
(d) follows from rearranging the terms, using the element-wise maximum, 
(e) follows from the dominating property implied by Assumptions~\ref{ass:approximation}~\ref{ass:approximation:rhs} and Assumption~\ref{ass:approximation:UE}, 
and (f) follows from $\embed^+$ being a left inverse of $\embed$.

For the second inequality, let $\xvec^\star$ be an optimal solution to $\baseproblem(\PU, \U)$ and consider $\xvec_i = (\beta(\vvec_0) + \rho)\xvec^\star\left(\embed^+\left(\frac{1}{\beta(\vvec_0) + \rho} \vvec_i\right)\right)$ for each $i\in[\numue]_0$. 
First, we find
$$
\frac{1}{\beta(\vvec_0)+\rho}\vvec_i 
\overset{\textnormal{(a)}}{=} \frac{1}{\beta(\vvec_0)+\rho}(\vvec_0 + \rho\unitvec_i) 
\overset{\textnormal{(b)}}{=} \frac{\beta(\vvec_0)}{\beta(\vvec_0)+\rho} \frac{1}{\beta(\vvec_0)}\vvec_0 +  \frac{\rho}{\beta(\vvec_0)+\rho} \unitvec_i 
\overset{\textnormal{(c)}}{\in} \UE_g
$$
for all $i\in[\numue]$ and $g\in[G]$.
Here, (a) and (b) are immediate, and (c) follows from $\unitvec_i\in \UE_g$ by Assumption~\ref{ass:approximation:UE}, $\frac{1}{\beta(\vvec_0)}\vvec_0 \in \UE_g$ by definition of $\beta$, and convexity of $\UE_g$. 
Similarly, $\frac{1}{\beta(\vvec_0)+\rho}\vvec_0\leq \frac{1}{\beta(\vvec_0)}\vvec_0\in \UE_g$ implies $\frac{1}{\beta(\vvec_0)+\rho}\vvec_0\in\UE_g$ by down-monotonicity from Assumption~\ref{ass:approximation:UE}. 
We now show that $(\xvec_0,\dots,\xvec_\numue)$ is a feasible solution for \ref{eq:hatP}. 
For the first set of constraints, we find
\begin{align*}
\amat_g \xvec_i &\overset{\textnormal{(a)}}{=}  (\beta(\vvec_0) + \rho)\amat_g\xvec^\star\left(\embed^+\left(\frac{1}{\beta(\vvec_0) + \rho} \vvec_i\right)\right) 
\\&\overset{\textnormal{(b)}}{\leq} 
(\beta(\vvec_0) + \rho) \bvec_g\left(\embed^+\left(\frac{1}{\beta(\vvec_0) + \rho} \vvec_i\right)\right) \overset{\textnormal{(c)}}{\leq} \bvec_g(\embed^+\vvec_i)
\end{align*}
for all $i\in[\numue]$ and $g\in[G]$. 
Here, (a) follows by definition of $\xvec_i$, 
(b) follows by $\embed^+\left(\frac{1}{\beta(\vvec_0) + \rho} \vvec_i\right)\in \U_g$ and feasibility of $\xvec^\star$, and (c) follows from $\bvec_g$ having non-positive coefficients, $\beta(\vvec_0) + \rho\geq 1$, and linearity of $\embed^+$. 
The non-anticipativity condition of \ref{eq:hatP} holds by non-anticipativity of $\xvec^\star$ and $\embed$ being information preserving. 
Finally, the objective value satisfies $\cvec^\intercal\xvec_{00} = (\beta(\vvec_0) + \rho)\cvec^\intercal \xvec^\star_0 = (\beta(\vvec_0) + \rho)\baseproblem(\PU,\U)$ by Assumption~\ref{ass:approximation}~\ref{ass:approximation:obj}. 
\end{proof}

We now use these policies and their approximation guarantees to prove Proposition~\ref{pro:bounds}. 
\begin{proof}[Proof of Proposition~\ref{pro:bounds}]
In the following, we explicitly prove the claim for $\aff \lproblem(\PU, \UFOV{}^{\mathrm{A}})$. 
Specifically, we show that $\aff\lproblem(\PUF, \UFOV{}^\mathrm{A}) \leq (\beta(\zvec) + \max_{g\in[G]}\maxdist_g([\nullvec, \zvec])) \cdot \baseproblem(\PU, \U)$ for each $\zvec\in Z^\mathrm{A}$. 
For the remainder of the proof, we w.l.o.g.~fix one $\zvec\in Z^\mathrm{A}$. 
Consider the polytope $\hat{\UE} = \convex(\bm{v}_0, \bm{v}_1, \dots, \bm{v}_\numue)$ as described above with $\bm{v}_0 = \zvec$ and $\rho = \max_{g\in[G]}\maxdist_g([\nullvec, \zvec])$. 
Then, $\hat{\UE}$ is a valid dominating polytope for $\UE$ since
\begin{align*}
\max_{\uevec \in \UE} \sum_{i\in[\numue]} (\uescal_i - z_i)_+ \overset{\textnormal{(a)}}{=} &
\max_{\uevec \in \UE} \sum_{i\in[\numue]} \min_{\uescal'_i \in [0, z_i]}\lvert\uescal_i - \uescal'_i\rvert 
= 
\max_{\uevec \in \UE} \min_{\uevec' \in [\nullvec, \zvec]} \norm{\uevec - \uevec'}_1 \\ \overset{\textnormal{(b)}}{=}&
\max_{\uevec \in \UE} \dist(\uevec, [\nullvec, \zvec]) 
\overset{\textnormal{(c)}}{=} 
\max_{g\in[G]} \max_{\uevec \in \UE_g} \dist(\uevec, [\nullvec, \zvec])
\overset{\textnormal{(d)}}{=} 
\max_{g\in[G]} \maxdist_g([\nullvec, \zvec])
\overset{\textnormal{(e)}}{=} \rho,
\end{align*}
where (a) follows from $\UE \geq \nullvec$ by Assumption~\ref{ass:approximation:UE},
(b) follows by definition of $\dist$,
(c) follows by definition of $\UE = \bigcup_{g\in[G]}\UE_g$,
(d) follows by definition of $\maxdist_g$, and
(e) follows by definition of $\rho$. 
Thus, \ref{eq:hatP} is a $\beta(\zvec) + \max_{g\in[G]}\maxdist_g([\nullvec, \zvec])$ approximation of $\baseproblem(\PU, \U)$ by Lemma~\ref{lem:dominatingpolicybound}. 
Let $(\xvec_0,\dots,\xvec_\numue)$ be an optimal solution to \ref{eq:hatP} and consider the affine policy $\xfvec(\ufvec) = \xvec_0 + \frac{1}{\rho}\sum_{i\in[\numue], j\in \numuef^+_i} \ufscal_{ij}(\xvec_i - \xvec_0)$ on $\UFOV{}^\mathrm{A}$. 
The nonanticipativity constraints of \ref{eq:hatP} imply that the here-and-now decisions of $\xfvec(\ufvec)$ are given by $\xfvec(\ufvec)_0 = \xvec_{00}$. 
Thus, $\mathbb{E}_{\PUF} [\cvec(\retr(\sufvec))^\intercal \xfvec(\sufvec)] = \cvec^\intercal\xvec_{00} = \hat{\baseproblem}(\hat{\UE})$ by part \ref{ass:approximation:obj} of Assumption~\ref{ass:approximation}. 
Finally, $\xfvec(\ufvec)$ is feasible for $\aff \lproblem(\PUF, \UFOV{}^\mathrm{A})$ as for any $g\in[G]$ and $\ufvec\in\UF_g$, we have
\begin{align*}
\amat_g \xfvec(\ufvec) 
& \overset{\textnormal{(a)}}{=} 
\amat_g \left(\left(1-\frac{1}{\rho}\sum_{i\in[\numue],j\in \numuef^+_i} \ufscal_{ij}\right) \xvec_0 + \frac{1}{\rho}\sum_{i\in[\numue],j\in \numuef^+_i} \ufscal_{ij}\xvec_i  \right)
\\
& \overset{\textnormal{(b)}}{\leq}
\left(\left(1-\frac{1}{\rho}\sum_{i\in[\numue],j\in \numuef^+_i} \ufscal_{ij}\right) \bvec_g(\embed^+\bm{v}_0) + \frac{1}{\rho}\sum_{i\in[\numue],j\in \numuef^+_i} \ufscal_{ij}\bvec_g(\embed^+\bm{v}_i)  \right)
\\
& \overset{\textnormal{(c)}}{=}
\bvec_g\left(\embed^+\left(\bm{v}_0 + \frac{1}{\rho}\sum_{i\in[\numue], j\in \numuef^+_i} \ufscal_{ij}(\bm{v}_i - \bm{v}_0)  \right)\right)
\\
& \overset{\textnormal{(d)}}{\leq}
\bvec_g(\embed^+(\fold^+(\ufvec))) \overset{\textnormal{(e)}}{=} \bvec_g(\retr(\ufvec)).
\end{align*}
Here, (a) follows from the definition of $\xfvec$, and (b) follows from $\amat_g \xvec_i \leq \bvec_g(\embed^+ \bm{v}_i)$ for all $i\in [\numue]_0$ by feasibility of $(\xvec_0, \dots, \xvec_\numue)$ in \ref{eq:hatP}, linearity, and the factors pre-multiplying $\bvec_g$ being non-negative. 
In particular, the cuts~\eqref{eq:counterpart_dual_cuts} for $\zvec$ imply that $1-\frac{1}{\rho}\sum_{i\in[\numue],j\in \numuef^+_i} \ufscal_{ij} \geq 0$ by $\sum_{i\in[\numue],j\in \numuef^+_i} \ufscal_{ij} = \fdist(\ufvec, [\nullvec, \zvec]) \leq \maxdist_g([\nullvec, \zvec]) \leq \rho$. 
Finally, (c) follows from the linearity of $\bvec_g$ and $\embed^+$, (d) follows from $\bm{v}_0 + \frac{1}{\rho}\sum_{i\in[\numue], j\in \numuef^+_i} \ufscal_{ij}(\bm{v}_i - \bm{v}_0) = \zvec + \sum_{i\in[\numue], j\in \numuef^+_i} \ufscal_{ij}\unitvec_i \geq \fold^+(\ufvec)$ and the dominating property implied by Assumption~\ref{ass:approximation}~\ref{ass:approximation:rhs} and Assumption~\ref{ass:approximation:UE}, and (e) follows from the definition of $\retr$. 
We thus observe that
$$
\aff\lproblem(\PUF, \UFOV{}^\mathrm{A}) \leq \hat{\baseproblem}(\hat{\UE}) \leq (\beta(\zvec) + \max_{g\in[G]}\maxdist_g([\nullvec, \zvec])) \cdot \baseproblem(\PU, \U)
$$
as desired.
The case for $\aff \lproblem(\PU, \UFO)$ follows analogously by setting $\UFOV{}^{\mathrm{A}} = \UFO$ and $Z^\mathrm{A} = Z$. 
\end{proof}

%% file: main.bbl
\begin{thebibliography}{}

\bibitem[Abbeel and Ng, 2004]{Abbeel2004}
Abbeel, P. and Ng, A.~Y. (2004).
\newblock Apprenticeship learning via inverse reinforcement learning.
\newblock In {\em Twenty-first International Conference on Machine Learning},
  pages 1--8, Banff, AB. ACM Press.

\bibitem[Bampou and Kuhn, 2011]{Bampou2011}
Bampou, D. and Kuhn, D. (2011).
\newblock Scenario-free stochastic programming with polynomial decision rules.
\newblock In {\em {IEEE} Conference on Decision and Control and European
  Control Conference}, pages 7806--7812, Orlando, FL. {IEEE}.

\bibitem[Baty et~al., 2024]{Baty2024}
Baty, L., Jungel, K., Klein, P.~S., Parmentier, A., and Schiffer, M. (2024).
\newblock Combinatorial optimization-enriched machine learning to solve the
  dynamic vehicle routing problem with time windows.
\newblock {\em Transportation Science}, 58(4):708--725.

\bibitem[Beck and Ben-Tal, 2009]{Beck2009}
Beck, A. and Ben-Tal, A. (2009).
\newblock Duality in robust optimization: Primal worst equals dual best.
\newblock {\em Operations Research Letters}, 37(1):1--6.

\bibitem[Ben-Tal et~al., 2004]{BenTal2004}
Ben-Tal, A., Goryashko, A., Guslitzer, E., and Nemirovski, A. (2004).
\newblock Adjustable robust solutions of uncertain linear programs.
\newblock {\em Mathematical Programming}, 99(2):351--376.

\bibitem[Ben-Tal et~al., 2020]{BenTal2020}
Ben-Tal, A., Housni, O.~E., and Goyal, V. (2020).
\newblock A tractable approach for designing piecewise affine policies in
  two-stage adjustable robust optimization.
\newblock {\em Mathematical Programming}, 182(1-2):57--102.

\bibitem[Bertsimas and Bidkhori, 2015]{Bertsimas2015}
Bertsimas, D. and Bidkhori, H. (2015).
\newblock On the performance of affine policies for two-stage adaptive
  optimization: a geometric perspective.
\newblock {\em Mathematical Programming}, 153(2):577--594.

\bibitem[Bertsimas and Georghiou, 2015]{Bertsimas2015b}
Bertsimas, D. and Georghiou, A. (2015).
\newblock Design of near optimal decision rules in multistage adaptive
  mixed-integer optimization.
\newblock {\em Operations Research}, 63(3):610--627.

\bibitem[Bertsimas and Goyal, 2012]{Bertsimas2012}
Bertsimas, D. and Goyal, V. (2012).
\newblock On the power and limitations of affine policies in two-stage adaptive
  optimization.
\newblock {\em Mathematical Programming}, 134(2):491--531.

\bibitem[Bertsimas et~al., 2011a]{Bertsimas2011}
Bertsimas, D., Goyal, V., and Sun, X.~A. (2011a).
\newblock A geometric characterization of the power of finite adaptability in
  multistage stochastic and adaptive optimization.
\newblock {\em Mathematics of Operations Research}, 36(1):24--54.

\bibitem[Bertsimas et~al., 2011b]{Bertsimas2011b}
Bertsimas, D., Iancu, D.~A., and Parrilo, P.~A. (2011b).
\newblock A hierarchy of near-optimal policies for multistage adaptive
  optimization.
\newblock {\em {IEEE} Transactions on Automatic Control}, 56(12):2809--2824.

\bibitem[Bertsimas and Kallus, 2020]{Bertsimas2020}
Bertsimas, D. and Kallus, N. (2020).
\newblock From predictive to prescriptive analytics.
\newblock {\em Management Science}, 66(3):1025--1044.

\bibitem[Bertsimas et~al., 2023]{Bertsimas2023}
Bertsimas, D., Shtern, S., and Sturt, B. (2023).
\newblock A data-driven approach to multistage stochastic linear optimization.
\newblock {\em Management Science}, 69(1):51--74.

\bibitem[Blanchet et~al., 2019]{Blanchet2019}
Blanchet, J., Kang, Y., and Murthy, K. (2019).
\newblock Robust {W}asserstein profile inference and applications to machine
  learning.
\newblock {\em Journal of Applied Probability}, 56(3):830--857.

\bibitem[Chen et~al., 2008]{Chen2008}
Chen, X., Sim, M., Sun, P., and Zhang, J. (2008).
\newblock A linear decision-based approximation approach to stochastic
  programming.
\newblock {\em Operations Research}, 56(2):344--357.

\bibitem[Chen and Zhang, 2009]{Chen2009}
Chen, X. and Zhang, Y. (2009).
\newblock Uncertain linear programs: Extended affinely adjustable robust
  counterparts.
\newblock {\em Operations Research}, 57(6):1469--1482.

\bibitem[Feige et~al., 2007]{Feigea2007}
Feige, U., Jain, K., Mahdian, M., and Mirrokni, V. (2007).
\newblock Robust combinatorial optimization with exponential scenarios.
\newblock In {\em Integer Programming and Combinatorial Optimization}, pages
  439--453. Springer.

\bibitem[François-Lavet et~al., 2018]{FrancoisLavet2018}
François-Lavet, V., Henderson, P., Islam, R., Bellemare, M.~G., and Pineau, J.
  (2018).
\newblock An introduction to deep reinforcement learning.
\newblock {\em Foundations and Trends® in Machine Learning},
  11(3–4):219--354.

\bibitem[Gao and Kleywegt, 2023]{Gao2023}
Gao, R. and Kleywegt, A. (2023).
\newblock Distributionally robust stochastic optimization with {W}asserstein
  distance.
\newblock {\em Mathematics of Operations Research}, 48(2):603--655.

\bibitem[Georghiou et~al., 2015]{Georghiou2015}
Georghiou, A., Wiesemann, W., and Kuhn, D. (2015).
\newblock Generalized decision rule approximations for stochastic programming
  via liftings.
\newblock {\em Mathematical Programming}, 152(1-2):301--338.

\bibitem[Goh and Sim, 2010]{GS10:dro_tractable_approximations}
Goh, J. and Sim, M. (2010).
\newblock Distributionally robust optimization and its tractable
  approximations.
\newblock {\em Operations Research}, 58(4):902--917.

\bibitem[Graves, 1999]{Graves1999}
Graves, S.~C. (1999).
\newblock A single-item inventory model for a nonstationary demand process.
\newblock {\em Manufacturing \& Service Operations Management}, 1(1):50--61.

\bibitem[Guslitser, 2002]{G02:immunized_solutions}
Guslitser, E. (2002).
\newblock Uncertainty-immunized solutions in linear programming.
\newblock Master's thesis, Technion.

\bibitem[Han and Nohadani, 2025]{Han2025}
Han, E. and Nohadani, O. (2025).
\newblock Nonlinear decision rules made scalable by nonparametric liftings.
\newblock {\em Management Science}, 71(4):3449--3471.

\bibitem[Karp, 1972]{Karp1972}
Karp, R.~M. (1972).
\newblock Reducibility among combinatorial problems.
\newblock In Miller, R.~E., Thatcher, J.~W., and Bohlinger, J.~D., editors,
  {\em Complexity of Computer Computations}, pages 85--103. Springer US.

\bibitem[Kuhn et~al., 2011]{Kuhn2011}
Kuhn, D., Wiesemann, W., and Georghiou, A. (2011).
\newblock Primal and dual linear decision rules in stochastic and robust
  optimization.
\newblock {\em Mathematical Programming}, 130(1):177--209.

\bibitem[Laguna, 1998]{Laguna1998}
Laguna, M. (1998).
\newblock Applying robust optimization to capacity expansion of one location in
  telecommunications with demand uncertainty.
\newblock {\em Management Science}, 44(11-part-2):101--110.

\bibitem[Mohajerin~Esfahani and Kuhn, 2018]{Esfahani2018}
Mohajerin~Esfahani, P. and Kuhn, D. (2018).
\newblock Data-driven distributionally robust optimization using the
  {W}asserstein metric: performance guarantees and tractable reformulations.
\newblock {\em Mathematical Programming}, 171(1-2):115--166.

\bibitem[Muth, 1960]{Muth1960}
Muth, J.~F. (1960).
\newblock Optimal properties of exponentially weighted forecasts.
\newblock {\em Journal of the American Statistical Association},
  55(290):299--306.

\bibitem[Powell, 2011]{Powell2011}
Powell, W.~B. (2011).
\newblock {\em Approximate Dynamic Programming: Solving the Curses of
  Dimensionality}.
\newblock Wiley.

\bibitem[Puterman, 1994]{Puterman1994}
Puterman, M.~L. (1994).
\newblock {\em Markov Decision Processes: Discrete Stochastic Dynamic
  Programming}.
\newblock Wiley.

\bibitem[Rahal et~al., 2021]{Rahal2021}
Rahal, S., Papageorgiou, D.~J., and Li, Z. (2021).
\newblock Hybrid strategies using linear and piecewise-linear decision rules
  for multistage adaptive linear optimization.
\newblock {\em European Journal of Operational Research}, 290(3):1014--1030.

\bibitem[See and Sim, 2010]{See2010}
See, C.-T. and Sim, M. (2010).
\newblock Robust approximation to multiperiod inventory management.
\newblock {\em Operations Research}, 58(3):583--594.

\bibitem[Thomä et~al., 2024]{Thomae2024}
Thomä, S., Walther, G., and Schiffer, M. (2024).
\newblock Designing tractable piecewise affine policies for multi-stage
  adjustable robust optimization.
\newblock {\em Mathematical Programming}, 208(1–2):661--716.

\bibitem[Vayanos et~al., 2011]{Vayanos2011}
Vayanos, P., Kuhn, D., and Rustem, B. (2011).
\newblock Decision rules for information discovery in multi-stage stochastic
  programming.
\newblock In {\em {IEEE} Conference on Decision and Control and European
  Control Conference}, pages 7368--7373, Orlando, FL. {IEEE}.

\bibitem[Wang and Qi, 2020]{Wang2020}
Wang, Z. and Qi, M. (2020).
\newblock Robust service network design under demand uncertainty.
\newblock {\em Transportation Science}, 54(3):676--689.

\bibitem[Zhen et~al., 2025]{Zhen2025}
Zhen, J., Kuhn, D., and Wiesemann, W. (2025).
\newblock A unified theory of robust and distributionally robust optimization
  via the primal-worst-equals-dual-best principle.
\newblock {\em Operations Research}, 73(2):862--878.

\end{thebibliography}
